\documentclass[a4paper,11pt,english]{smfart}
\usepackage[applemac]{inputenc}
\usepackage[francais,english]{babel}
\usepackage{amsfonts}
\usepackage{amsmath} 
\usepackage{hyperref} 
\usepackage{latexsym}
\usepackage{array}
\usepackage{amssymb}
\usepackage{enumerate}
\usepackage{smfthm}
\usepackage{graphicx}
\theoremstyle{plain}

\renewcommand{\(}{  \big(   }
\renewcommand{\)}{  \big)   }
\newcommand{\R}{  \mathbb{R}   }

\newcommand{\eps}{\varepsilon}
\newcommand{\e}{  \text{e}   }
\newcommand{\wt}{  \widetilde   }

\newcommand{\Z}{  \mathbb{Z}   }
\newcommand{\N}{  \mathbb{N}   }

\newcommand{\Pc}{  \mathcal{P}   }

\newcommand{\A}{  \mathcal{A}   }
\newcommand{\Q}{  \mathbb{Q}   }
\newcommand{\T}{  \S^{1} }
\newcommand{\Tc}{  \mathbb T}

\newcommand{\Mc}{  \mathcal{M}   }
\newcommand{\Rc}{  \mathcal{R}_{0}   }
\newcommand{\Ac}{  \mathcal{A}   }
\newcommand{\dis}{\displaystyle}

\newcommand{\Om}{  \Omega   }
\newcommand{\om}{  \omega   }
\newcommand{\ov}{  \overline  }

\renewcommand{\phi}{  \varphi  }
\renewcommand{\L}{  \mathcal{L}   }
\newcommand{\wh}{  \widehat   }

\newcommand{\cp}[1]{  \big\{\,{ #1 }\, \big\}  }

\renewcommand{\S}{  \mathbb{S}  }
\numberwithin{equation}{section}

 \author{ Beno\^it Gr\'ebert}
\address{Laboratoire de Math\'ematiques J. Leray, Universit\'e de Nantes, UMR CNRS 6629\\
2, rue de la Houssini\`ere \\
44322 Nantes Cedex 03, France.}
\email{benoit.grebert@univ-nantes.fr}

\author{ Laurent Thomann }
\address{Laboratoire de Math\'ematiques J. Leray, Universit\'e de Nantes, UMR CNRS 6629\\
2, rue de la Houssini\`ere \\
44322 Nantes Cedex 03, France.}
\email{laurent.thomann@univ-nantes.fr}

\title[Resonant dynamics for NLS]
{Resonant dynamics  for the quintic non linear Schr\"odinger equation}

\begin{document}
%
%
%
%
\ifx\figforTeXisloaded\relax\endinput\else\global\let\figforTeXisloaded=\relax\fi
\message{version 1.8.4}
\catcode`\@=11
\ifx\ctr@ln@m\undefined\else%
    \immediate\write16{*** Fig4TeX WARNING : \string\ctr@ln@m\space already defined.}\fi
\def\ctr@ln@m#1{\ifx#1\undefined\else%
    \immediate\write16{*** Fig4TeX WARNING : \string#1 already defined.}\fi}
\ctr@ln@m\ctr@ld@f
\def\ctr@ld@f#1#2{\ctr@ln@m#2#1#2}
\ctr@ld@f\def\ctr@ln@w#1#2{\ctr@ln@m#2\csname#1\endcsname#2}
{\catcode`\/=0 \catcode`/\=12 /ctr@ld@f/gdef/BS@{\}}
\ctr@ld@f\def\ctr@lcsn@m#1{\expandafter\ifx\csname#1\endcsname\relax\else%
    \immediate\write16{*** Fig4TeX WARNING : \BS@\expandafter\string#1\space already defined.}\fi}
\ctr@ld@f\edef\colonc@tcode{\the\catcode`\:}
\ctr@ld@f\edef\semicolonc@tcode{\the\catcode`\;}
\ctr@ld@f\def\t@stc@tcodech@nge{{\let\c@tcodech@nged=\z@%
    \ifnum\colonc@tcode=\the\catcode`\:\else\let\c@tcodech@nged=\@ne\fi%
    \ifnum\semicolonc@tcode=\the\catcode`\;\else\let\c@tcodech@nged=\@ne\fi%
    \ifx\c@tcodech@nged\@ne%
    \immediate\write16{}
    \immediate\write16{!!!=============================================================!!!}
    \immediate\write16{ Fig4TeX WARNING :}
    \immediate\write16{ The category code of some characters has been changed, which will}
    \immediate\write16{ result in an error (message "Runaway argument?").}
    \immediate\write16{ This probably comes from another package that changed the category}
    \immediate\write16{ code after Fig4TeX was loaded. If that proves to be exact, the}
    \immediate\write16{ solution is to exchange the loading commands on top of your file}
    \immediate\write16{ so that Fig4TeX is loaded last. For example, in LaTeX, we should}
    \immediate\write16{ say :}
    \immediate\write16{\BS@ usepackage[french]{babel}}
    \immediate\write16{\BS@ usepackage{fig4tex}}
    \immediate\write16{!!!=============================================================!!!}
    \immediate\write16{}
    \fi}}
\ctr@ld@f\def\FigforTeX{F\kern-.05em i\kern-.05em g\kern-.1em\raise-.14em\hbox{4}\kern-.19em\TeX}
\ctr@ln@w{newdimen}\epsil@n\epsil@n=0.00005pt
\ctr@ln@w{newdimen}\Cepsil@n\Cepsil@n=0.005pt
\ctr@ln@w{newdimen}\dcq@\dcq@=254pt
\ctr@ln@w{newdimen}\PI@\PI@=3.141592pt
\ctr@ln@w{newdimen}\DemiPI@deg\DemiPI@deg=90pt
\ctr@ln@w{newdimen}\PI@deg\PI@deg=180pt
\ctr@ln@w{newdimen}\DePI@deg\DePI@deg=360pt
\ctr@ld@f\chardef\t@n=10
\ctr@ld@f\chardef\c@nt=100
\ctr@ld@f\chardef\@lxxiv=74
\ctr@ld@f\chardef\@xci=91
\ctr@ld@f\mathchardef\@nMnCQn=9949
\ctr@ld@f\chardef\@vi=6
\ctr@ld@f\chardef\@xxx=30
\ctr@ld@f\chardef\@lvi=56
\ctr@ld@f\chardef\@@lxxi=71
\ctr@ld@f\chardef\@lxxxv=85
\ctr@ld@f\mathchardef\@@mmmmlxviii=4068
\ctr@ld@f\mathchardef\@ccclx=360
\ctr@ld@f\mathchardef\@dccxx=720
\ctr@ln@w{newcount}\p@rtent \ctr@ln@w{newcount}\f@ctech \ctr@ln@w{newcount}\result@tent
\ctr@ln@w{newdimen}\v@lmin \ctr@ln@w{newdimen}\v@lmax \ctr@ln@w{newdimen}\v@leur
\ctr@ln@w{newdimen}\result@t\ctr@ln@w{newdimen}\result@@t
\ctr@ln@w{newdimen}\mili@u \ctr@ln@w{newdimen}\c@rre \ctr@ln@w{newdimen}\delt@
\ctr@ld@f\def\degT@rd{0.017453 }  
\ctr@ld@f\def\rdT@deg{57.295779 } 
\ctr@ln@m\v@leurseule
{\catcode`p=12 \catcode`t=12 \gdef\v@leurseule#1pt{#1}}
\ctr@ld@f\def\repdecn@mb#1{\expandafter\v@leurseule\the#1\space}
\ctr@ld@f\def\arct@n#1(#2,#3){{\v@lmin=#2\v@lmax=#3%
    \maxim@m{\mili@u}{-\v@lmin}{\v@lmin}\maxim@m{\c@rre}{-\v@lmax}{\v@lmax}%
    \delt@=\mili@u\m@ech\mili@u%
    \ifdim\c@rre>\@nMnCQn\mili@u\divide\v@lmax\tw@\c@lATAN\v@leur(\z@,\v@lmax)
    \else%
    \maxim@m{\mili@u}{-\v@lmin}{\v@lmin}\maxim@m{\c@rre}{-\v@lmax}{\v@lmax}%
    \m@ech\c@rre%
    \ifdim\mili@u>\@nMnCQn\c@rre\divide\v@lmin\tw@
    \maxim@m{\mili@u}{-\v@lmin}{\v@lmin}\c@lATAN\v@leur(\mili@u,\z@)%
    \else\c@lATAN\v@leur(\delt@,\v@lmax)\fi\fi%
    \ifdim\v@lmin<\z@\v@leur=-\v@leur\ifdim\v@lmax<\z@\advance\v@leur-\PI@%
    \else\advance\v@leur\PI@\fi\fi%
    \global\result@t=\v@leur}#1=\result@t}
\ctr@ld@f\def\m@ech#1{\ifdim#1>1.646pt\divide\mili@u\t@n\divide\c@rre\t@n\m@ech#1\fi}
\ctr@ld@f\def\c@lATAN#1(#2,#3){{\v@lmin=#2\v@lmax=#3\v@leur=\z@\delt@=\tw@ pt%
    \un@iter{0.785398}{\v@lmax<}%
    \un@iter{0.463648}{\v@lmax<}%
    \un@iter{0.244979}{\v@lmax<}%
    \un@iter{0.124355}{\v@lmax<}%
    \un@iter{0.062419}{\v@lmax<}%
    \un@iter{0.031240}{\v@lmax<}%
    \un@iter{0.015624}{\v@lmax<}%
    \un@iter{0.007812}{\v@lmax<}%
    \un@iter{0.003906}{\v@lmax<}%
    \un@iter{0.001953}{\v@lmax<}%
    \un@iter{0.000976}{\v@lmax<}%
    \un@iter{0.000488}{\v@lmax<}%
    \un@iter{0.000244}{\v@lmax<}%
    \un@iter{0.000122}{\v@lmax<}%
    \un@iter{0.000061}{\v@lmax<}%
    \un@iter{0.000030}{\v@lmax<}%
    \un@iter{0.000015}{\v@lmax<}%
    \global\result@t=\v@leur}#1=\result@t}
\ctr@ld@f\def\un@iter#1#2{%
    \divide\delt@\tw@\edef\dpmn@{\repdecn@mb{\delt@}}%
    \mili@u=\v@lmin%
    \ifdim#2\z@%
      \advance\v@lmin-\dpmn@\v@lmax\advance\v@lmax\dpmn@\mili@u%
      \advance\v@leur-#1pt%
    \else%
      \advance\v@lmin\dpmn@\v@lmax\advance\v@lmax-\dpmn@\mili@u%
      \advance\v@leur#1pt%
    \fi}
\ctr@ld@f\def\c@ssin#1#2#3{\expandafter\ifx\csname COS@\number#3\endcsname\relax\c@lCS{#3pt}%
    \expandafter\xdef\csname COS@\number#3\endcsname{\repdecn@mb\result@t}%
    \expandafter\xdef\csname SIN@\number#3\endcsname{\repdecn@mb\result@@t}\fi%
    \edef#1{\csname COS@\number#3\endcsname}\edef#2{\csname SIN@\number#3\endcsname}}
\ctr@ld@f\def\c@lCS#1{{\mili@u=#1\p@rtent=\@ne%
    \relax\ifdim\mili@u<\z@\red@ng<-\else\red@ng>+\fi\f@ctech=\p@rtent%
    \relax\ifdim\mili@u<\z@\mili@u=-\mili@u\f@ctech=-\f@ctech\fi\c@@lCS}}
\ctr@ld@f\def\c@@lCS{\v@lmin=\mili@u\c@rre=-\mili@u\advance\c@rre\DemiPI@deg\v@lmax=\c@rre%
    \mili@u\@@lxxi\mili@u\divide\mili@u\@@mmmmlxviii%
    \edef\v@larg{\repdecn@mb{\mili@u}}\mili@u=-\v@larg\mili@u%
    \edef\v@lmxde{\repdecn@mb{\mili@u}}%
    \c@rre\@@lxxi\c@rre\divide\c@rre\@@mmmmlxviii%
    \edef\v@largC{\repdecn@mb{\c@rre}}\c@rre=-\v@largC\c@rre%
    \edef\v@lmxdeC{\repdecn@mb{\c@rre}}%
    \fctc@s\mili@u\v@lmin\global\result@t\p@rtent\v@leur%
    \let\t@mp=\v@larg\let\v@larg=\v@largC\let\v@largC=\t@mp%
    \let\t@mp=\v@lmxde\let\v@lmxde=\v@lmxdeC\let\v@lmxdeC=\t@mp%
    \fctc@s\c@rre\v@lmax\global\result@@t\f@ctech\v@leur}
\ctr@ld@f\def\fctc@s#1#2{\v@leur=#1\relax\ifdim#2<\@lxxxv\p@\cosser@h\else\sinser@t\fi}
\ctr@ld@f\def\cosser@h{\advance\v@leur\@lvi\p@\divide\v@leur\@lvi%
    \v@leur=\v@lmxde\v@leur\advance\v@leur\@xxx\p@%
    \v@leur=\v@lmxde\v@leur\advance\v@leur\@ccclx\p@%
    \v@leur=\v@lmxde\v@leur\advance\v@leur\@dccxx\p@\divide\v@leur\@dccxx}
\ctr@ld@f\def\sinser@t{\v@leur=\v@lmxdeC\p@\advance\v@leur\@vi\p@%
    \v@leur=\v@largC\v@leur\divide\v@leur\@vi}
\ctr@ld@f\def\red@ng#1#2{\relax\ifdim\mili@u#1#2\DemiPI@deg\advance\mili@u#2-\PI@deg%
    \p@rtent=-\p@rtent\red@ng#1#2\fi}
\ctr@ld@f\def\pr@c@lCS#1#2#3{\ctr@lcsn@m{COS@\number#3 }%
    \expandafter\xdef\csname COS@\number#3\endcsname{#1}%
    \expandafter\xdef\csname SIN@\number#3\endcsname{#2}}
\pr@c@lCS{1}{0}{0}
\pr@c@lCS{0.7071}{0.7071}{45}\pr@c@lCS{0.7071}{-0.7071}{-45}
\pr@c@lCS{0}{1}{90}          \pr@c@lCS{0}{-1}{-90}
\pr@c@lCS{-1}{0}{180}        \pr@c@lCS{-1}{0}{-180}
\pr@c@lCS{0}{-1}{270}        \pr@c@lCS{0}{1}{-270}
\ctr@ld@f\def\invers@#1#2{{\v@leur=#2\maxim@m{\v@lmax}{-\v@leur}{\v@leur}%
    \f@ctech=\@ne\m@inv@rs%
    \multiply\v@leur\f@ctech\edef\v@lv@leur{\repdecn@mb{\v@leur}}%
    \p@rtentiere{\p@rtent}{\v@leur}\v@lmin=\p@\divide\v@lmin\p@rtent%
    \inv@rs@\multiply\v@lmax\f@ctech\global\result@t=\v@lmax}#1=\result@t}
\ctr@ld@f\def\m@inv@rs{\ifdim\v@lmax<\p@\multiply\v@lmax\t@n\multiply\f@ctech\t@n\m@inv@rs\fi}
\ctr@ld@f\def\inv@rs@{\v@lmax=-\v@lmin\v@lmax=\v@lv@leur\v@lmax%
    \advance\v@lmax\tw@ pt\v@lmax=\repdecn@mb{\v@lmin}\v@lmax%
    \delt@=\v@lmax\advance\delt@-\v@lmin\ifdim\delt@<\z@\delt@=-\delt@\fi%
    \ifdim\delt@>\epsil@n\v@lmin=\v@lmax\inv@rs@\fi}
\ctr@ld@f\def\minim@m#1#2#3{\relax\ifdim#2<#3#1=#2\else#1=#3\fi}
\ctr@ld@f\def\maxim@m#1#2#3{\relax\ifdim#2>#3#1=#2\else#1=#3\fi}
\ctr@ld@f\def\p@rtentiere#1#2{#1=#2\divide#1by65536 }
\ctr@ld@f\def\r@undint#1#2{{\v@leur=#2\divide\v@leur\t@n\p@rtentiere{\p@rtent}{\v@leur}%
    \v@leur=\p@rtent pt\global\result@t=\t@n\v@leur}#1=\result@t}
\ctr@ld@f\def\sqrt@#1#2{{\v@leur=#2%
    \minim@m{\v@lmin}{\p@}{\v@leur}\maxim@m{\v@lmax}{\p@}{\v@leur}%
    \f@ctech=\@ne\m@sqrt@\sqrt@@%
    \mili@u=\v@lmin\advance\mili@u\v@lmax\divide\mili@u\tw@\multiply\mili@u\f@ctech%
    \global\result@t=\mili@u}#1=\result@t}
\ctr@ld@f\def\m@sqrt@{\ifdim\v@leur>\dcq@\divide\v@leur\c@nt\v@lmax=\v@leur%
    \multiply\f@ctech\t@n\m@sqrt@\fi}
\ctr@ld@f\def\sqrt@@{\mili@u=\v@lmin\advance\mili@u\v@lmax\divide\mili@u\tw@%
    \c@rre=\repdecn@mb{\mili@u}\mili@u%
    \ifdim\c@rre<\v@leur\v@lmin=\mili@u\else\v@lmax=\mili@u\fi%
    \delt@=\v@lmax\advance\delt@-\v@lmin\ifdim\delt@>\epsil@n\sqrt@@\fi}
\ctr@ld@f\def\extrairelepremi@r#1\de#2{\expandafter\lepremi@r#2@#1#2}
\ctr@ld@f\def\lepremi@r#1,#2@#3#4{\def#3{#1}\def#4{#2}\ignorespaces}
\ctr@ld@f\def\@cfor#1:=#2\do#3{%
  \edef\@fortemp{#2}%
  \ifx\@fortemp\empty\else\@cforloop#2,\@nil,\@nil\@@#1{#3}\fi}
\ctr@ln@m\@nextwhile
\ctr@ld@f\def\@cforloop#1,#2\@@#3#4{%
  \def#3{#1}%
  \ifx#3\Fig@nnil\let\@nextwhile=\Fig@fornoop\else#4\relax\let\@nextwhile=\@cforloop\fi%
  \@nextwhile#2\@@#3{#4}}

\ctr@ld@f\def\@ecfor#1:=#2\do#3{%
  \def\@@cfor{\@cfor#1:=}%
  \edef\@@@cfor{#2}%
  \expandafter\@@cfor\@@@cfor\do{#3}}
\ctr@ld@f\def\Fig@nnil{\@nil}
\ctr@ld@f\def\Fig@fornoop#1\@@#2#3{}
\ctr@ln@m\list@@rg
\ctr@ld@f\def\trtlis@rg#1#2{\def\list@@rg{#1}%
    \@ecfor\p@rv@l:=\list@@rg\do{\expandafter#2\p@rv@l|}}
\ctr@ln@w{newbox}\b@xvisu
\ctr@ln@w{newtoks}\let@xte
\ctr@ln@w{newif}\ifitis@K
\ctr@ln@w{newcount}\s@mme
\ctr@ln@w{newcount}\l@mbd@un \ctr@ln@w{newcount}\l@mbd@de
\ctr@ln@w{newcount}\superc@ntr@l\superc@ntr@l=\@ne        
\ctr@ln@w{newcount}\typec@ntr@l\typec@ntr@l=\superc@ntr@l 
\ctr@ln@w{newdimen}\v@lX  \ctr@ln@w{newdimen}\v@lY  \ctr@ln@w{newdimen}\v@lZ
\ctr@ln@w{newdimen}\v@lXa \ctr@ln@w{newdimen}\v@lYa \ctr@ln@w{newdimen}\v@lZa
\ctr@ln@w{newdimen}\unit@\unit@=\p@ 
\ctr@ld@f\def\unit@util{pt}
\ctr@ld@f\def\ptT@ptps{0.996264 }
\ctr@ld@f\def\ptpsT@pt{1.00375 }
\ctr@ld@f\def\ptT@unit@{1} 
\ctr@ld@f\def\setunit@#1{\def\unit@util{#1}\setunit@@#1:\invers@{\result@t}{\unit@}%
    \edef\ptT@unit@{\repdecn@mb\result@t}}
\ctr@ld@f\def\setunit@@#1#2:{\ifcat#1a\unit@=\@ne#1#2\else\unit@=#1#2\fi}
\ctr@ld@f\def\d@fm@cdim#1#2{{\v@leur=#2\v@leur=\ptT@unit@\v@leur\xdef#1{\repdecn@mb\v@leur}}}
\ctr@ln@w{newif}\ifBdingB@x\BdingB@xtrue
\ctr@ln@w{newdimen}\c@@rdXmin \ctr@ln@w{newdimen}\c@@rdYmin  
\ctr@ln@w{newdimen}\c@@rdXmax \ctr@ln@w{newdimen}\c@@rdYmax
\ctr@ld@f\def\b@undb@x#1#2{\ifBdingB@x%
    \relax\ifdim#1<\c@@rdXmin\global\c@@rdXmin=#1\fi%
    \relax\ifdim#2<\c@@rdYmin\global\c@@rdYmin=#2\fi%
    \relax\ifdim#1>\c@@rdXmax\global\c@@rdXmax=#1\fi%
    \relax\ifdim#2>\c@@rdYmax\global\c@@rdYmax=#2\fi\fi}
\ctr@ld@f\def\b@undb@xP#1{{\Figg@tXY{#1}\b@undb@x{\v@lX}{\v@lY}}}
\ctr@ld@f\def\ellBB@x#1;#2,#3(#4,#5,#6){{\s@uvc@ntr@l\et@tellBB@x%
    \setc@ntr@l{2}\figptell-2::#1;#2,#3(#4,#6)\b@undb@xP{-2}%
    \figptell-2::#1;#2,#3(#5,#6)\b@undb@xP{-2}%
    \c@ssin{\C@}{\S@}{#6}\v@lmin=\C@ pt\v@lmax=\S@ pt%
    \mili@u=#3\v@lmin\delt@=#2\v@lmax\arct@n\v@leur(\delt@,\mili@u)%
    \mili@u=-#3\v@lmax\delt@=#2\v@lmin\arct@n\c@rre(\delt@,\mili@u)%
    \v@leur=\rdT@deg\v@leur\advance\v@leur-\DePI@deg%
    \c@rre=\rdT@deg\c@rre\advance\c@rre-\DePI@deg%
    \v@lmin=#4pt\v@lmax=#5pt%
    \loop\ifdim\v@leur<\v@lmax\ifdim\v@leur>\v@lmin%
    \edef\@ngle{\repdecn@mb\v@leur}\figptell-2::#1;#2,#3(\@ngle,#6)%
    \b@undb@xP{-2}\fi\advance\v@leur\PI@deg\repeat%
    \loop\ifdim\c@rre<\v@lmax\ifdim\c@rre>\v@lmin%
    \edef\@ngle{\repdecn@mb\c@rre}\figptell-2::#1;#2,#3(\@ngle,#6)%
    \b@undb@xP{-2}\fi\advance\c@rre\PI@deg\repeat%
    \resetc@ntr@l\et@tellBB@x}\ignorespaces}
\ctr@ld@f\def\initb@undb@x{\c@@rdXmin=\maxdimen\c@@rdYmin=\maxdimen%
    \c@@rdXmax=-\maxdimen\c@@rdYmax=-\maxdimen}
\ctr@ld@f\def\c@ntr@lnum#1{%
    \relax\ifnum\typec@ntr@l=\@ne%
    \ifnum#1<\z@%
    \immediate\write16{*** Forbidden point number (#1). Abort.}\end\fi\fi%
    \set@bjc@de{#1}}
\ctr@ln@m\objc@de
\ctr@ld@f\def\set@bjc@de#1{\edef\objc@de{@BJ\ifnum#1<\z@ M\romannumeral-#1\else\romannumeral#1\fi}}
\s@mme=\m@ne\loop\ifnum\s@mme>-19
  \set@bjc@de{\s@mme}\ctr@lcsn@m\objc@de\ctr@lcsn@m{\objc@de T}
\advance\s@mme\m@ne\repeat
\s@mme=\@ne\loop\ifnum\s@mme<6
  \set@bjc@de{\s@mme}\ctr@lcsn@m\objc@de\ctr@lcsn@m{\objc@de T}
\advance\s@mme\@ne\repeat
\ctr@ld@f\def\setc@ntr@l#1{\ifnum\superc@ntr@l>#1\typec@ntr@l=\superc@ntr@l%
    \else\typec@ntr@l=#1\fi}
\ctr@ld@f\def\resetc@ntr@l#1{\global\superc@ntr@l=#1\setc@ntr@l{#1}}
\ctr@ld@f\def\s@uvc@ntr@l#1{\edef#1{\the\superc@ntr@l}}
\ctr@ln@m\c@lproscal
\ctr@ld@f\def\c@lproscalDD#1[#2,#3]{{\Figg@tXY{#2}%
    \edef\Xu@{\repdecn@mb{\v@lX}}\edef\Yu@{\repdecn@mb{\v@lY}}\Figg@tXY{#3}%
    \global\result@t=\Xu@\v@lX\global\advance\result@t\Yu@\v@lY}#1=\result@t}
\ctr@ld@f\def\c@lproscalTD#1[#2,#3]{{\Figg@tXY{#2}\edef\Xu@{\repdecn@mb{\v@lX}}%
    \edef\Yu@{\repdecn@mb{\v@lY}}\edef\Zu@{\repdecn@mb{\v@lZ}}%
    \Figg@tXY{#3}\global\result@t=\Xu@\v@lX\global\advance\result@t\Yu@\v@lY%
    \global\advance\result@t\Zu@\v@lZ}#1=\result@t}
\ctr@ld@f\def\c@lprovec#1{%
    \det@rmC\v@lZa(\v@lX,\v@lY,\v@lmin,\v@lmax)%
    \det@rmC\v@lXa(\v@lY,\v@lZ,\v@lmax,\v@leur)%
    \det@rmC\v@lYa(\v@lZ,\v@lX,\v@leur,\v@lmin)%
    \Figv@ctCreg#1(\v@lXa,\v@lYa,\v@lZa)}
\ctr@ld@f\def\det@rm#1[#2,#3]{{\Figg@tXY{#2}\Figg@tXYa{#3}%
    \delt@=\repdecn@mb{\v@lX}\v@lYa\advance\delt@-\repdecn@mb{\v@lY}\v@lXa%
    \global\result@t=\delt@}#1=\result@t}
\ctr@ld@f\def\det@rmC#1(#2,#3,#4,#5){{\global\result@t=\repdecn@mb{#2}#5%
    \global\advance\result@t-\repdecn@mb{#3}#4}#1=\result@t}
\ctr@ld@f\def\getredf@ctDD#1(#2,#3){{\maxim@m{\v@lXa}{-#2}{#2}\maxim@m{\v@lYa}{-#3}{#3}%
    \maxim@m{\v@lXa}{\v@lXa}{\v@lYa}
    \ifdim\v@lXa>\@xci pt\divide\v@lXa\@xci%
    \p@rtentiere{\p@rtent}{\v@lXa}\advance\p@rtent\@ne\else\p@rtent=\@ne\fi%
    \global\result@tent=\p@rtent}#1=\result@tent\ignorespaces}
\ctr@ld@f\def\getredf@ctTD#1(#2,#3,#4){{\maxim@m{\v@lXa}{-#2}{#2}\maxim@m{\v@lYa}{-#3}{#3}%
    \maxim@m{\v@lZa}{-#4}{#4}\maxim@m{\v@lXa}{\v@lXa}{\v@lYa}%
    \maxim@m{\v@lXa}{\v@lXa}{\v@lZa}
    \ifdim\v@lXa>\@lxxiv pt\divide\v@lXa\@lxxiv%
    \p@rtentiere{\p@rtent}{\v@lXa}\advance\p@rtent\@ne\else\p@rtent=\@ne\fi%
    \global\result@tent=\p@rtent}#1=\result@tent\ignorespaces}
\ctr@ld@f\def\FigptintercircB@zDD#1:#2:#3,#4[#5,#6,#7,#8]{{\s@uvc@ntr@l\et@tfigptintercircB@zDD%
    \setc@ntr@l{2}\figvectPDD-1[#5,#8]\Figg@tXY{-1}\getredf@ctDD\f@ctech(\v@lX,\v@lY)%
    \mili@u=#4\unit@\divide\mili@u\f@ctech\c@rre=\repdecn@mb{\mili@u}\mili@u%
    \figptBezierDD-5::#3[#5,#6,#7,#8]%
    \v@lmin=#3\p@\v@lmax=\v@lmin\advance\v@lmax0.1\p@%
    \loop\edef\T@{\repdecn@mb{\v@lmax}}\figptBezierDD-2::\T@[#5,#6,#7,#8]%
    \figvectPDD-1[-5,-2]\n@rmeucCDD{\delt@}{-1}\ifdim\delt@<\c@rre\v@lmin=\v@lmax%
    \advance\v@lmax0.1\p@\repeat%
    \loop\mili@u=\v@lmin\advance\mili@u\v@lmax%
    \divide\mili@u\tw@\edef\T@{\repdecn@mb{\mili@u}}\figptBezierDD-2::\T@[#5,#6,#7,#8]%
    \figvectPDD-1[-5,-2]\n@rmeucCDD{\delt@}{-1}\ifdim\delt@>\c@rre\v@lmax=\mili@u%
    \else\v@lmin=\mili@u\fi\v@leur=\v@lmax\advance\v@leur-\v@lmin%
    \ifdim\v@leur>\epsil@n\repeat\figptcopyDD#1:#2/-2/%
    \resetc@ntr@l\et@tfigptintercircB@zDD}\ignorespaces}
\ctr@ln@m\figptinterlines
\ctr@ld@f\def\inters@cDD#1:#2[#3,#4;#5,#6]{{\s@uvc@ntr@l\et@tinters@cDD%
    \setc@ntr@l{2}\vecunit@{-1}{#4}\vecunit@{-2}{#6}%
    \Figg@tXY{-1}\setc@ntr@l{1}\Figg@tXYa{#3}%
    \edef\A@{\repdecn@mb{\v@lX}}\edef\B@{\repdecn@mb{\v@lY}}%
    \v@lmin=\B@\v@lXa\advance\v@lmin-\A@\v@lYa%
    \Figg@tXYa{#5}\setc@ntr@l{2}\Figg@tXY{-2}%
    \edef\C@{\repdecn@mb{\v@lX}}\edef\D@{\repdecn@mb{\v@lY}}%
    \v@lmax=\D@\v@lXa\advance\v@lmax-\C@\v@lYa%
    \delt@=\A@\v@lY\advance\delt@-\B@\v@lX%
    \invers@{\v@leur}{\delt@}\edef\v@ldelta{\repdecn@mb{\v@leur}}%
    \v@lXa=\A@\v@lmax\advance\v@lXa-\C@\v@lmin%
    \v@lYa=\B@\v@lmax\advance\v@lYa-\D@\v@lmin%
    \v@lXa=\v@ldelta\v@lXa\v@lYa=\v@ldelta\v@lYa%
    \setc@ntr@l{1}\Figp@intregDD#1:{#2}(\v@lXa,\v@lYa)%
    \resetc@ntr@l\et@tinters@cDD}\ignorespaces}
\ctr@ld@f\def\inters@cTD#1:#2[#3,#4;#5,#6]{{\s@uvc@ntr@l\et@tinters@cTD%
    \setc@ntr@l{2}\figvectNVTD-1[#4,#6]\figvectNVTD-2[#6,-1]\figvectPTD-1[#3,#5]%
    \r@pPSTD\v@leur[-2,-1,#4]\edef\v@lcoef{\repdecn@mb{\v@leur}}%
    \figpttraTD#1:{#2}=#3/\v@lcoef,#4/\resetc@ntr@l\et@tinters@cTD}\ignorespaces}
\ctr@ld@f\def\r@pPSTD#1[#2,#3,#4]{{\Figg@tXY{#2}\edef\Xu@{\repdecn@mb{\v@lX}}%
    \edef\Yu@{\repdecn@mb{\v@lY}}\edef\Zu@{\repdecn@mb{\v@lZ}}%
    \Figg@tXY{#3}\v@lmin=\Xu@\v@lX\advance\v@lmin\Yu@\v@lY\advance\v@lmin\Zu@\v@lZ%
    \Figg@tXY{#4}\v@lmax=\Xu@\v@lX\advance\v@lmax\Yu@\v@lY\advance\v@lmax\Zu@\v@lZ%
    \invers@{\v@leur}{\v@lmax}\global\result@t=\repdecn@mb{\v@leur}\v@lmin}%
    #1=\result@t}
\ctr@ln@m\n@rminf
\ctr@ld@f\def\n@rminfDD#1#2{{\Figg@tXY{#2}\maxim@m{\v@lX}{\v@lX}{-\v@lX}%
    \maxim@m{\v@lY}{\v@lY}{-\v@lY}\maxim@m{\global\result@t}{\v@lX}{\v@lY}}%
    #1=\result@t}
\ctr@ld@f\def\n@rminfTD#1#2{{\Figg@tXY{#2}\maxim@m{\v@lX}{\v@lX}{-\v@lX}%
    \maxim@m{\v@lY}{\v@lY}{-\v@lY}\maxim@m{\v@lZ}{\v@lZ}{-\v@lZ}%
    \maxim@m{\v@lX}{\v@lX}{\v@lY}\maxim@m{\global\result@t}{\v@lX}{\v@lZ}}%
    #1=\result@t}
\ctr@ld@f\def\n@rmeucCDD#1#2{\Figg@tXY{#2}\divide\v@lX\f@ctech\divide\v@lY\f@ctech%
    #1=\repdecn@mb{\v@lX}\v@lX\v@lX=\repdecn@mb{\v@lY}\v@lY\advance#1\v@lX}
\ctr@ld@f\def\n@rmeucCTD#1#2{\Figg@tXY{#2}%
    \divide\v@lX\f@ctech\divide\v@lY\f@ctech\divide\v@lZ\f@ctech%
    #1=\repdecn@mb{\v@lX}\v@lX\v@lX=\repdecn@mb{\v@lY}\v@lY\advance#1\v@lX%
    \v@lX=\repdecn@mb{\v@lZ}\v@lZ\advance#1\v@lX}
\ctr@ln@m\n@rmeucSV
\ctr@ld@f\def\n@rmeucSVDD#1#2{{\Figg@tXY{#2}%
    \v@lXa=\repdecn@mb{\v@lX}\v@lX\v@lYa=\repdecn@mb{\v@lY}\v@lY%
    \advance\v@lXa\v@lYa\sqrt@{\global\result@t}{\v@lXa}}#1=\result@t}
\ctr@ld@f\def\n@rmeucSVTD#1#2{{\Figg@tXY{#2}\v@lXa=\repdecn@mb{\v@lX}\v@lX%
    \v@lYa=\repdecn@mb{\v@lY}\v@lY\v@lZa=\repdecn@mb{\v@lZ}\v@lZ%
    \advance\v@lXa\v@lYa\advance\v@lXa\v@lZa\sqrt@{\global\result@t}{\v@lXa}}#1=\result@t}
\ctr@ln@m\n@rmeuc
\ctr@ld@f\def\n@rmeucDD#1#2{{\Figg@tXY{#2}\getredf@ctDD\f@ctech(\v@lX,\v@lY)%
    \divide\v@lX\f@ctech\divide\v@lY\f@ctech%
    \v@lXa=\repdecn@mb{\v@lX}\v@lX\v@lYa=\repdecn@mb{\v@lY}\v@lY%
    \advance\v@lXa\v@lYa\sqrt@{\global\result@t}{\v@lXa}%
    \global\multiply\result@t\f@ctech}#1=\result@t}
\ctr@ld@f\def\n@rmeucTD#1#2{{\Figg@tXY{#2}\getredf@ctTD\f@ctech(\v@lX,\v@lY,\v@lZ)%
    \divide\v@lX\f@ctech\divide\v@lY\f@ctech\divide\v@lZ\f@ctech%
    \v@lXa=\repdecn@mb{\v@lX}\v@lX%
    \v@lYa=\repdecn@mb{\v@lY}\v@lY\v@lZa=\repdecn@mb{\v@lZ}\v@lZ%
    \advance\v@lXa\v@lYa\advance\v@lXa\v@lZa\sqrt@{\global\result@t}{\v@lXa}%
    \global\multiply\result@t\f@ctech}#1=\result@t}
\ctr@ln@m\vecunit@
\ctr@ld@f\def\vecunit@DD#1#2{{\Figg@tXY{#2}\getredf@ctDD\f@ctech(\v@lX,\v@lY)%
    \divide\v@lX\f@ctech\divide\v@lY\f@ctech%
    \Figv@ctCreg#1(\v@lX,\v@lY)\n@rmeucSV{\v@lYa}{#1}%
    \invers@{\v@lXa}{\v@lYa}\edef\v@lv@lXa{\repdecn@mb{\v@lXa}}%
    \v@lX=\v@lv@lXa\v@lX\v@lY=\v@lv@lXa\v@lY%
    \Figv@ctCreg#1(\v@lX,\v@lY)\multiply\v@lYa\f@ctech\global\result@t=\v@lYa}}
\ctr@ld@f\def\vecunit@TD#1#2{{\Figg@tXY{#2}\getredf@ctTD\f@ctech(\v@lX,\v@lY,\v@lZ)%
    \divide\v@lX\f@ctech\divide\v@lY\f@ctech\divide\v@lZ\f@ctech%
    \Figv@ctCreg#1(\v@lX,\v@lY,\v@lZ)\n@rmeucSV{\v@lYa}{#1}%
    \invers@{\v@lXa}{\v@lYa}\edef\v@lv@lXa{\repdecn@mb{\v@lXa}}%
    \v@lX=\v@lv@lXa\v@lX\v@lY=\v@lv@lXa\v@lY\v@lZ=\v@lv@lXa\v@lZ%
    \Figv@ctCreg#1(\v@lX,\v@lY,\v@lZ)\multiply\v@lYa\f@ctech\global\result@t=\v@lYa}}
\ctr@ld@f\def\vecunitC@TD[#1,#2]{\Figg@tXYa{#1}\Figg@tXY{#2}%
    \advance\v@lX-\v@lXa\advance\v@lY-\v@lYa\advance\v@lZ-\v@lZa\c@lvecunitTD}
\ctr@ld@f\def\vecunitCV@TD#1{\Figg@tXY{#1}\c@lvecunitTD}
\ctr@ld@f\def\c@lvecunitTD{\getredf@ctTD\f@ctech(\v@lX,\v@lY,\v@lZ)%
    \divide\v@lX\f@ctech\divide\v@lY\f@ctech\divide\v@lZ\f@ctech%
    \v@lXa=\repdecn@mb{\v@lX}\v@lX%
    \v@lYa=\repdecn@mb{\v@lY}\v@lY\v@lZa=\repdecn@mb{\v@lZ}\v@lZ%
    \advance\v@lXa\v@lYa\advance\v@lXa\v@lZa\sqrt@{\v@lYa}{\v@lXa}%
    \invers@{\v@lXa}{\v@lYa}\edef\v@lv@lXa{\repdecn@mb{\v@lXa}}%
    \v@lX=\v@lv@lXa\v@lX\v@lY=\v@lv@lXa\v@lY\v@lZ=\v@lv@lXa\v@lZ}
\ctr@ln@m\figgetangle
\ctr@ld@f\def\figgetangleDD#1[#2,#3,#4]{\ifps@cri{\s@uvc@ntr@l\et@tfiggetangleDD\setc@ntr@l{2}%
    \figvectPDD-1[#2,#3]\figvectPDD-2[#2,#4]\vecunit@{-1}{-1}%
    \c@lproscalDD\delt@[-2,-1]\figvectNVDD-1[-1]\c@lproscalDD\v@leur[-2,-1]%
    \arct@n\v@lmax(\delt@,\v@leur)\v@lmax=\rdT@deg\v@lmax%
    \ifdim\v@lmax<\z@\advance\v@lmax\DePI@deg\fi\xdef#1{\repdecn@mb{\v@lmax}}%
    \resetc@ntr@l\et@tfiggetangleDD}\ignorespaces\fi}
\ctr@ld@f\def\figgetangleTD#1[#2,#3,#4,#5]{\ifps@cri{\s@uvc@ntr@l\et@tfiggetangleTD\setc@ntr@l{2}%
    \figvectPTD-1[#2,#3]\figvectPTD-2[#2,#5]\figvectNVTD-3[-1,-2]%
    \figvectPTD-2[#2,#4]\figvectNVTD-4[-3,-1]%
    \vecunit@{-1}{-1}\c@lproscalTD\delt@[-2,-1]\c@lproscalTD\v@leur[-2,-4]%
    \arct@n\v@lmax(\delt@,\v@leur)\v@lmax=\rdT@deg\v@lmax%
    \ifdim\v@lmax<\z@\advance\v@lmax\DePI@deg\fi\xdef#1{\repdecn@mb{\v@lmax}}%
    \resetc@ntr@l\et@tfiggetangleTD}\ignorespaces\fi}    
\ctr@ld@f\def\figgetdist#1[#2,#3]{\ifps@cri{\s@uvc@ntr@l\et@tfiggetdist\setc@ntr@l{2}%
    \figvectP-1[#2,#3]\n@rmeuc{\v@lX}{-1}\v@lX=\ptT@unit@\v@lX\xdef#1{\repdecn@mb{\v@lX}}%
    \resetc@ntr@l\et@tfiggetdist}\ignorespaces\fi}
\ctr@ld@f\def\Figg@tT#1{\c@ntr@lnum{#1}%
    {\expandafter\expandafter\expandafter\extr@ctT\csname\objc@de\endcsname:%
     \ifnum\B@@ltxt=\z@\ptn@me{#1}\else\csname\objc@de T\endcsname\fi}}
\ctr@ld@f\def\extr@ctT#1,#2,#3/#4:{\def\B@@ltxt{#3}}
\ctr@ld@f\def\Figg@tXY#1{\c@ntr@lnum{#1}%
    \expandafter\expandafter\expandafter\extr@ctC\csname\objc@de\endcsname:}
\ctr@ln@m\extr@ctC
\ctr@ld@f\def\extr@ctCDD#1/#2,#3,#4:{\v@lX=#2\v@lY=#3}
\ctr@ld@f\def\extr@ctCTD#1/#2,#3,#4:{\v@lX=#2\v@lY=#3\v@lZ=#4}
\ctr@ld@f\def\Figg@tXYa#1{\c@ntr@lnum{#1}%
    \expandafter\expandafter\expandafter\extr@ctCa\csname\objc@de\endcsname:}
\ctr@ln@m\extr@ctCa
\ctr@ld@f\def\extr@ctCaDD#1/#2,#3,#4:{\v@lXa=#2\v@lYa=#3}
\ctr@ld@f\def\extr@ctCaTD#1/#2,#3,#4:{\v@lXa=#2\v@lYa=#3\v@lZa=#4}
\ctr@ln@m\t@xt@
\ctr@ld@f\def\figinit#1{\t@stc@tcodech@nge\initpr@lim\Figinit@#1,:\initpss@ttings\ignorespaces}
\ctr@ld@f\def\Figinit@#1,#2:{\setunit@{#1}\def\t@xt@{#2}\ifx\t@xt@\empty\else\Figinit@@#2:\fi}
\ctr@ld@f\def\Figinit@@#1#2:{\if#12 \else\Figs@tproj{#1}\initTD@\fi}
\ctr@ln@w{newif}\ifTr@isDim
\ctr@ld@f\def\UnD@fined{UNDEFINED}
\ctr@ld@f\def\ifundefined#1{\expandafter\ifx\csname#1\endcsname\relax}
\ctr@ln@m\@utoFN
\ctr@ln@m\@utoFInDone
\ctr@ln@m\disob@unit
\ctr@ld@f\def\initpr@lim{\initb@undb@x\figsetmark{}\figsetptname{$A_{##1}$}\def\Sc@leFact{1}%
    \initDD@\figsetroundcoord{yes}\ps@critrue\expandafter\setupd@te\defaultupdate:%
    \edef\disob@unit{\UnD@fined}\edef\t@rgetpt{\UnD@fined}\gdef\@utoFInDone{1}\gdef\@utoFN{0}}
\ctr@ld@f\def\initDD@{\Tr@isDimfalse%
    \ifPDFm@ke%
     \let\Ps@rcerc=\Ps@rcercBz%
     \let\Ps@rell=\Ps@rellBz%
    \fi
    \let\c@lDCUn=\c@lDCUnDD%
    \let\c@lDCDeux=\c@lDCDeuxDD%
    \let\c@ldefproj=\relax%
    \let\c@lproscal=\c@lproscalDD%
    \let\c@lprojSP=\relax%
    \let\extr@ctC=\extr@ctCDD%
    \let\extr@ctCa=\extr@ctCaDD%
    \let\extr@ctCF=\extr@ctCFDD%
    \let\Figp@intreg=\Figp@intregDD%
    \let\Figpts@xes=\Figpts@xesDD%
    \let\n@rmeucSV=\n@rmeucSVDD\let\n@rmeuc=\n@rmeucDD\let\n@rminf=\n@rminfDD%
    \let\pr@dMatV=\pr@dMatVDD%
    \let\ps@xes=\ps@xesDD%
    \let\vecunit@=\vecunit@DD%
    \let\figcoord=\figcoordDD%
    \let\figgetangle=\figgetangleDD%
    \let\figpt=\figptDD%
    \let\figptBezier=\figptBezierDD%
    \let\figptbary=\figptbaryDD%
    \let\figptcirc=\figptcircDD%
    \let\figptcircumcenter=\figptcircumcenterDD%
    \let\figptcopy=\figptcopyDD%
    \let\figptcurvcenter=\figptcurvcenterDD%
    \let\figptell=\figptellDD%
    \let\figptendnormal=\figptendnormalDD%
    \let\figptinterlineplane=\figptinterlineplaneDD%
    \let\figptinterlines=\inters@cDD%
    \let\figptorthocenter=\figptorthocenterDD%
    \let\figptorthoprojline=\figptorthoprojlineDD%
    \let\figptorthoprojplane=\figptorthoprojplaneDD%
    \let\figptrot=\figptrotDD%
    \let\figptscontrol=\figptscontrolDD%
    \let\figptsintercirc=\figptsintercircDD%
    \let\figptsinterlinell=\figptsinterlinellDD%
    \let\figptsorthoprojline=\figptsorthoprojlineDD%
    \let\figptorthoprojplane=\figptorthoprojplaneDD%
    \let\figptsrot=\figptsrotDD%
    \let\figptssym=\figptssymDD%
    \let\figptstra=\figptstraDD%
    \let\figptsym=\figptsymDD%
    \let\figpttraC=\figpttraCDD%
    \let\figpttra=\figpttraDD%
    \let\figptvisilimSL=\figptvisilimSLDD%
    \let\figsetobdist=\figsetobdistDD%
    \let\figsettarget=\figsettargetDD%
    \let\figsetview=\figsetviewDD%
    \let\figvectDBezier=\figvectDBezierDD%
    \let\figvectN=\figvectNDD%
    \let\figvectNV=\figvectNVDD%
    \let\figvectP=\figvectPDD%
    \let\figvectU=\figvectUDD%
    \let\psarccircP=\psarccircPDD%
    \let\psarccirc=\psarccircDD%
    \let\psarcell=\psarcellDD%
    \let\psarcellPA=\psarcellPADD%
    \let\psarrowBezier=\psarrowBezierDD%
    \let\psarrowcircP=\psarrowcircPDD%
    \let\psarrowcirc=\psarrowcircDD%
    \let\psarrowhead=\psarrowheadDD%
    \let\psarrow=\psarrowDD%
    \let\psBezier=\psBezierDD%
    \let\pscirc=\pscircDD%
    \let\pscurve=\pscurveDD%
    \let\psnormal=\psnormalDD%
    }
\ctr@ld@f\def\initTD@{\Tr@isDimtrue\initb@undb@xTD\newt@rgetptfalse\newdis@bfalse%
    \let\c@lDCUn=\c@lDCUnTD%
    \let\c@lDCDeux=\c@lDCDeuxTD%
    \let\c@ldefproj=\c@ldefprojTD%
    \let\c@lproscal=\c@lproscalTD%
    \let\extr@ctC=\extr@ctCTD%
    \let\extr@ctCa=\extr@ctCaTD%
    \let\extr@ctCF=\extr@ctCFTD%
    \let\Figp@intreg=\Figp@intregTD%
    \let\Figpts@xes=\Figpts@xesTD%
    \let\n@rmeucSV=\n@rmeucSVTD\let\n@rmeuc=\n@rmeucTD\let\n@rminf=\n@rminfTD%
    \let\pr@dMatV=\pr@dMatVTD%
    \let\ps@xes=\ps@xesTD%
    \let\vecunit@=\vecunit@TD%
    \let\figcoord=\figcoordTD%
    \let\figgetangle=\figgetangleTD%
    \let\figpt=\figptTD%
    \let\figptBezier=\figptBezierTD%
    \let\figptbary=\figptbaryTD%
    \let\figptcirc=\figptcircTD%
    \let\figptcircumcenter=\figptcircumcenterTD%
    \let\figptcopy=\figptcopyTD%
    \let\figptcurvcenter=\figptcurvcenterTD%
    \let\figptinterlineplane=\figptinterlineplaneTD%
    \let\figptinterlines=\inters@cTD%
    \let\figptorthocenter=\figptorthocenterTD%
    \let\figptorthoprojline=\figptorthoprojlineTD%
    \let\figptorthoprojplane=\figptorthoprojplaneTD%
    \let\figptrot=\figptrotTD%
    \let\figptscontrol=\figptscontrolTD%
    \let\figptsintercirc=\figptsintercircTD%
    \let\figptsorthoprojline=\figptsorthoprojlineTD%
    \let\figptsorthoprojplane=\figptsorthoprojplaneTD%
    \let\figptsrot=\figptsrotTD%
    \let\figptssym=\figptssymTD%
    \let\figptstra=\figptstraTD%
    \let\figptsym=\figptsymTD%
    \let\figpttraC=\figpttraCTD%
    \let\figpttra=\figpttraTD%
    \let\figptvisilimSL=\figptvisilimSLTD%
    \let\figsetobdist=\figsetobdistTD%
    \let\figsettarget=\figsettargetTD%
    \let\figsetview=\figsetviewTD%
    \let\figvectDBezier=\figvectDBezierTD%
    \let\figvectN=\figvectNTD%
    \let\figvectNV=\figvectNVTD%
    \let\figvectP=\figvectPTD%
    \let\figvectU=\figvectUTD%
    \let\psarccircP=\psarccircPTD%
    \let\psarccirc=\psarccircTD%
    \let\psarcell=\psarcellTD%
    \let\psarcellPA=\psarcellPATD%
    \let\psarrowBezier=\psarrowBezierTD%
    \let\psarrowcircP=\psarrowcircPTD%
    \let\psarrowcirc=\psarrowcircTD%
    \let\psarrowhead=\psarrowheadTD%
    \let\psarrow=\psarrowTD%
    \let\psBezier=\psBezierTD%
    \let\pscirc=\pscircTD%
    \let\pscurve=\pscurveTD%
    }
\ctr@ld@f\def\un@v@ilable#1{\immediate\write16{*** The macro #1 is not available in the current context.}}
\ctr@ld@f\def\figinsert#1{{\def\t@xt@{#1}\relax%
    \ifx\t@xt@\empty\ifnum\@utoFInDone>\z@\Figinsert@\DefGIfilen@me,:\fi%
    \else\expandafter\FiginsertNu@#1 :\fi}\ignorespaces}
\ctr@ld@f\def\FiginsertNu@#1 #2:{\def\t@xt@{#1}\relax\ifx\t@xt@\empty\def\t@xt@{#2}%
    \ifx\t@xt@\empty\ifnum\@utoFInDone>\z@\Figinsert@\DefGIfilen@me,:\fi%
    \else\FiginsertNu@#2:\fi\else\expandafter\FiginsertNd@#1 #2:\fi}
\ctr@ld@f\def\FiginsertNd@#1#2:{\ifcat#1a\Figinsert@#1#2,:\else%
    \ifnum\@utoFInDone>\z@\Figinsert@\DefGIfilen@me,#1#2,:\fi\fi}
\ctr@ln@m\Sc@leFact
\ctr@ld@f\def\Figinsert@#1,#2:{\def\t@xt@{#2}\ifx\t@xt@\empty\xdef\Sc@leFact{1}\else%
    \X@rgdeux@#2\xdef\Sc@leFact{\@rgdeux}\fi%
    \Figdisc@rdLTS{#1}{\t@xt@}\@psfgetbb{\t@xt@}%
    \v@lX=\@psfllx\p@\v@lX=\ptpsT@pt\v@lX\v@lX=\Sc@leFact\v@lX%
    \v@lY=\@psflly\p@\v@lY=\ptpsT@pt\v@lY\v@lY=\Sc@leFact\v@lY%
    \b@undb@x{\v@lX}{\v@lY}%
    \v@lX=\@psfurx\p@\v@lX=\ptpsT@pt\v@lX\v@lX=\Sc@leFact\v@lX%
    \v@lY=\@psfury\p@\v@lY=\ptpsT@pt\v@lY\v@lY=\Sc@leFact\v@lY%
    \b@undb@x{\v@lX}{\v@lY}%
    \ifPDFm@ke\Figinclud@PDF{\t@xt@}{\Sc@leFact}\else%
    \v@lX=\c@nt pt\v@lX=\Sc@leFact\v@lX\edef\F@ct{\repdecn@mb{\v@lX}}%
    \ifx\TeXturesonMacOSltX\special{postscriptfile #1 vscale=\F@ct\space hscale=\F@ct}%
    \else\includegraphics{#1}\fi\fi%
    \message{[\t@xt@]}\ignorespaces}
\ctr@ld@f\def\Figdisc@rdLTS#1#2{\expandafter\Figdisc@rdLTS@#1 :#2}
\ctr@ld@f\def\Figdisc@rdLTS@#1 #2:#3{\def#3{#1}\relax\ifx#3\empty\expandafter\Figdisc@rdLTS@#2:#3\fi}
\ctr@ld@f\def\figinsertE#1{\FiginsertE@#1,:\ignorespaces}
\ctr@ld@f\def\FiginsertE@#1,#2:{{\def\t@xt@{#2}\ifx\t@xt@\empty\xdef\Sc@leFact{1}\else%
    \X@rgdeux@#2\xdef\Sc@leFact{\@rgdeux}\fi%
    \Figdisc@rdLTS{#1}{\t@xt@}\pdfximage{\t@xt@}%
    \setbox\Gb@x=\hbox{\pdfrefximage\pdflastximage}%
    \v@lX=\z@\v@lY=-\Sc@leFact\dp\Gb@x\b@undb@x{\v@lX}{\v@lY}%
    \advance\v@lX\Sc@leFact\wd\Gb@x\advance\v@lY\Sc@leFact\dp\Gb@x%
    \advance\v@lY\Sc@leFact\ht\Gb@x\b@undb@x{\v@lX}{\v@lY}%
    \v@lX=\Sc@leFact\wd\Gb@x\pdfximage width \v@lX {\t@xt@}%
    \rlap{\pdfrefximage\pdflastximage}\message{[\t@xt@]}}\ignorespaces}
\ctr@ld@f\def\X@rgdeux@#1,{\edef\@rgdeux{#1}}
\ctr@ln@m\figpt
\ctr@ld@f\def\figptDD#1:#2(#3,#4){\ifps@cri\c@ntr@lnum{#1}%
    {\v@lX=#3\unit@\v@lY=#4\unit@\Fig@dmpt{#2}{\z@}}\ignorespaces\fi}
\ctr@ld@f\def\Fig@dmpt#1#2{\def\t@xt@{#1}\ifx\t@xt@\empty\def\B@@ltxt{\z@}%
    \else\expandafter\gdef\csname\objc@de T\endcsname{#1}\def\B@@ltxt{\@ne}\fi%
    \expandafter\xdef\csname\objc@de\endcsname{\ifitis@vect@r\C@dCl@svect%
    \else\C@dCl@spt\fi,\z@,\B@@ltxt/\the\v@lX,\the\v@lY,#2}}
\ctr@ld@f\def\C@dCl@spt{P}
\ctr@ld@f\def\C@dCl@svect{V}
\ctr@ln@m\c@@rdYZ
\ctr@ln@m\c@@rdY
\ctr@ld@f\def\figptTD#1:#2(#3,#4){\ifps@cri\c@ntr@lnum{#1}%
    \def\c@@rdYZ{#4,0,0}\extrairelepremi@r\c@@rdY\de\c@@rdYZ%
    \extrairelepremi@r\c@@rdZ\de\c@@rdYZ%
    {\v@lX=#3\unit@\v@lY=\c@@rdY\unit@\v@lZ=\c@@rdZ\unit@\Fig@dmpt{#2}{\the\v@lZ}%
    \b@undb@xTD{\v@lX}{\v@lY}{\v@lZ}}\ignorespaces\fi}
\ctr@ln@m\Figp@intreg
\ctr@ld@f\def\Figp@intregDD#1:#2(#3,#4){\c@ntr@lnum{#1}%
    {\result@t=#4\v@lX=#3\v@lY=\result@t\Fig@dmpt{#2}{\z@}}\ignorespaces}
\ctr@ld@f\def\Figp@intregTD#1:#2(#3,#4){\c@ntr@lnum{#1}%
    \def\c@@rdYZ{#4,\z@,\z@}\extrairelepremi@r\c@@rdY\de\c@@rdYZ%
    \extrairelepremi@r\c@@rdZ\de\c@@rdYZ%
    {\v@lX=#3\v@lY=\c@@rdY\v@lZ=\c@@rdZ\Fig@dmpt{#2}{\the\v@lZ}%
    \b@undb@xTD{\v@lX}{\v@lY}{\v@lZ}}\ignorespaces}
\ctr@ln@m\figptBezier
\ctr@ld@f\def\figptBezierDD#1:#2:#3[#4,#5,#6,#7]{\ifps@cri{\s@uvc@ntr@l\et@tfigptBezierDD%
    \FigptBezier@#3[#4,#5,#6,#7]\Figp@intregDD#1:{#2}(\v@lX,\v@lY)%
    \resetc@ntr@l\et@tfigptBezierDD}\ignorespaces\fi}
\ctr@ld@f\def\figptBezierTD#1:#2:#3[#4,#5,#6,#7]{\ifps@cri{\s@uvc@ntr@l\et@tfigptBezierTD%
    \FigptBezier@#3[#4,#5,#6,#7]\Figp@intregTD#1:{#2}(\v@lX,\v@lY,\v@lZ)%
    \resetc@ntr@l\et@tfigptBezierTD}\ignorespaces\fi}
\ctr@ld@f\def\FigptBezier@#1[#2,#3,#4,#5]{\setc@ntr@l{2}%
    \edef\T@{#1}\v@leur=\p@\advance\v@leur-#1pt\edef\UNmT@{\repdecn@mb{\v@leur}}%
    \figptcopy-4:/#2/\figptcopy-3:/#3/\figptcopy-2:/#4/\figptcopy-1:/#5/%
    \l@mbd@un=-4 \l@mbd@de=-\thr@@\p@rtent=\m@ne\c@lDecast%
    \l@mbd@un=-4 \l@mbd@de=-\thr@@\p@rtent=-\tw@\c@lDecast%
    \l@mbd@un=-4 \l@mbd@de=-\thr@@\p@rtent=-\thr@@\c@lDecast\Figg@tXY{-4}}
\ctr@ln@m\c@lDCUn
\ctr@ld@f\def\c@lDCUnDD#1#2{\Figg@tXY{#1}\v@lX=\UNmT@\v@lX\v@lY=\UNmT@\v@lY%
    \Figg@tXYa{#2}\advance\v@lX\T@\v@lXa\advance\v@lY\T@\v@lYa%
    \Figp@intregDD#1:(\v@lX,\v@lY)}
\ctr@ld@f\def\c@lDCUnTD#1#2{\Figg@tXY{#1}\v@lX=\UNmT@\v@lX\v@lY=\UNmT@\v@lY\v@lZ=\UNmT@\v@lZ%
    \Figg@tXYa{#2}\advance\v@lX\T@\v@lXa\advance\v@lY\T@\v@lYa\advance\v@lZ\T@\v@lZa%
    \Figp@intregTD#1:(\v@lX,\v@lY,\v@lZ)}
\ctr@ld@f\def\c@lDecast{\relax\ifnum\l@mbd@un<\p@rtent\c@lDCUn{\l@mbd@un}{\l@mbd@de}%
    \advance\l@mbd@un\@ne\advance\l@mbd@de\@ne\c@lDecast\fi}
\ctr@ld@f\def\figptmap#1:#2=#3/#4/#5/{\ifps@cri{\s@uvc@ntr@l\et@tfigptmap%
    \setc@ntr@l{2}\figvectP-1[#4,#3]\Figg@tXY{-1}%
    \pr@dMatV/#5/\figpttra#1:{#2}=#4/1,-1/%
    \resetc@ntr@l\et@tfigptmap}\ignorespaces\fi}
\ctr@ln@m\pr@dMatV
\ctr@ld@f\def\pr@dMatVDD/#1,#2;#3,#4/{\v@lXa=#1\v@lX\advance\v@lXa#2\v@lY%
    \v@lYa=#3\v@lX\advance\v@lYa#4\v@lY\Figv@ctCreg-1(\v@lXa,\v@lYa)}
\ctr@ld@f\def\pr@dMatVTD/#1,#2,#3;#4,#5,#6;#7,#8,#9/{%
    \v@lXa=#1\v@lX\advance\v@lXa#2\v@lY\advance\v@lXa#3\v@lZ%
    \v@lYa=#4\v@lX\advance\v@lYa#5\v@lY\advance\v@lYa#6\v@lZ%
    \v@lZa=#7\v@lX\advance\v@lZa#8\v@lY\advance\v@lZa#9\v@lZ%
    \Figv@ctCreg-1(\v@lXa,\v@lYa,\v@lZa)}
\ctr@ln@m\figptbary
\ctr@ld@f\def\figptbaryDD#1:#2[#3;#4]{\ifps@cri{\edef\list@num{#3}\extrairelepremi@r\p@int\de\list@num%
    \s@mme=\z@\@ecfor\c@ef:=#4\do{\advance\s@mme\c@ef}%
    \edef\listec@ef{#4,0}\extrairelepremi@r\c@ef\de\listec@ef%
    \Figg@tXY{\p@int}\divide\v@lX\s@mme\divide\v@lY\s@mme%
    \multiply\v@lX\c@ef\multiply\v@lY\c@ef%
    \@ecfor\p@int:=\list@num\do{\extrairelepremi@r\c@ef\de\listec@ef%
           \Figg@tXYa{\p@int}\divide\v@lXa\s@mme\divide\v@lYa\s@mme%
           \multiply\v@lXa\c@ef\multiply\v@lYa\c@ef%
           \advance\v@lX\v@lXa\advance\v@lY\v@lYa}%
    \Figp@intregDD#1:{#2}(\v@lX,\v@lY)}\ignorespaces\fi}
\ctr@ld@f\def\figptbaryTD#1:#2[#3;#4]{\ifps@cri{\edef\list@num{#3}\extrairelepremi@r\p@int\de\list@num%
    \s@mme=\z@\@ecfor\c@ef:=#4\do{\advance\s@mme\c@ef}%
    \edef\listec@ef{#4,0}\extrairelepremi@r\c@ef\de\listec@ef%
    \Figg@tXY{\p@int}\divide\v@lX\s@mme\divide\v@lY\s@mme\divide\v@lZ\s@mme%
    \multiply\v@lX\c@ef\multiply\v@lY\c@ef\multiply\v@lZ\c@ef%
    \@ecfor\p@int:=\list@num\do{\extrairelepremi@r\c@ef\de\listec@ef%
           \Figg@tXYa{\p@int}\divide\v@lXa\s@mme\divide\v@lYa\s@mme\divide\v@lZa\s@mme%
           \multiply\v@lXa\c@ef\multiply\v@lYa\c@ef\multiply\v@lZa\c@ef%
           \advance\v@lX\v@lXa\advance\v@lY\v@lYa\advance\v@lZ\v@lZa}%
    \Figp@intregTD#1:{#2}(\v@lX,\v@lY,\v@lZ)}\ignorespaces\fi}
\ctr@ld@f\def\figptbaryR#1:#2[#3;#4]{\ifps@cri{%
    \v@leur=\z@\@ecfor\c@ef:=#4\do{\maxim@m{\v@lmax}{\c@ef pt}{-\c@ef pt}%
    \ifdim\v@lmax>\v@leur\v@leur=\v@lmax\fi}%
    \ifdim\v@leur<\p@\f@ctech=\@M\else\ifdim\v@leur<\t@n\p@\f@ctech=\@m\else%
    \ifdim\v@leur<\c@nt\p@\f@ctech=\c@nt\else\ifdim\v@leur<\@m\p@\f@ctech=\t@n\else%
    \f@ctech=\@ne\fi\fi\fi\fi%
    \def\listec@ef{0}%
    \@ecfor\c@ef:=#4\do{\sc@lec@nvRI{\c@ef pt}\edef\listec@ef{\listec@ef,\the\s@mme}}%
    \extrairelepremi@r\c@ef\de\listec@ef\figptbary#1:#2[#3;\listec@ef]}\ignorespaces\fi}
\ctr@ld@f\def\sc@lec@nvRI#1{\v@leur=#1\p@rtentiere{\s@mme}{\v@leur}\advance\v@leur-\s@mme\p@%
    \multiply\v@leur\f@ctech\p@rtentiere{\p@rtent}{\v@leur}%
    \multiply\s@mme\f@ctech\advance\s@mme\p@rtent}
\ctr@ln@m\figptcirc
\ctr@ld@f\def\figptcircDD#1:#2:#3;#4(#5){\ifps@cri{\s@uvc@ntr@l\et@tfigptcircDD%
    \c@lptellDD#1:{#2}:#3;#4,#4(#5)\resetc@ntr@l\et@tfigptcircDD}\ignorespaces\fi}
\ctr@ld@f\def\figptcircTD#1:#2:#3,#4,#5;#6(#7){\ifps@cri{\s@uvc@ntr@l\et@tfigptcircTD%
    \setc@ntr@l{2}\c@lExtAxes#3,#4,#5(#6)\figptellP#1:{#2}:#3,-4,-5(#7)%
    \resetc@ntr@l\et@tfigptcircTD}\ignorespaces\fi}
\ctr@ln@m\figptcircumcenter
\ctr@ld@f\def\figptcircumcenterDD#1:#2[#3,#4,#5]{\ifps@cri{\s@uvc@ntr@l\et@tfigptcircumcenterDD%
    \setc@ntr@l{2}\figvectNDD-5[#3,#4]\figptbaryDD-3:[#3,#4;1,1]%
                  \figvectNDD-6[#4,#5]\figptbaryDD-4:[#4,#5;1,1]%
    \resetc@ntr@l{2}\inters@cDD#1:{#2}[-3,-5;-4,-6]%
    \resetc@ntr@l\et@tfigptcircumcenterDD}\ignorespaces\fi}
\ctr@ld@f\def\figptcircumcenterTD#1:#2[#3,#4,#5]{\ifps@cri{\s@uvc@ntr@l\et@tfigptcircumcenterTD%
    \setc@ntr@l{2}\figvectNTD-1[#3,#4,#5]%
    \figvectPTD-3[#3,#4]\figvectNVTD-5[-1,-3]\figptbaryTD-3:[#3,#4;1,1]%
    \figvectPTD-4[#4,#5]\figvectNVTD-6[-1,-4]\figptbaryTD-4:[#4,#5;1,1]%
    \resetc@ntr@l{2}\inters@cTD#1:{#2}[-3,-5;-4,-6]%
    \resetc@ntr@l\et@tfigptcircumcenterTD}\ignorespaces\fi}
\ctr@ln@m\figptcopy
\ctr@ld@f\def\figptcopyDD#1:#2/#3/{\ifps@cri{\Figg@tXY{#3}%
    \Figp@intregDD#1:{#2}(\v@lX,\v@lY)}\ignorespaces\fi}
\ctr@ld@f\def\figptcopyTD#1:#2/#3/{\ifps@cri{\Figg@tXY{#3}%
    \Figp@intregTD#1:{#2}(\v@lX,\v@lY,\v@lZ)}\ignorespaces\fi}
\ctr@ln@m\figptcurvcenter
\ctr@ld@f\def\figptcurvcenterDD#1:#2:#3[#4,#5,#6,#7]{\ifps@cri{\s@uvc@ntr@l\et@tfigptcurvcenterDD%
    \setc@ntr@l{2}\c@lcurvradDD#3[#4,#5,#6,#7]\edef\Sprim@{\repdecn@mb{\result@t}}%
    \figptBezierDD-1::#3[#4,#5,#6,#7]\figpttraDD#1:{#2}=-1/\Sprim@,-5/%
    \resetc@ntr@l\et@tfigptcurvcenterDD}\ignorespaces\fi}
\ctr@ld@f\def\figptcurvcenterTD#1:#2:#3[#4,#5,#6,#7]{\ifps@cri{\s@uvc@ntr@l\et@tfigptcurvcenterTD%
    \setc@ntr@l{2}\figvectDBezierTD -5:1,#3[#4,#5,#6,#7]%
    \figvectDBezierTD -6:2,#3[#4,#5,#6,#7]\vecunit@TD{-5}{-5}%
    \edef\Sprim@{\repdecn@mb{\result@t}}\figvectNVTD-1[-6,-5]%
    \figvectNVTD-5[-5,-1]\c@lproscalTD\v@leur[-6,-5]%
    \invers@{\v@leur}{\v@leur}\v@leur=\Sprim@\v@leur\v@leur=\Sprim@\v@leur%
    \figptBezierTD-1::#3[#4,#5,#6,#7]\edef\Sprim@{\repdecn@mb{\v@leur}}%
    \figpttraTD#1:{#2}=-1/\Sprim@,-5/\resetc@ntr@l\et@tfigptcurvcenterTD}\ignorespaces\fi}
\ctr@ld@f\def\c@lcurvradDD#1[#2,#3,#4,#5]{{\figvectDBezierDD -5:1,#1[#2,#3,#4,#5]%
    \figvectDBezierDD -6:2,#1[#2,#3,#4,#5]\vecunit@DD{-5}{-5}%
    \edef\Sprim@{\repdecn@mb{\result@t}}\figvectNVDD-5[-5]\c@lproscalDD\v@leur[-6,-5]%
    \invers@{\v@leur}{\v@leur}\v@leur=\Sprim@\v@leur\v@leur=\Sprim@\v@leur%
    \global\result@t=\v@leur}}
\ctr@ln@m\figptell
\ctr@ld@f\def\figptellDD#1:#2:#3;#4,#5(#6,#7){\ifps@cri{\s@uvc@ntr@l\et@tfigptell%
    \c@lptellDD#1::#3;#4,#5(#6)\figptrotDD#1:{#2}=#1/#3,#7/%
    \resetc@ntr@l\et@tfigptell}\ignorespaces\fi}
\ctr@ld@f\def\c@lptellDD#1:#2:#3;#4,#5(#6){\c@ssin{\C@}{\S@}{#6}\v@lmin=\C@ pt\v@lmax=\S@ pt%
    \v@lmin=#4\v@lmin\v@lmax=#5\v@lmax%
    \edef\Xc@mp{\repdecn@mb{\v@lmin}}\edef\Yc@mp{\repdecn@mb{\v@lmax}}%
    \setc@ntr@l{2}\figvectC-1(\Xc@mp,\Yc@mp)\figpttraDD#1:{#2}=#3/1,-1/}
\ctr@ld@f\def\figptellP#1:#2:#3,#4,#5(#6){\ifps@cri{\s@uvc@ntr@l\et@tfigptellP%
    \setc@ntr@l{2}\figvectP-1[#3,#4]\figvectP-2[#3,#5]%
    \v@leur=#6pt\c@lptellP{#3}{-1}{-2}\figptcopy#1:{#2}/-3/%
    \resetc@ntr@l\et@tfigptellP}\ignorespaces\fi}
\ctr@ln@m\@ngle
\ctr@ld@f\def\c@lptellP#1#2#3{\edef\@ngle{\repdecn@mb\v@leur}\c@ssin{\C@}{\S@}{\@ngle}%
    \figpttra-3:=#1/\C@,#2/\figpttra-3:=-3/\S@,#3/}
\ctr@ln@m\figptendnormal
\ctr@ld@f\def\figptendnormalDD#1:#2:#3,#4[#5,#6]{\ifps@cri{\s@uvc@ntr@l\et@tfigptendnormal%
    \Figg@tXYa{#5}\Figg@tXY{#6}%
    \advance\v@lX-\v@lXa\advance\v@lY-\v@lYa%
    \setc@ntr@l{2}\Figv@ctCreg-1(\v@lX,\v@lY)\vecunit@{-1}{-1}\Figg@tXY{-1}%
    \delt@=#3\unit@\maxim@m{\delt@}{\delt@}{-\delt@}\edef\l@ngueur{\repdecn@mb{\delt@}}%
    \v@lX=\l@ngueur\v@lX\v@lY=\l@ngueur\v@lY%
    \delt@=\p@\advance\delt@-#4pt\edef\l@ngueur{\repdecn@mb{\delt@}}%
    \figptbaryR-1:[#5,#6;#4,\l@ngueur]\Figg@tXYa{-1}%
    \advance\v@lXa\v@lY\advance\v@lYa-\v@lX%
    \setc@ntr@l{1}\Figp@intregDD#1:{#2}(\v@lXa,\v@lYa)\resetc@ntr@l\et@tfigptendnormal}%
    \ignorespaces\fi}
\ctr@ld@f\def\figptexcenter#1:#2[#3,#4,#5]{\ifps@cri{\let@xte={-}%
    \Figptexinsc@nter#1:#2[#3,#4,#5]}\ignorespaces\fi}
\ctr@ld@f\def\figptincenter#1:#2[#3,#4,#5]{\ifps@cri{\let@xte={}%
    \Figptexinsc@nter#1:#2[#3,#4,#5]}\ignorespaces\fi}
\ctr@ld@f\let\figptinscribedcenter=\figptincenter
\ctr@ld@f\def\Figptexinsc@nter#1:#2[#3,#4,#5]{%
    \figgetdist\LA@[#4,#5]\figgetdist\LB@[#3,#5]\figgetdist\LC@[#3,#4]%
    \figptbaryR#1:{#2}[#3,#4,#5;\the\let@xte\LA@,\LB@,\LC@]}
\ctr@ln@m\figptinterlineplane
\ctr@ld@f\def\figptinterlineplaneDD{\un@v@ilable{figptinterlineplane}}
\ctr@ld@f\def\figptinterlineplaneTD#1:#2[#3,#4;#5,#6]{\ifps@cri{\s@uvc@ntr@l\et@tfigptinterlineplane%
    \setc@ntr@l{2}\figvectPTD-1[#3,#5]\vecunit@TD{-2}{#6}%
    \r@pPSTD\v@leur[-2,-1,#4]\edef\v@lcoef{\repdecn@mb{\v@leur}}%
    \figpttraTD#1:{#2}=#3/\v@lcoef,#4/\resetc@ntr@l\et@tfigptinterlineplane}\ignorespaces\fi}
\ctr@ln@m\figptorthocenter
\ctr@ld@f\def\figptorthocenterDD#1:#2[#3,#4,#5]{\ifps@cri{\s@uvc@ntr@l\et@tfigptorthocenterDD%
    \setc@ntr@l{2}\figvectNDD-3[#3,#4]\figvectNDD-4[#4,#5]%
    \resetc@ntr@l{2}\inters@cDD#1:{#2}[#5,-3;#3,-4]%
    \resetc@ntr@l\et@tfigptorthocenterDD}\ignorespaces\fi}
\ctr@ld@f\def\figptorthocenterTD#1:#2[#3,#4,#5]{\ifps@cri{\s@uvc@ntr@l\et@tfigptorthocenterTD%
    \setc@ntr@l{2}\figvectNTD-1[#3,#4,#5]%
    \figvectPTD-2[#3,#4]\figvectNVTD-3[-1,-2]%
    \figvectPTD-2[#4,#5]\figvectNVTD-4[-1,-2]%
    \resetc@ntr@l{2}\inters@cTD#1:{#2}[#5,-3;#3,-4]%
    \resetc@ntr@l\et@tfigptorthocenterTD}\ignorespaces\fi}
\ctr@ln@m\figptorthoprojline
\ctr@ld@f\def\figptorthoprojlineDD#1:#2=#3/#4,#5/{\ifps@cri{\s@uvc@ntr@l\et@tfigptorthoprojlineDD%
    \setc@ntr@l{2}\figvectPDD-3[#4,#5]\figvectNVDD-4[-3]\resetc@ntr@l{2}%
    \inters@cDD#1:{#2}[#3,-4;#4,-3]\resetc@ntr@l\et@tfigptorthoprojlineDD}\ignorespaces\fi}
\ctr@ld@f\def\figptorthoprojlineTD#1:#2=#3/#4,#5/{\ifps@cri{\s@uvc@ntr@l\et@tfigptorthoprojlineTD%
    \setc@ntr@l{2}\figvectPTD-1[#4,#3]\figvectPTD-2[#4,#5]\vecunit@TD{-2}{-2}%
    \c@lproscalTD\v@leur[-1,-2]\edef\v@lcoef{\repdecn@mb{\v@leur}}%
    \figpttraTD#1:{#2}=#4/\v@lcoef,-2/\resetc@ntr@l\et@tfigptorthoprojlineTD}\ignorespaces\fi}
\ctr@ln@m\figptorthoprojplane
\ctr@ld@f\def\figptorthoprojplaneDD{\un@v@ilable{figptorthoprojplane}}
\ctr@ld@f\def\figptorthoprojplaneTD#1:#2=#3/#4,#5/{\ifps@cri{\s@uvc@ntr@l\et@tfigptorthoprojplane%
    \setc@ntr@l{2}\figvectPTD-1[#3,#4]\vecunit@TD{-2}{#5}%
    \c@lproscalTD\v@leur[-1,-2]\edef\v@lcoef{\repdecn@mb{\v@leur}}%
    \figpttraTD#1:{#2}=#3/\v@lcoef,-2/\resetc@ntr@l\et@tfigptorthoprojplane}\ignorespaces\fi}
\ctr@ld@f\def\figpthom#1:#2=#3/#4,#5/{\ifps@cri{\s@uvc@ntr@l\et@tfigpthom%
    \setc@ntr@l{2}\figvectP-1[#4,#3]\figpttra#1:{#2}=#4/#5,-1/%
    \resetc@ntr@l\et@tfigpthom}\ignorespaces\fi}
\ctr@ln@m\figptrot
\ctr@ld@f\def\figptrotDD#1:#2=#3/#4,#5/{\ifps@cri{\s@uvc@ntr@l\et@tfigptrotDD%
    \c@ssin{\C@}{\S@}{#5}\setc@ntr@l{2}\figvectPDD-1[#4,#3]\Figg@tXY{-1}%
    \v@lXa=\C@\v@lX\advance\v@lXa-\S@\v@lY%
    \v@lYa=\S@\v@lX\advance\v@lYa\C@\v@lY%
    \Figv@ctCreg-1(\v@lXa,\v@lYa)\figpttraDD#1:{#2}=#4/1,-1/%
    \resetc@ntr@l\et@tfigptrotDD}\ignorespaces\fi}
\ctr@ld@f\def\figptrotTD#1:#2=#3/#4,#5,#6/{\ifps@cri{\s@uvc@ntr@l\et@tfigptrotTD%
    \c@ssin{\C@}{\S@}{#5}%
    \setc@ntr@l{2}\figptorthoprojplaneTD-3:=#4/#3,#6/\figvectPTD-2[-3,#3]%
    \n@rmeucTD\v@leur{-2}\ifdim\v@leur<\Cepsil@n\Figg@tXYa{#3}\else%
    \edef\v@lcoef{\repdecn@mb{\v@leur}}\figvectNVTD-1[#6,-2]%
    \Figg@tXYa{-1}\v@lXa=\v@lcoef\v@lXa\v@lYa=\v@lcoef\v@lYa\v@lZa=\v@lcoef\v@lZa%
    \v@lXa=\S@\v@lXa\v@lYa=\S@\v@lYa\v@lZa=\S@\v@lZa\Figg@tXY{-2}%
    \advance\v@lXa\C@\v@lX\advance\v@lYa\C@\v@lY\advance\v@lZa\C@\v@lZ%
    \Figg@tXY{-3}\advance\v@lXa\v@lX\advance\v@lYa\v@lY\advance\v@lZa\v@lZ\fi%
    \Figp@intregTD#1:{#2}(\v@lXa,\v@lYa,\v@lZa)\resetc@ntr@l\et@tfigptrotTD}\ignorespaces\fi}
\ctr@ln@m\figptsym
\ctr@ld@f\def\figptsymDD#1:#2=#3/#4,#5/{\ifps@cri{\s@uvc@ntr@l\et@tfigptsymDD%
    \resetc@ntr@l{2}\figptorthoprojlineDD-5:=#3/#4,#5/\figvectPDD-2[#3,-5]%
    \figpttraDD#1:{#2}=#3/2,-2/\resetc@ntr@l\et@tfigptsymDD}\ignorespaces\fi}
\ctr@ld@f\def\figptsymTD#1:#2=#3/#4,#5/{\ifps@cri{\s@uvc@ntr@l\et@tfigptsymTD%
    \resetc@ntr@l{2}\figptorthoprojplaneTD-3:=#3/#4,#5/\figvectPTD-2[#3,-3]%
    \figpttraTD#1:{#2}=#3/2,-2/\resetc@ntr@l\et@tfigptsymTD}\ignorespaces\fi}
\ctr@ln@m\figpttra
\ctr@ld@f\def\figpttraDD#1:#2=#3/#4,#5/{\ifps@cri{\Figg@tXYa{#5}\v@lXa=#4\v@lXa\v@lYa=#4\v@lYa%
    \Figg@tXY{#3}\advance\v@lX\v@lXa\advance\v@lY\v@lYa%
    \Figp@intregDD#1:{#2}(\v@lX,\v@lY)}\ignorespaces\fi}
\ctr@ld@f\def\figpttraTD#1:#2=#3/#4,#5/{\ifps@cri{\Figg@tXYa{#5}\v@lXa=#4\v@lXa\v@lYa=#4\v@lYa%
    \v@lZa=#4\v@lZa\Figg@tXY{#3}\advance\v@lX\v@lXa\advance\v@lY\v@lYa%
    \advance\v@lZ\v@lZa\Figp@intregTD#1:{#2}(\v@lX,\v@lY,\v@lZ)}\ignorespaces\fi}
\ctr@ln@m\figpttraC
\ctr@ld@f\def\figpttraCDD#1:#2=#3/#4,#5/{\ifps@cri{\v@lXa=#4\unit@\v@lYa=#5\unit@%
    \Figg@tXY{#3}\advance\v@lX\v@lXa\advance\v@lY\v@lYa%
    \Figp@intregDD#1:{#2}(\v@lX,\v@lY)}\ignorespaces\fi}
\ctr@ld@f\def\figpttraCTD#1:#2=#3/#4,#5,#6/{\ifps@cri{\v@lXa=#4\unit@\v@lYa=#5\unit@\v@lZa=#6\unit@%
    \Figg@tXY{#3}\advance\v@lX\v@lXa\advance\v@lY\v@lYa\advance\v@lZ\v@lZa%
    \Figp@intregTD#1:{#2}(\v@lX,\v@lY,\v@lZ)}\ignorespaces\fi}
\ctr@ld@f\def\figptsaxes#1:#2(#3){\ifps@cri{\an@lys@xes#3,:\ifx\t@xt@\empty%
    \ifTr@isDim\Figpts@xes#1:#2(0,#3,0,#3,0,#3)\else\Figpts@xes#1:#2(0,#3,0,#3)\fi%
    \else\Figpts@xes#1:#2(#3)\fi}\ignorespaces\fi}
\ctr@ln@m\Figpts@xes
\ctr@ld@f\def\Figpts@xesDD#1:#2(#3,#4,#5,#6){%
    \s@mme=#1\figpttraC\the\s@mme:$x$=#2/#4,0/%
    \advance\s@mme\@ne\figpttraC\the\s@mme:$y$=#2/0,#6/}
\ctr@ld@f\def\Figpts@xesTD#1:#2(#3,#4,#5,#6,#7,#8){%
    \s@mme=#1\figpttraC\the\s@mme:$x$=#2/#4,0,0/%
    \advance\s@mme\@ne\figpttraC\the\s@mme:$y$=#2/0,#6,0/%
    \advance\s@mme\@ne\figpttraC\the\s@mme:$z$=#2/0,0,#8/}
\ctr@ld@f\def\figptsmap#1=#2/#3/#4/{\ifps@cri{\s@uvc@ntr@l\et@tfigptsmap%
    \setc@ntr@l{2}\def\list@num{#2}\s@mme=#1%
    \@ecfor\p@int:=\list@num\do{\figvectP-1[#3,\p@int]\Figg@tXY{-1}%
    \pr@dMatV/#4/\figpttra\the\s@mme:=#3/1,-1/\advance\s@mme\@ne}%
    \resetc@ntr@l\et@tfigptsmap}\ignorespaces\fi}
\ctr@ln@m\figptscontrol
\ctr@ld@f\def\figptscontrolDD#1[#2,#3,#4,#5]{\ifps@cri{\s@uvc@ntr@l\et@tfigptscontrolDD\setc@ntr@l{2}%
    \v@lX=\z@\v@lY=\z@\Figtr@nptDD{-5}{#2}\Figtr@nptDD{2}{#5}%
    \divide\v@lX\@vi\divide\v@lY\@vi%
    \Figtr@nptDD{3}{#3}\Figtr@nptDD{-1.5}{#4}\Figp@intregDD-1:(\v@lX,\v@lY)%
    \v@lX=\z@\v@lY=\z@\Figtr@nptDD{2}{#2}\Figtr@nptDD{-5}{#5}%
    \divide\v@lX\@vi\divide\v@lY\@vi\Figtr@nptDD{-1.5}{#3}\Figtr@nptDD{3}{#4}%
    \s@mme=#1\advance\s@mme\@ne\Figp@intregDD\the\s@mme:(\v@lX,\v@lY)%
    \figptcopyDD#1:/-1/\resetc@ntr@l\et@tfigptscontrolDD}\ignorespaces\fi}
\ctr@ld@f\def\figptscontrolTD#1[#2,#3,#4,#5]{\ifps@cri{\s@uvc@ntr@l\et@tfigptscontrolTD\setc@ntr@l{2}%
    \v@lX=\z@\v@lY=\z@\v@lZ=\z@\Figtr@nptTD{-5}{#2}\Figtr@nptTD{2}{#5}%
    \divide\v@lX\@vi\divide\v@lY\@vi\divide\v@lZ\@vi%
    \Figtr@nptTD{3}{#3}\Figtr@nptTD{-1.5}{#4}\Figp@intregTD-1:(\v@lX,\v@lY,\v@lZ)%
    \v@lX=\z@\v@lY=\z@\v@lZ=\z@\Figtr@nptTD{2}{#2}\Figtr@nptTD{-5}{#5}%
    \divide\v@lX\@vi\divide\v@lY\@vi\divide\v@lZ\@vi\Figtr@nptTD{-1.5}{#3}\Figtr@nptTD{3}{#4}%
    \s@mme=#1\advance\s@mme\@ne\Figp@intregTD\the\s@mme:(\v@lX,\v@lY,\v@lZ)%
    \figptcopyTD#1:/-1/\resetc@ntr@l\et@tfigptscontrolTD}\ignorespaces\fi}
\ctr@ld@f\def\Figtr@nptDD#1#2{\Figg@tXYa{#2}\v@lXa=#1\v@lXa\v@lYa=#1\v@lYa%
    \advance\v@lX\v@lXa\advance\v@lY\v@lYa}
\ctr@ld@f\def\Figtr@nptTD#1#2{\Figg@tXYa{#2}\v@lXa=#1\v@lXa\v@lYa=#1\v@lYa\v@lZa=#1\v@lZa%
    \advance\v@lX\v@lXa\advance\v@lY\v@lYa\advance\v@lZ\v@lZa}
\ctr@ld@f\def\figptscontrolcurve#1,#2[#3]{\ifps@cri{\s@uvc@ntr@l\et@tfigptscontrolcurve%
    \def\list@num{#3}\extrairelepremi@r\Ak@\de\list@num%
    \extrairelepremi@r\Ai@\de\list@num\extrairelepremi@r\Aj@\de\list@num%
    \s@mme=#1\figptcopy\the\s@mme:/\Ai@/%
    \setc@ntr@l{2}\figvectP -1[\Ak@,\Aj@]%
    \@ecfor\Ak@:=\list@num\do{\advance\s@mme\@ne\figpttra\the\s@mme:=\Ai@/\curv@roundness,-1/%
       \figvectP -1[\Ai@,\Ak@]\advance\s@mme\@ne\figpttra\the\s@mme:=\Aj@/-\curv@roundness,-1/%
       \advance\s@mme\@ne\figptcopy\the\s@mme:/\Aj@/%
       \edef\Ai@{\Aj@}\edef\Aj@{\Ak@}}\advance\s@mme-#1\divide\s@mme\thr@@%
       \xdef#2{\the\s@mme}%
    \resetc@ntr@l\et@tfigptscontrolcurve}\ignorespaces\fi}
\ctr@ln@m\figptsintercirc
\ctr@ld@f\def\figptsintercircDD#1[#2,#3;#4,#5]{\ifps@cri{\s@uvc@ntr@l\et@tfigptsintercircDD%
    \setc@ntr@l{2}\let\c@lNVintc=\c@lNVintcDD\Figptsintercirc@#1[#2,#3;#4,#5]%
    \resetc@ntr@l\et@tfigptsintercircDD}\ignorespaces\fi}
\ctr@ld@f\def\figptsintercircTD#1[#2,#3;#4,#5;#6]{\ifps@cri{\s@uvc@ntr@l\et@tfigptsintercircTD%
    \setc@ntr@l{2}\let\c@lNVintc=\c@lNVintcTD\vecunitC@TD[#2,#6]%
    \Figv@ctCreg-3(\v@lX,\v@lY,\v@lZ)\Figptsintercirc@#1[#2,#3;#4,#5]%
    \resetc@ntr@l\et@tfigptsintercircTD}\ignorespaces\fi}
\ctr@ld@f\def\Figptsintercirc@#1[#2,#3;#4,#5]{\figvectP-1[#2,#4]%
    \vecunit@{-1}{-1}\delt@=\result@t\f@ctech=\result@tent%
    \s@mme=#1\advance\s@mme\@ne\figptcopy#1:/#2/\figptcopy\the\s@mme:/#4/%
    \ifdim\delt@=\z@\else%
    \v@lmin=#3\unit@\v@lmax=#5\unit@\v@leur=\v@lmin\advance\v@leur\v@lmax%
    \ifdim\v@leur>\delt@%
    \v@leur=\v@lmin\advance\v@leur-\v@lmax\maxim@m{\v@leur}{\v@leur}{-\v@leur}%
    \ifdim\v@leur<\delt@%
    \divide\v@lmin\f@ctech\divide\v@lmax\f@ctech\divide\delt@\f@ctech%
    \v@lmin=\repdecn@mb{\v@lmin}\v@lmin\v@lmax=\repdecn@mb{\v@lmax}\v@lmax%
    \invers@{\v@leur}{\delt@}\advance\v@lmax-\v@lmin%
    \v@lmax=-\repdecn@mb{\v@leur}\v@lmax\advance\delt@\v@lmax\delt@=.5\delt@%
    \v@lmax=\delt@\multiply\v@lmax\f@ctech%
    \edef\t@ille{\repdecn@mb{\v@lmax}}\figpttra-2:=#2/\t@ille,-1/%
    \delt@=\repdecn@mb{\delt@}\delt@\advance\v@lmin-\delt@%
    \sqrt@{\v@leur}{\v@lmin}\multiply\v@leur\f@ctech\edef\t@ille{\repdecn@mb{\v@leur}}%
    \c@lNVintc\figpttra#1:=-2/-\t@ille,-1/\figpttra\the\s@mme:=-2/\t@ille,-1/\fi\fi\fi}
\ctr@ld@f\def\c@lNVintcDD{\Figg@tXY{-1}\Figv@ctCreg-1(-\v@lY,\v@lX)} 
\ctr@ld@f\def\c@lNVintcTD{{\Figg@tXY{-3}\v@lmin=\v@lX\v@lmax=\v@lY\v@leur=\v@lZ%
    \Figg@tXY{-1}\c@lprovec{-3}\vecunit@{-3}{-3}
    \Figg@tXY{-1}\v@lmin=\v@lX\v@lmax=\v@lY%
    \v@leur=\v@lZ\Figg@tXY{-3}\c@lprovec{-1}}} 
\ctr@ln@m\figptsinterlinell
\ctr@ld@f\def\figptsinterlinellDD#1[#2,#3,#4,#5;#6,#7]{\ifps@cri{\s@uvc@ntr@l\et@tfigptsinterlinellDD%
    \figptcopy#1:/#6/\s@mme=#1\advance\s@mme\@ne\figptcopy\the\s@mme:/#7/%
    \v@lmin=#3\unit@\v@lmax=#4\unit@
    \setc@ntr@l{2}\figptbaryDD-4:[#6,#7;1,1]\figptsrotDD-3=-4,#7/#2,-#5/
    \Figg@tXY{-3}\Figg@tXYa{#2}\advance\v@lX-\v@lXa\advance\v@lY-\v@lYa
    \figvectP-1[-3,-2]\Figg@tXYa{-1}\figvectP-3[-4,#7]\Figptsint@rLE{#1}
    \resetc@ntr@l\et@tfigptsinterlinellDD}\ignorespaces\fi}
\ctr@ld@f\def\figptsinterlinellP#1[#2,#3,#4;#5,#6]{\ifps@cri{\s@uvc@ntr@l\et@tfigptsinterlinellP%
    \figptcopy#1:/#5/\s@mme=#1\advance\s@mme\@ne\figptcopy\the\s@mme:/#6/\setc@ntr@l{2}%
    \figvectP-1[#2,#3]\vecunit@{-1}{-1}\v@lmin=\result@t
    \figvectP-2[#2,#4]\vecunit@{-2}{-2}\v@lmax=\result@t
    \figptbary-4:[#5,#6;1,1]
    \figvectP-3[#2,-4]\c@lproscal\v@lX[-3,-1]\c@lproscal\v@lY[-3,-2]
    \figvectP-3[-4,#6]\c@lproscal\v@lXa[-3,-1]\c@lproscal\v@lYa[-3,-2]
    \Figptsint@rLE{#1}\resetc@ntr@l\et@tfigptsinterlinellP}\ignorespaces\fi}
\ctr@ld@f\def\Figptsint@rLE#1{%
    \getredf@ctDD\f@ctech(\v@lmin,\v@lmax)%
    \getredf@ctDD\p@rtent(\v@lX,\v@lY)\ifnum\p@rtent>\f@ctech\f@ctech=\p@rtent\fi%
    \getredf@ctDD\p@rtent(\v@lXa,\v@lYa)\ifnum\p@rtent>\f@ctech\f@ctech=\p@rtent\fi%
    \divide\v@lmin\f@ctech\divide\v@lmax\f@ctech\divide\v@lX\f@ctech\divide\v@lY\f@ctech%
    \divide\v@lXa\f@ctech\divide\v@lYa\f@ctech%
    \c@rre=\repdecn@mb\v@lXa\v@lmax\mili@u=\repdecn@mb\v@lYa\v@lmin%
    \getredf@ctDD\f@ctech(\c@rre,\mili@u)%
    \c@rre=\repdecn@mb\v@lX\v@lmax\mili@u=\repdecn@mb\v@lY\v@lmin%
    \getredf@ctDD\p@rtent(\c@rre,\mili@u)\ifnum\p@rtent>\f@ctech\f@ctech=\p@rtent\fi%
    \divide\v@lmin\f@ctech\divide\v@lmax\f@ctech\divide\v@lX\f@ctech\divide\v@lY\f@ctech%
    \divide\v@lXa\f@ctech\divide\v@lYa\f@ctech%
    \v@lmin=\repdecn@mb{\v@lmin}\v@lmin\v@lmax=\repdecn@mb{\v@lmax}\v@lmax%
    \edef\G@xde{\repdecn@mb\v@lmin}\edef\P@xde{\repdecn@mb\v@lmax}%
    \c@rre=-\v@lmax\v@leur=\repdecn@mb\v@lY\v@lY\advance\c@rre\v@leur\c@rre=\G@xde\c@rre%
    \v@leur=\repdecn@mb\v@lX\v@lX\v@leur=\P@xde\v@leur\advance\c@rre\v@leur
    \v@lmin=\repdecn@mb\v@lYa\v@lmin\v@lmax=\repdecn@mb\v@lXa\v@lmax%
    \mili@u=\repdecn@mb\v@lX\v@lmax\advance\mili@u\repdecn@mb\v@lY\v@lmin
    \v@lmax=\repdecn@mb\v@lXa\v@lmax\advance\v@lmax\repdecn@mb\v@lYa\v@lmin
    \ifdim\v@lmax>\epsil@n%
    \maxim@m{\v@leur}{\c@rre}{-\c@rre}\maxim@m{\v@lmin}{\mili@u}{-\mili@u}%
    \maxim@m{\v@leur}{\v@leur}{\v@lmin}\maxim@m{\v@lmin}{\v@lmax}{-\v@lmax}%
    \maxim@m{\v@leur}{\v@leur}{\v@lmin}\p@rtentiere{\p@rtent}{\v@leur}\advance\p@rtent\@ne%
    \divide\c@rre\p@rtent\divide\mili@u\p@rtent\divide\v@lmax\p@rtent%
    \delt@=\repdecn@mb{\mili@u}\mili@u\v@leur=\repdecn@mb{\v@lmax}\c@rre%
    \advance\delt@-\v@leur\ifdim\delt@<\z@\else\sqrt@\delt@\delt@%
    \invers@\v@lmax\v@lmax\edef\Uns@rAp{\repdecn@mb\v@lmax}%
    \v@leur=-\mili@u\advance\v@leur-\delt@\v@leur=\Uns@rAp\v@leur%
    \edef\t@ille{\repdecn@mb\v@leur}\figpttra#1:=-4/\t@ille,-3/\s@mme=#1\advance\s@mme\@ne%
    \v@leur=-\mili@u\advance\v@leur\delt@\v@leur=\Uns@rAp\v@leur%
    \edef\t@ille{\repdecn@mb\v@leur}\figpttra\the\s@mme:=-4/\t@ille,-3/\fi\fi}
\ctr@ln@m\figptsorthoprojline
\ctr@ld@f\def\figptsorthoprojlineDD#1=#2/#3,#4/{\ifps@cri{\s@uvc@ntr@l\et@tfigptsorthoprojlineDD%
    \setc@ntr@l{2}\figvectPDD-3[#3,#4]\figvectNVDD-4[-3]\resetc@ntr@l{2}%
    \def\list@num{#2}\s@mme=#1\@ecfor\p@int:=\list@num\do{%
    \inters@cDD\the\s@mme:[\p@int,-4;#3,-3]\advance\s@mme\@ne}%
    \resetc@ntr@l\et@tfigptsorthoprojlineDD}\ignorespaces\fi}
\ctr@ld@f\def\figptsorthoprojlineTD#1=#2/#3,#4/{\ifps@cri{\s@uvc@ntr@l\et@tfigptsorthoprojlineTD%
    \setc@ntr@l{2}\figvectPTD-2[#3,#4]\vecunit@TD{-2}{-2}%
    \def\list@num{#2}\s@mme=#1\@ecfor\p@int:=\list@num\do{%
    \figvectPTD-1[#3,\p@int]\c@lproscalTD\v@leur[-1,-2]%
    \edef\v@lcoef{\repdecn@mb{\v@leur}}\figpttraTD\the\s@mme:=#3/\v@lcoef,-2/%
    \advance\s@mme\@ne}\resetc@ntr@l\et@tfigptsorthoprojlineTD}\ignorespaces\fi}
\ctr@ln@m\figptsorthoprojplane
\ctr@ld@f\def\figptsorthoprojplaneDD{\un@v@ilable{figptsorthoprojplane}}
\ctr@ld@f\def\figptsorthoprojplaneTD#1=#2/#3,#4/{\ifps@cri{\s@uvc@ntr@l\et@tfigptsorthoprojplane%
    \setc@ntr@l{2}\vecunit@TD{-2}{#4}%
    \def\list@num{#2}\s@mme=#1\@ecfor\p@int:=\list@num\do{\figvectPTD-1[\p@int,#3]%
    \c@lproscalTD\v@leur[-1,-2]\edef\v@lcoef{\repdecn@mb{\v@leur}}%
    \figpttraTD\the\s@mme:=\p@int/\v@lcoef,-2/\advance\s@mme\@ne}%
    \resetc@ntr@l\et@tfigptsorthoprojplane}\ignorespaces\fi}
\ctr@ld@f\def\figptshom#1=#2/#3,#4/{\ifps@cri{\s@uvc@ntr@l\et@tfigptshom%
    \setc@ntr@l{2}\def\list@num{#2}\s@mme=#1%
    \@ecfor\p@int:=\list@num\do{\figvectP-1[#3,\p@int]%
    \figpttra\the\s@mme:=#3/#4,-1/\advance\s@mme\@ne}%
    \resetc@ntr@l\et@tfigptshom}\ignorespaces\fi}
\ctr@ln@m\figptsrot
\ctr@ld@f\def\figptsrotDD#1=#2/#3,#4/{\ifps@cri{\s@uvc@ntr@l\et@tfigptsrotDD%
    \c@ssin{\C@}{\S@}{#4}\setc@ntr@l{2}\def\list@num{#2}\s@mme=#1%
    \@ecfor\p@int:=\list@num\do{\figvectPDD-1[#3,\p@int]\Figg@tXY{-1}%
    \v@lXa=\C@\v@lX\advance\v@lXa-\S@\v@lY%
    \v@lYa=\S@\v@lX\advance\v@lYa\C@\v@lY%
    \Figv@ctCreg-1(\v@lXa,\v@lYa)\figpttraDD\the\s@mme:=#3/1,-1/\advance\s@mme\@ne}%
    \resetc@ntr@l\et@tfigptsrotDD}\ignorespaces\fi}
\ctr@ld@f\def\figptsrotTD#1=#2/#3,#4,#5/{\ifps@cri{\s@uvc@ntr@l\et@tfigptsrotTD%
    \c@ssin{\C@}{\S@}{#4}%
    \setc@ntr@l{2}\def\list@num{#2}\s@mme=#1%
    \@ecfor\p@int:=\list@num\do{\figptorthoprojplaneTD-3:=#3/\p@int,#5/%
    \figvectPTD-2[-3,\p@int]%
    \figvectNVTD-1[#5,-2]\n@rmeucTD\v@leur{-2}\edef\v@lcoef{\repdecn@mb{\v@leur}}%
    \Figg@tXYa{-1}\v@lXa=\v@lcoef\v@lXa\v@lYa=\v@lcoef\v@lYa\v@lZa=\v@lcoef\v@lZa%
    \v@lXa=\S@\v@lXa\v@lYa=\S@\v@lYa\v@lZa=\S@\v@lZa\Figg@tXY{-2}%
    \advance\v@lXa\C@\v@lX\advance\v@lYa\C@\v@lY\advance\v@lZa\C@\v@lZ%
    \Figg@tXY{-3}\advance\v@lXa\v@lX\advance\v@lYa\v@lY\advance\v@lZa\v@lZ%
    \Figp@intregTD\the\s@mme:(\v@lXa,\v@lYa,\v@lZa)\advance\s@mme\@ne}%
    \resetc@ntr@l\et@tfigptsrotTD}\ignorespaces\fi}
\ctr@ln@m\figptssym
\ctr@ld@f\def\figptssymDD#1=#2/#3,#4/{\ifps@cri{\s@uvc@ntr@l\et@tfigptssymDD%
    \setc@ntr@l{2}\figvectPDD-3[#3,#4]\Figg@tXY{-3}\Figv@ctCreg-4(-\v@lY,\v@lX)%
    \resetc@ntr@l{2}\def\list@num{#2}\s@mme=#1%
    \@ecfor\p@int:=\list@num\do{\inters@cDD-5:[#3,-3;\p@int,-4]\figvectPDD-2[\p@int,-5]%
    \figpttraDD\the\s@mme:=\p@int/2,-2/\advance\s@mme\@ne}%
    \resetc@ntr@l\et@tfigptssymDD}\ignorespaces\fi}
\ctr@ld@f\def\figptssymTD#1=#2/#3,#4/{\ifps@cri{\s@uvc@ntr@l\et@tfigptssymTD%
    \setc@ntr@l{2}\vecunit@TD{-2}{#4}\def\list@num{#2}\s@mme=#1%
    \@ecfor\p@int:=\list@num\do{\figvectPTD-1[\p@int,#3]%
    \c@lproscalTD\v@leur[-1,-2]\v@leur=2\v@leur\edef\v@lcoef{\repdecn@mb{\v@leur}}%
    \figpttraTD\the\s@mme:=\p@int/\v@lcoef,-2/\advance\s@mme\@ne}%
    \resetc@ntr@l\et@tfigptssymTD}\ignorespaces\fi}
\ctr@ln@m\figptstra
\ctr@ld@f\def\figptstraDD#1=#2/#3,#4/{\ifps@cri{\Figg@tXYa{#4}\v@lXa=#3\v@lXa\v@lYa=#3\v@lYa%
    \def\list@num{#2}\s@mme=#1\@ecfor\p@int:=\list@num\do{\Figg@tXY{\p@int}%
    \advance\v@lX\v@lXa\advance\v@lY\v@lYa%
    \Figp@intregDD\the\s@mme:(\v@lX,\v@lY)\advance\s@mme\@ne}}\ignorespaces\fi}
\ctr@ld@f\def\figptstraTD#1=#2/#3,#4/{\ifps@cri{\Figg@tXYa{#4}\v@lXa=#3\v@lXa\v@lYa=#3\v@lYa%
    \v@lZa=#3\v@lZa\def\list@num{#2}\s@mme=#1\@ecfor\p@int:=\list@num\do{\Figg@tXY{\p@int}%
    \advance\v@lX\v@lXa\advance\v@lY\v@lYa\advance\v@lZ\v@lZa%
    \Figp@intregTD\the\s@mme:(\v@lX,\v@lY,\v@lZ)\advance\s@mme\@ne}}\ignorespaces\fi}
\ctr@ln@m\figptvisilimSL
\ctr@ld@f\def\figptvisilimSLDD{\un@v@ilable{figptvisilimSL}}
\ctr@ld@f\def\figptvisilimSLTD#1:#2[#3,#4;#5,#6]{\ifps@cri{\s@uvc@ntr@l\et@tfigptvisilimSLTD%
    \setc@ntr@l{2}\figvectP-1[#3,#4]\n@rminf{\delt@}{-1}%
    \ifcase\curr@ntproj\v@lX=\cxa@\p@\v@lY=-\p@\v@lZ=\cxb@\p@
    \Figv@ctCreg-2(\v@lX,\v@lY,\v@lZ)\figvectP-3[#5,#6]\figvectNV-1[-2,-3]%
    \or\figvectP-1[#5,#6]\vecunitCV@TD{-1}\v@lmin=\v@lX\v@lmax=\v@lY
    \v@leur=\v@lZ\v@lX=\cza@\p@\v@lY=\czb@\p@\v@lZ=\czc@\p@\c@lprovec{-1}%
    \or\c@ley@pt{-2}\figvectN-1[#5,#6,-2]\fi
    \edef\Ai@{#3}\edef\Aj@{#4}\figvectP-2[#5,\Ai@]\c@lproscal\v@leur[-1,-2]%
    \ifdim\v@leur>\z@\p@rtent=\@ne\else\p@rtent=\m@ne\fi%
    \figvectP-2[#5,\Aj@]\c@lproscal\v@leur[-1,-2]%
    \ifdim\p@rtent\v@leur>\z@\figptcopy#1:#2/#3/%
    \message{*** \BS@ figptvisilimSL: points are on the same side.}\else%
    \figptcopy-3:/#3/\figptcopy-4:/#4/%
    \loop\figptbary-5:[-3,-4;1,1]\figvectP-2[#5,-5]\c@lproscal\v@leur[-1,-2]%
    \ifdim\p@rtent\v@leur>\z@\figptcopy-3:/-5/\else\figptcopy-4:/-5/\fi%
    \divide\delt@\tw@\ifdim\delt@>\epsil@n\repeat%
    \figptbary#1:#2[-3,-4;1,1]\fi\resetc@ntr@l\et@tfigptvisilimSLTD}\ignorespaces\fi}
\ctr@ld@f\def\c@ley@pt#1{\t@stp@r\ifitis@K\v@lX=\cza@\p@\v@lY=\czb@\p@\v@lZ=\czc@\p@%
    \Figv@ctCreg-1(\v@lX,\v@lY,\v@lZ)\Figp@intreg-2:(\wd\Bt@rget,\ht\Bt@rget,\dp\Bt@rget)%
    \figpttra#1:=-2/-\disob@intern,-1/\else\end\fi}
\ctr@ld@f\def\t@stp@r{\itis@Ktrue\ifnewt@rgetpt\else\itis@Kfalse%
    \message{*** \BS@ figptvisilimXX: target point undefined.}\fi\ifnewdis@b\else%
    \itis@Kfalse\message{*** \BS@ figptvisilimXX: observation distance undefined.}\fi%
    \ifitis@K\else\message{*** This macro must be called after \BS@ psbeginfig or after
    having set the missing parameter(s) with \BS@ figset proj()}\fi}
\ctr@ld@f\def\figscan#1(#2,#3){{\s@uvc@ntr@l\et@tfigscan\@psfgetbb{#1}\if@psfbbfound\else%
    \def\@psfllx{0}\def\@psflly{20}\def\@psfurx{540}\def\@psfury{640}\fi\figscan@{#2}{#3}%
    \resetc@ntr@l\et@tfigscan}\ignorespaces}
\ctr@ld@f\def\figscan@#1#2{%
    \unit@=\@ne bp\setc@ntr@l{2}\figsetmark{}%
    \def\minst@p{20pt}%
    \v@lX=\@psfllx\p@\v@lX=\Sc@leFact\v@lX\r@undint\v@lX\v@lX%
    \v@lY=\@psflly\p@\v@lY=\Sc@leFact\v@lY\ifdim\v@lY>\z@\r@undint\v@lY\v@lY\fi%
    \delt@=\@psfury\p@\delt@=\Sc@leFact\delt@%
    \advance\delt@-\v@lY\v@lXa=\@psfurx\p@\v@lXa=\Sc@leFact\v@lXa\v@leur=\minst@p%
    \edef\valv@lY{\repdecn@mb{\v@lY}}\edef\LgTr@it{\the\delt@}%
    \loop\ifdim\v@lX<\v@lXa\edef\valv@lX{\repdecn@mb{\v@lX}}%
    \figptDD -1:(\valv@lX,\valv@lY)\figwriten -1:\hbox{\vrule height\LgTr@it}(0)%
    \ifdim\v@leur<\minst@p\else\figsetmark{\raise-8bp\hbox{$\scriptscriptstyle\triangle$}}%
    \figwrites -1:\@ffichnb{0}{\valv@lX}(6)\v@leur=\z@\figsetmark{}\fi%
    \advance\v@leur#1pt\advance\v@lX#1pt\repeat%
    \def\minst@p{10pt}%
    \v@lX=\@psfllx\p@\v@lX=\Sc@leFact\v@lX\ifdim\v@lX>\z@\r@undint\v@lX\v@lX\fi%
    \v@lY=\@psflly\p@\v@lY=\Sc@leFact\v@lY\r@undint\v@lY\v@lY%
    \delt@=\@psfurx\p@\delt@=\Sc@leFact\delt@%
    \advance\delt@-\v@lX\v@lYa=\@psfury\p@\v@lYa=\Sc@leFact\v@lYa\v@leur=\minst@p%
    \edef\valv@lX{\repdecn@mb{\v@lX}}\edef\LgTr@it{\the\delt@}%
    \loop\ifdim\v@lY<\v@lYa\edef\valv@lY{\repdecn@mb{\v@lY}}%
    \figptDD -1:(\valv@lX,\valv@lY)\figwritee -1:\vbox{\hrule width\LgTr@it}(0)%
    \ifdim\v@leur<\minst@p\else\figsetmark{$\triangleright$\kern4bp}%
    \figwritew -1:\@ffichnb{0}{\valv@lY}(6)\v@leur=\z@\figsetmark{}\fi%
    \advance\v@leur#2pt\advance\v@lY#2pt\repeat}
\ctr@ld@f\let\figscanI=\figscan
\ctr@ld@f\def\figscan@E#1(#2,#3){{\s@uvc@ntr@l\et@tfigscan@E%
    \Figdisc@rdLTS{#1}{\t@xt@}\pdfximage{\t@xt@}%
    \setbox\Gb@x=\hbox{\pdfrefximage\pdflastximage}%
    \edef\@psfllx{0}\v@lY=-\dp\Gb@x\edef\@psflly{\repdecn@mb{\v@lY}}%
    \edef\@psfurx{\repdecn@mb{\wd\Gb@x}}%
    \v@lY=\dp\Gb@x\advance\v@lY\ht\Gb@x\edef\@psfury{\repdecn@mb{\v@lY}}%
    \figscan@{#2}{#3}\resetc@ntr@l\et@tfigscan@E}\ignorespaces}
\ctr@ld@f\def\figshowpts[#1,#2]{{\figsetmark{$\bullet$}\figsetptname{\bf ##1}%
    \p@rtent=#2\relax\ifnum\p@rtent<\z@\p@rtent=\z@\fi%
    \s@mme=#1\relax\ifnum\s@mme<\z@\s@mme=\z@\fi%
    \loop\ifnum\s@mme<\p@rtent\pt@rvect{\s@mme}%
    \ifitis@K\figwriten{\the\s@mme}:(4pt)\fi\advance\s@mme\@ne\repeat%
    \pt@rvect{\s@mme}\ifitis@K\figwriten{\the\s@mme}:(4pt)\fi}\ignorespaces}
\ctr@ld@f\def\pt@rvect#1{\set@bjc@de{#1}%
    \expandafter\expandafter\expandafter\inqpt@rvec\csname\objc@de\endcsname:}
\ctr@ld@f\def\inqpt@rvec#1#2:{\if#1\C@dCl@spt\itis@Ktrue\else\itis@Kfalse\fi}
\ctr@ld@f\def\figshowsettings{{%
    \immediate\write16{====================================================================}%
    \immediate\write16{ Current settings about:}%
    \immediate\write16{ --- GENERAL ---}%
    \immediate\write16{Scale factor and Unit = \unit@util\space (\the\unit@)
     \space -> \BS@ figinit{ScaleFactorUnit}}%
    \immediate\write16{Update mode = \ifpsupdatem@de yes\else no\fi
     \space-> \BS@ psset(update=yes/no) or \BS@ pssetdefault(update=yes/no)}%
    \immediate\write16{ --- PRINTING ---}%
    \immediate\write16{Implicit point name = \ptn@me{i} \space-> \BS@ figsetptname{Name}}%
    \immediate\write16{Point marker = \the\c@nsymb \space -> \BS@ figsetmark{Mark}}%
    \immediate\write16{Print rounded coordinates = \ifr@undcoord yes\else no\fi
     \space-> \BS@ figsetroundcoord{yes/no}}%
    \immediate\write16{ --- GRAPHICAL (general) ---}%
    \immediate\write16{First-level (or primary) settings:}%
    \immediate\write16{ Color = \curr@ntcolor \space-> \BS@ psset(color=ColorDefinition)}%
    \immediate\write16{ Filling mode = \iffillm@de yes\else no\fi
     \space-> \BS@ psset(fillmode=yes/no)}%
    \immediate\write16{ Line join = \curr@ntjoin \space-> \BS@ psset(join=miter/round/bevel)}%
    \immediate\write16{ Line style = \curr@ntdash \space-> \BS@ psset(dash=Index/Pattern)}%
    \immediate\write16{ Line width = \curr@ntwidth
     \space-> \BS@ psset(width=real in PostScript units)}%
    \immediate\write16{Second-level (or secondary) settings:}%
    \immediate\write16{ Color = \sec@ndcolor \space-> \BS@ psset second(color=ColorDefinition)}%
    \immediate\write16{ Line style = \curr@ntseconddash
     \space-> \BS@ psset second(dash=Index/Pattern)}%
    \immediate\write16{ Line width = \curr@ntsecondwidth
     \space-> \BS@ psset second(width=real in PostScript units)}%
    \immediate\write16{Third-level (or ternary) settings:}%
    \immediate\write16{ Color = \th@rdcolor \space-> \BS@ psset third(color=ColorDefinition)}%
    \immediate\write16{ --- GRAPHICAL (specific) ---}%
    \immediate\write16{Arrow-head:}%
    \immediate\write16{ (half-)Angle = \@rrowheadangle
     \space-> \BS@ psset arrowhead(angle=real in degrees)}%
    \immediate\write16{ Filling mode = \if@rrowhfill yes\else no\fi
     \space-> \BS@ psset arrowhead(fillmode=yes/no)}%
    \immediate\write16{ "Outside" = \if@rrowhout yes\else no\fi
     \space-> \BS@ psset arrowhead(out=yes/no)}%
    \immediate\write16{ Length = \@rrowheadlength
     \if@rrowratio\space(not active)\else\space(active)\fi
     \space-> \BS@ psset arrowhead(length=real in user coord.)}%
    \immediate\write16{ Ratio = \@rrowheadratio
     \if@rrowratio\space(active)\else\space(not active)\fi
     \space-> \BS@ psset arrowhead(ratio=real in [0,1])}%
    \immediate\write16{Curve: Roundness = \curv@roundness
     \space-> \BS@ psset curve(roundness=real in [0,0.5])}%
    \immediate\write16{Mesh: Diagonal = \c@ntrolmesh
     \space-> \BS@ psset mesh(diag=integer in {-1,0,1})}%
    \immediate\write16{Flow chart:}%
    \immediate\write16{ Arrow position = \@rrowp@s
     \space-> \BS@ psset flowchart(arrowposition=real in [0,1])}%
    \immediate\write16{ Arrow reference point = \ifcase\@rrowr@fpt start\else end\fi
     \space-> \BS@ psset flowchart(arrowrefpt = start/end)}%
    \immediate\write16{ Line type = \ifcase\fclin@typ@ curve\else polygon\fi
     \space-> \BS@ psset flowchart(line=polygon/curve)}%
    \immediate\write16{ Padding = (\Xp@dd, \Yp@dd)
     \space-> \BS@ psset flowchart(padding = real in user coord.)}%
    \immediate\write16{\space\space\space\space(or
     \BS@ psset flowchart(xpadding=real, ypadding=real) )}%
    \immediate\write16{ Radius = \fclin@r@d
     \space-> \BS@ psset flowchart(radius=positive real in user coord.)}%
    \immediate\write16{ Shape = \fcsh@pe
     \space-> \BS@ psset flowchart(shape = rectangle, ellipse or lozenge)}%
    \immediate\write16{ Thickness = \thickn@ss
     \space-> \BS@ psset flowchart(thickness = real in user coord.)}%
    \ifTr@isDim%
    \immediate\write16{ --- 3D to 2D PROJECTION ---}%
    \immediate\write16{Projection : \typ@proj \space-> \BS@ figinit{ScaleFactorUnit, ProjType}}%
    \immediate\write16{Longitude (psi) = \v@lPsi \space-> \BS@ figset proj(psi=real in degrees)}%
    \ifcase\curr@ntproj\immediate\write16{Depth coeff. (Lambda)
     \space = \v@lTheta \space-> \BS@ figset proj(lambda=real in [0,1])}%
    \else\immediate\write16{Latitude (theta)
     \space = \v@lTheta \space-> \BS@ figset proj(theta=real in degrees)}%
    \fi%
    \ifnum\curr@ntproj=\tw@%
    \immediate\write16{Observation distance = \disob@unit
     \space-> \BS@ figset proj(dist=real in user coord.)}%
    \immediate\write16{Target point = \t@rgetpt \space-> \BS@ figset proj(targetpt=pt number)}%
     \v@lX=\ptT@unit@\wd\Bt@rget\v@lY=\ptT@unit@\ht\Bt@rget\v@lZ=\ptT@unit@\dp\Bt@rget%
    \immediate\write16{ Its coordinates are
     (\repdecn@mb{\v@lX}, \repdecn@mb{\v@lY}, \repdecn@mb{\v@lZ})}%
    \fi%
    \fi%
    \immediate\write16{====================================================================}%
    \ignorespaces}}
\ctr@ln@w{newif}\ifitis@vect@r
\ctr@ld@f\def\figvectC#1(#2,#3){{\itis@vect@rtrue\figpt#1:(#2,#3)}\ignorespaces}
\ctr@ld@f\def\Figv@ctCreg#1(#2,#3){{\itis@vect@rtrue\Figp@intreg#1:(#2,#3)}\ignorespaces}
\ctr@ln@m\figvectDBezier
\ctr@ld@f\def\figvectDBezierDD#1:#2,#3[#4,#5,#6,#7]{\ifps@cri{\s@uvc@ntr@l\et@tfigvectDBezierDD%
    \FigvectDBezier@#2,#3[#4,#5,#6,#7]\v@lX=\c@ef\v@lX\v@lY=\c@ef\v@lY%
    \Figv@ctCreg#1(\v@lX,\v@lY)\resetc@ntr@l\et@tfigvectDBezierDD}\ignorespaces\fi}
\ctr@ld@f\def\figvectDBezierTD#1:#2,#3[#4,#5,#6,#7]{\ifps@cri{\s@uvc@ntr@l\et@tfigvectDBezierTD%
    \FigvectDBezier@#2,#3[#4,#5,#6,#7]\v@lX=\c@ef\v@lX\v@lY=\c@ef\v@lY\v@lZ=\c@ef\v@lZ%
    \Figv@ctCreg#1(\v@lX,\v@lY,\v@lZ)\resetc@ntr@l\et@tfigvectDBezierTD}\ignorespaces\fi}
\ctr@ld@f\def\FigvectDBezier@#1,#2[#3,#4,#5,#6]{\setc@ntr@l{2}%
    \edef\T@{#2}\v@leur=\p@\advance\v@leur-#2pt\edef\UNmT@{\repdecn@mb{\v@leur}}%
    \ifnum#1=\tw@\def\c@ef{6}\else\def\c@ef{3}\fi%
    \figptcopy-4:/#3/\figptcopy-3:/#4/\figptcopy-2:/#5/\figptcopy-1:/#6/%
    \l@mbd@un=-4 \l@mbd@de=-\thr@@\p@rtent=\m@ne\c@lDecast%
    \ifnum#1=\tw@\c@lDCDeux{-4}{-3}\c@lDCDeux{-3}{-2}\c@lDCDeux{-4}{-3}\else%
    \l@mbd@un=-4 \l@mbd@de=-\thr@@\p@rtent=-\tw@\c@lDecast%
    \c@lDCDeux{-4}{-3}\fi\Figg@tXY{-4}}
\ctr@ln@m\c@lDCDeux
\ctr@ld@f\def\c@lDCDeuxDD#1#2{\Figg@tXY{#2}\Figg@tXYa{#1}%
    \advance\v@lX-\v@lXa\advance\v@lY-\v@lYa\Figp@intregDD#1:(\v@lX,\v@lY)}
\ctr@ld@f\def\c@lDCDeuxTD#1#2{\Figg@tXY{#2}\Figg@tXYa{#1}\advance\v@lX-\v@lXa%
    \advance\v@lY-\v@lYa\advance\v@lZ-\v@lZa\Figp@intregTD#1:(\v@lX,\v@lY,\v@lZ)}
\ctr@ln@m\figvectN
\ctr@ld@f\def\figvectNDD#1[#2,#3]{\ifps@cri{\Figg@tXYa{#2}\Figg@tXY{#3}%
    \advance\v@lX-\v@lXa\advance\v@lY-\v@lYa%
    \Figv@ctCreg#1(-\v@lY,\v@lX)}\ignorespaces\fi}
\ctr@ld@f\def\figvectNTD#1[#2,#3,#4]{\ifps@cri{\vecunitC@TD[#2,#4]\v@lmin=\v@lX\v@lmax=\v@lY%
    \v@leur=\v@lZ\vecunitC@TD[#2,#3]\c@lprovec{#1}}\ignorespaces\fi}
\ctr@ln@m\figvectNV
\ctr@ld@f\def\figvectNVDD#1[#2]{\ifps@cri{\Figg@tXY{#2}\Figv@ctCreg#1(-\v@lY,\v@lX)}\ignorespaces\fi}
\ctr@ld@f\def\figvectNVTD#1[#2,#3]{\ifps@cri{\vecunitCV@TD{#3}\v@lmin=\v@lX\v@lmax=\v@lY%
    \v@leur=\v@lZ\vecunitCV@TD{#2}\c@lprovec{#1}}\ignorespaces\fi}
\ctr@ln@m\figvectP
\ctr@ld@f\def\figvectPDD#1[#2,#3]{\ifps@cri{\Figg@tXYa{#2}\Figg@tXY{#3}%
    \advance\v@lX-\v@lXa\advance\v@lY-\v@lYa%
    \Figv@ctCreg#1(\v@lX,\v@lY)}\ignorespaces\fi}
\ctr@ld@f\def\figvectPTD#1[#2,#3]{\ifps@cri{\Figg@tXYa{#2}\Figg@tXY{#3}%
    \advance\v@lX-\v@lXa\advance\v@lY-\v@lYa\advance\v@lZ-\v@lZa%
    \Figv@ctCreg#1(\v@lX,\v@lY,\v@lZ)}\ignorespaces\fi}
\ctr@ln@m\figvectU
\ctr@ld@f\def\figvectUDD#1[#2]{\ifps@cri{\n@rmeuc\v@leur{#2}\invers@\v@leur\v@leur%
    \delt@=\repdecn@mb{\v@leur}\unit@\edef\v@ldelt@{\repdecn@mb{\delt@}}%
    \Figg@tXY{#2}\v@lX=\v@ldelt@\v@lX\v@lY=\v@ldelt@\v@lY%
    \Figv@ctCreg#1(\v@lX,\v@lY)}\ignorespaces\fi}
\ctr@ld@f\def\figvectUTD#1[#2]{\ifps@cri{\n@rmeuc\v@leur{#2}\invers@\v@leur\v@leur%
    \delt@=\repdecn@mb{\v@leur}\unit@\edef\v@ldelt@{\repdecn@mb{\delt@}}%
    \Figg@tXY{#2}\v@lX=\v@ldelt@\v@lX\v@lY=\v@ldelt@\v@lY\v@lZ=\v@ldelt@\v@lZ%
    \Figv@ctCreg#1(\v@lX,\v@lY,\v@lZ)}\ignorespaces\fi}
\ctr@ld@f\def\figvisu#1#2#3{\c@ldefproj\initb@undb@x\xdef\figforTeXFigno{\figforTeXnextFigno}%
    \s@mme=\figforTeXnextFigno\advance\s@mme\@ne\xdef\figforTeXnextFigno{\number\s@mme}%
    \setbox\b@xvisu=\hbox{\ifnum\@utoFN>\z@\figinsert{}\gdef\@utoFInDone{0}\fi\ignorespaces#3}%
    \gdef\@utoFInDone{1}\gdef\@utoFN{0}%
    \v@lXa=-\c@@rdYmin\v@lYa=\c@@rdYmax\advance\v@lYa-\c@@rdYmin%
    \v@lX=\c@@rdXmax\advance\v@lX-\c@@rdXmin%
    \setbox#1=\hbox{#2}\v@lY=-\v@lX\maxim@m{\v@lX}{\v@lX}{\wd#1}%
    \advance\v@lY\v@lX\divide\v@lY\tw@\advance\v@lY-\c@@rdXmin%
    \setbox#1=\vbox{\parindent0mm\hsize=\v@lX\vskip\v@lYa%
    \rlap{\hskip\v@lY\smash{\raise\v@lXa\box\b@xvisu}}%
    \def\t@xt@{#2}\ifx\t@xt@\empty\else\medskip\centerline{#2}\fi}\wd#1=\v@lX}
\ctr@ld@f\def\figDecrementFigno{{\xdef\figforTeXnextFigno{\figforTeXFigno}%
    \s@mme=\figforTeXFigno\advance\s@mme\m@ne\xdef\figforTeXFigno{\number\s@mme}}}
\ctr@ln@w{newbox}\Bt@rget\setbox\Bt@rget=\null
\ctr@ln@w{newbox}\BminTD@\setbox\BminTD@=\null
\ctr@ln@w{newbox}\BmaxTD@\setbox\BmaxTD@=\null
\ctr@ln@w{newif}\ifnewt@rgetpt\ctr@ln@w{newif}\ifnewdis@b
\ctr@ld@f\def\b@undb@xTD#1#2#3{%
    \relax\ifdim#1<\wd\BminTD@\global\wd\BminTD@=#1\fi%
    \relax\ifdim#2<\ht\BminTD@\global\ht\BminTD@=#2\fi%
    \relax\ifdim#3<\dp\BminTD@\global\dp\BminTD@=#3\fi%
    \relax\ifdim#1>\wd\BmaxTD@\global\wd\BmaxTD@=#1\fi%
    \relax\ifdim#2>\ht\BmaxTD@\global\ht\BmaxTD@=#2\fi%
    \relax\ifdim#3>\dp\BmaxTD@\global\dp\BmaxTD@=#3\fi}
\ctr@ld@f\def\c@ldefdisob{{\ifdim\wd\BminTD@<\maxdimen\v@leur=\wd\BmaxTD@\advance\v@leur-\wd\BminTD@%
    \delt@=\ht\BmaxTD@\advance\delt@-\ht\BminTD@\maxim@m{\v@leur}{\v@leur}{\delt@}%
    \delt@=\dp\BmaxTD@\advance\delt@-\dp\BminTD@\maxim@m{\v@leur}{\v@leur}{\delt@}%
    \v@leur=5\v@leur\else\v@leur=800pt\fi\c@ldefdisob@{\v@leur}}}
\ctr@ln@m\disob@intern
\ctr@ln@m\disob@
\ctr@ln@m\divf@ctproj
\ctr@ld@f\def\c@ldefdisob@#1{{\v@leur=#1\ifdim\v@leur<\p@\v@leur=800pt\fi%
    \xdef\disob@intern{\repdecn@mb{\v@leur}}%
    \delt@=\ptT@unit@\v@leur\xdef\disob@unit{\repdecn@mb{\delt@}}%
    \f@ctech=\@ne\loop\ifdim\v@leur>\t@n pt\divide\v@leur\t@n\multiply\f@ctech\t@n\repeat%
    \xdef\disob@{\repdecn@mb{\v@leur}}\xdef\divf@ctproj{\the\f@ctech}}%
    \global\newdis@btrue}
\ctr@ln@m\t@rgetpt
\ctr@ld@f\def\c@ldeft@rgetpt{\newt@rgetpttrue\def\t@rgetpt{CenterBoundBox}{%
    \delt@=\wd\BmaxTD@\advance\delt@-\wd\BminTD@\divide\delt@\tw@%
    \v@leur=\wd\BminTD@\advance\v@leur\delt@\global\wd\Bt@rget=\v@leur%
    \delt@=\ht\BmaxTD@\advance\delt@-\ht\BminTD@\divide\delt@\tw@%
    \v@leur=\ht\BminTD@\advance\v@leur\delt@\global\ht\Bt@rget=\v@leur%
    \delt@=\dp\BmaxTD@\advance\delt@-\dp\BminTD@\divide\delt@\tw@%
    \v@leur=\dp\BminTD@\advance\v@leur\delt@\global\dp\Bt@rget=\v@leur}}
\ctr@ln@m\c@ldefproj
\ctr@ld@f\def\c@ldefprojTD{\ifnewt@rgetpt\else\c@ldeft@rgetpt\fi\ifnewdis@b\else\c@ldefdisob\fi}
\ctr@ld@f\def\c@lprojcav{
    \v@lZa=\cxa@\v@lY\advance\v@lX\v@lZa%
    \v@lZa=\cxb@\v@lY\v@lY=\v@lZ\advance\v@lY\v@lZa\ignorespaces}
\ctr@ln@m\v@lcoef
\ctr@ld@f\def\c@lprojrea{
    \advance\v@lX-\wd\Bt@rget\advance\v@lY-\ht\Bt@rget\advance\v@lZ-\dp\Bt@rget%
    \v@lZa=\cza@\v@lX\advance\v@lZa\czb@\v@lY\advance\v@lZa\czc@\v@lZ%
    \divide\v@lZa\divf@ctproj\advance\v@lZa\disob@ pt\invers@{\v@lZa}{\v@lZa}%
    \v@lZa=\disob@\v@lZa\edef\v@lcoef{\repdecn@mb{\v@lZa}}%
    \v@lXa=\cxa@\v@lX\advance\v@lXa\cxb@\v@lY\v@lXa=\v@lcoef\v@lXa%
    \v@lY=\cyb@\v@lY\advance\v@lY\cya@\v@lX\advance\v@lY\cyc@\v@lZ%
    \v@lY=\v@lcoef\v@lY\v@lX=\v@lXa\ignorespaces}
\ctr@ld@f\def\c@lprojort{
    \v@lXa=\cxa@\v@lX\advance\v@lXa\cxb@\v@lY%
    \v@lY=\cyb@\v@lY\advance\v@lY\cya@\v@lX\advance\v@lY\cyc@\v@lZ%
    \v@lX=\v@lXa\ignorespaces}
\ctr@ld@f\def\Figptpr@j#1:#2/#3/{{\Figg@tXY{#3}\superc@lprojSP%
    \Figp@intregDD#1:{#2}(\v@lX,\v@lY)}\ignorespaces}
\ctr@ln@m\figsetobdist
\ctr@ld@f\def\figsetobdistDD{\un@v@ilable{figsetobdist}}
\ctr@ld@f\def\figsetobdistTD(#1){{\ifcurr@ntPS%
    \immediate\write16{*** \BS@ figsetobdist is ignored inside a
     \BS@ psbeginfig-\BS@ psendfig block.}%
    \else\v@leur=#1\unit@\c@ldefdisob@{\v@leur}\fi}\ignorespaces}
\ctr@ln@m\c@lprojSP
\ctr@ln@m\curr@ntproj
\ctr@ln@m\typ@proj
\ctr@ln@m\superc@lprojSP
\ctr@ld@f\def\Figs@tproj#1{%
    \if#13 \d@faultproj\else\if#1c\d@faultproj%
    \else\if#1o\xdef\curr@ntproj{1}\xdef\typ@proj{orthogonal}%
         \figsetviewTD(\def@ultpsi,\def@ulttheta)%
         \global\let\c@lprojSP=\c@lprojort\global\let\superc@lprojSP=\c@lprojort%
    \else\if#1r\xdef\curr@ntproj{2}\xdef\typ@proj{realistic}%
         \figsetviewTD(\def@ultpsi,\def@ulttheta)%
         \global\let\c@lprojSP=\c@lprojrea\global\let\superc@lprojSP=\c@lprojrea%
    \else\d@faultproj\message{*** Unknown projection. Cavalier projection assumed.}%
    \fi\fi\fi\fi}
\ctr@ld@f\def\d@faultproj{\xdef\curr@ntproj{0}\xdef\typ@proj{cavalier}\figsetviewTD(\def@ultpsi,0.5)%
         \global\let\c@lprojSP=\c@lprojcav\global\let\superc@lprojSP=\c@lprojcav}
\ctr@ln@m\figsettarget
\ctr@ld@f\def\figsettargetDD{\un@v@ilable{figsettarget}}
\ctr@ld@f\def\figsettargetTD[#1]{{\ifcurr@ntPS%
    \immediate\write16{*** \BS@ figsettarget is ignored inside a
     \BS@ psbeginfig-\BS@ psendfig block.}%
    \else\global\newt@rgetpttrue\xdef\t@rgetpt{#1}\Figg@tXY{#1}\global\wd\Bt@rget=\v@lX%
    \global\ht\Bt@rget=\v@lY\global\dp\Bt@rget=\v@lZ\fi}\ignorespaces}
\ctr@ln@m\figsetview
\ctr@ld@f\def\figsetviewDD{\un@v@ilable{figsetview}}
\ctr@ld@f\def\figsetviewTD(#1){\ifcurr@ntPS%
     \immediate\write16{*** \BS@ figsetview is ignored inside a
     \BS@ psbeginfig-\BS@ psendfig block.}\else\Figsetview@#1,:\fi\ignorespaces}
\ctr@ld@f\def\Figsetview@#1,#2:{{\xdef\v@lPsi{#1}\def\t@xt@{#2}%
    \ifx\t@xt@\empty\def\@rgdeux{\v@lTheta}\else\X@rgdeux@#2\fi%
    \c@ssin{\costhet@}{\sinthet@}{#1}\v@lmin=\costhet@ pt\v@lmax=\sinthet@ pt%
    \ifcase\curr@ntproj%
    \v@leur=\@rgdeux\v@lmin\xdef\cxa@{\repdecn@mb{\v@leur}}%
    \v@leur=\@rgdeux\v@lmax\xdef\cxb@{\repdecn@mb{\v@leur}}\v@leur=\@rgdeux pt%
    \relax\ifdim\v@leur>\p@\message{*** Lambda too large ! See \BS@ figset proj() !}\fi%
    \else%
    \v@lmax=-\v@lmax\xdef\cxa@{\repdecn@mb{\v@lmax}}\xdef\cxb@{\costhet@}%
    \ifx\t@xt@\empty\edef\@rgdeux{\def@ulttheta}\fi\c@ssin{\C@}{\S@}{\@rgdeux}%
    \v@lmax=-\S@ pt%
    \v@leur=\v@lmax\v@leur=\costhet@\v@leur\xdef\cya@{\repdecn@mb{\v@leur}}%
    \v@leur=\v@lmax\v@leur=\sinthet@\v@leur\xdef\cyb@{\repdecn@mb{\v@leur}}%
    \xdef\cyc@{\C@}\v@lmin=-\C@ pt%
    \v@leur=\v@lmin\v@leur=\costhet@\v@leur\xdef\cza@{\repdecn@mb{\v@leur}}%
    \v@leur=\v@lmin\v@leur=\sinthet@\v@leur\xdef\czb@{\repdecn@mb{\v@leur}}%
    \xdef\czc@{\repdecn@mb{\v@lmax}}\fi%
    \xdef\v@lTheta{\@rgdeux}}}
\ctr@ld@f\def\def@ultpsi{40}
\ctr@ld@f\def\def@ulttheta{25}
\ctr@ln@m\l@debut
\ctr@ln@m\n@mref
\ctr@ld@f\def\figset#1(#2){\def\t@xt@{#1}\ifx\t@xt@\empty\trtlis@rg{#2}{\Figsetwr@te}
    \else\keln@mde#1|%
    \def\n@mref{pr}\ifx\l@debut\n@mref\ifcurr@ntPS
     \immediate\write16{*** \BS@ figset proj(...) is ignored inside a
     \BS@ psbeginfig-\BS@ psendfig block.}\else\trtlis@rg{#2}{\Figsetpr@j}\fi\else%
    \def\n@mref{wr}\ifx\l@debut\n@mref\trtlis@rg{#2}{\Figsetwr@te}\else
    \immediate\write16{*** Unknown keyword: \BS@ figset #1(...)}%
    \fi\fi\fi\ignorespaces}
\ctr@ld@f\def\Figsetpr@j#1=#2|{\keln@mtr#1|%
    \def\n@mref{dep}\ifx\l@debut\n@mref\Figsetd@p{#2}\else
    \def\n@mref{dis}\ifx\l@debut\n@mref%
     \ifnum\curr@ntproj=\tw@\figsetobdist(#2)\else\Figset@rr\fi\else
    \def\n@mref{lam}\ifx\l@debut\n@mref\Figsetd@p{#2}\else
    \def\n@mref{lat}\ifx\l@debut\n@mref\Figsetth@{#2}\else
    \def\n@mref{lon}\ifx\l@debut\n@mref\figsetview(#2)\else
    \def\n@mref{psi}\ifx\l@debut\n@mref\figsetview(#2)\else
    \def\n@mref{tar}\ifx\l@debut\n@mref%
     \ifnum\curr@ntproj=\tw@\figsettarget[#2]\else\Figset@rr\fi\else
    \def\n@mref{the}\ifx\l@debut\n@mref\Figsetth@{#2}\else
    \immediate\write16{*** Unknown attribute: \BS@ figset proj(..., #1=...).}%
    \fi\fi\fi\fi\fi\fi\fi\fi}
\ctr@ld@f\def\Figsetd@p#1{\ifnum\curr@ntproj=\z@\figsetview(\v@lPsi,#1)\else\Figset@rr\fi}
\ctr@ld@f\def\Figsetth@#1{\ifnum\curr@ntproj=\z@\Figset@rr\else\figsetview(\v@lPsi,#1)\fi}
\ctr@ld@f\def\Figset@rr{\message{*** \BS@ figset proj(): Attribute "\n@mref" ignored, incompatible
    with current projection}}
\ctr@ld@f\def\initb@undb@xTD{\wd\BminTD@=\maxdimen\ht\BminTD@=\maxdimen\dp\BminTD@=\maxdimen%
    \wd\BmaxTD@=-\maxdimen\ht\BmaxTD@=-\maxdimen\dp\BmaxTD@=-\maxdimen}
\ctr@ln@w{newbox}\Gb@x      
\ctr@ln@w{newbox}\Gb@xSC    
\ctr@ln@w{newtoks}\c@nsymb  
\ctr@ln@w{newif}\ifr@undcoord\ctr@ln@w{newif}\ifunitpr@sent
\ctr@ld@f\def\unssqrttw@{0.707106 }
\ctr@ld@f\def\figAst{\raise-1.15ex\hbox{$\ast$}}
\ctr@ld@f\def\figBullet{\raise-1.15ex\hbox{$\bullet$}}
\ctr@ld@f\def\figCirc{\raise-1.15ex\hbox{$\circ$}}
\ctr@ld@f\def\figDiamond{\raise-1.15ex\hbox{$\diamond$}}%
\ctr@ld@f\def\boxit#1#2{\leavevmode\hbox{\vrule\vbox{\hrule\vglue#1%
    \vtop{\hbox{\kern#1{#2}\kern#1}\vglue#1\hrule}}\vrule}}
\ctr@ld@f\def\centertext#1#2{\vbox{\hsize#1\parindent0cm%
    \leftskip=0pt plus 1fil\rightskip=0pt plus 1fil\parfillskip=0pt{#2}}}
\ctr@ld@f\def\lefttext#1#2{\vbox{\hsize#1\parindent0cm\rightskip=0pt plus 1fil#2}}
\ctr@ld@f\def\c@nterpt{\ignorespaces%
    \kern-.5\wd\Gb@xSC%
    \raise-.5\ht\Gb@xSC\rlap{\hbox{\raise.5\dp\Gb@xSC\hbox{\copy\Gb@xSC}}}%
    \kern .5\wd\Gb@xSC\ignorespaces}
\ctr@ld@f\def\b@undb@xSC#1#2{{\v@lXa=#1\v@lYa=#2%
    \v@leur=\ht\Gb@xSC\advance\v@leur\dp\Gb@xSC%
    \advance\v@lXa-.5\wd\Gb@xSC\advance\v@lYa-.5\v@leur\b@undb@x{\v@lXa}{\v@lYa}%
    \advance\v@lXa\wd\Gb@xSC\advance\v@lYa\v@leur\b@undb@x{\v@lXa}{\v@lYa}}}
\ctr@ln@m\Dist@n
\ctr@ln@m\l@suite
\ctr@ld@f\def\@keldist#1#2{\edef\Dist@n{#2}\y@tiunit{\Dist@n}%
    \ifunitpr@sent#1=\Dist@n\else#1=\Dist@n\unit@\fi}
\ctr@ld@f\def\y@tiunit#1{\unitpr@sentfalse\expandafter\y@tiunit@#1:}
\ctr@ld@f\def\y@tiunit@#1#2:{\ifcat#1a\unitpr@senttrue\else\def\l@suite{#2}%
    \ifx\l@suite\empty\else\y@tiunit@#2:\fi\fi}
\ctr@ln@m\figcoord
\ctr@ld@f\def\figcoordDD#1{{\v@lX=\ptT@unit@\v@lX\v@lY=\ptT@unit@\v@lY%
    \ifr@undcoord\ifcase#1\v@leur=0.5pt\or\v@leur=0.05pt\or\v@leur=0.005pt%
    \or\v@leur=0.0005pt\else\v@leur=\z@\fi%
    \ifdim\v@lX<\z@\advance\v@lX-\v@leur\else\advance\v@lX\v@leur\fi%
    \ifdim\v@lY<\z@\advance\v@lY-\v@leur\else\advance\v@lY\v@leur\fi\fi%
    (\@ffichnb{#1}{\repdecn@mb{\v@lX}},\ifmmode\else\thinspace\fi%
    \@ffichnb{#1}{\repdecn@mb{\v@lY}})}}
\ctr@ld@f\def\@ffichnb#1#2{{\def\@@ffich{\@ffich#1(}\edef\n@mbre{#2}%
    \expandafter\@@ffich\n@mbre)}}
\ctr@ld@f\def\@ffich#1(#2.#3){{#2\ifnum#1>\z@.\fi\def\dig@ts{#3}\s@mme=\z@%
    \loop\ifnum\s@mme<#1\expandafter\@ffichdec\dig@ts:\advance\s@mme\@ne\repeat}}
\ctr@ld@f\def\@ffichdec#1#2:{\relax#1\def\dig@ts{#20}}
\ctr@ld@f\def\figcoordTD#1{{\v@lX=\ptT@unit@\v@lX\v@lY=\ptT@unit@\v@lY\v@lZ=\ptT@unit@\v@lZ%
    \ifr@undcoord\ifcase#1\v@leur=0.5pt\or\v@leur=0.05pt\or\v@leur=0.005pt%
    \or\v@leur=0.0005pt\else\v@leur=\z@\fi%
    \ifdim\v@lX<\z@\advance\v@lX-\v@leur\else\advance\v@lX\v@leur\fi%
    \ifdim\v@lY<\z@\advance\v@lY-\v@leur\else\advance\v@lY\v@leur\fi%
    \ifdim\v@lZ<\z@\advance\v@lZ-\v@leur\else\advance\v@lZ\v@leur\fi\fi%
    (\@ffichnb{#1}{\repdecn@mb{\v@lX}},\ifmmode\else\thinspace\fi%
     \@ffichnb{#1}{\repdecn@mb{\v@lY}},\ifmmode\else\thinspace\fi%
     \@ffichnb{#1}{\repdecn@mb{\v@lZ}})}}
\ctr@ld@f\def\figsetroundcoord#1{\expandafter\Figsetr@undcoord#1:\ignorespaces}
\ctr@ld@f\def\Figsetr@undcoord#1#2:{\if#1n\r@undcoordfalse\else\r@undcoordtrue\fi}
\ctr@ld@f\def\Figsetwr@te#1=#2|{\keln@mun#1|%
    \def\n@mref{m}\ifx\l@debut\n@mref\figsetmark{#2}\else
    \immediate\write16{*** Unknown attribute: \BS@ figset (..., #1=...)}%
    \fi}
\ctr@ld@f\def\figsetmark#1{\c@nsymb={#1}\setbox\Gb@xSC=\hbox{\the\c@nsymb}\ignorespaces}
\ctr@ln@m\ptn@me
\ctr@ld@f\def\figsetptname#1{\def\ptn@me##1{#1}\ignorespaces}
\ctr@ld@f\def\FigWrit@L#1:#2(#3,#4){\ignorespaces\@keldist\v@leur{#3}\@keldist\delt@{#4}%
    \C@rp@r@m\def\list@num{#1}\@ecfor\p@int:=\list@num\do{\FigWrit@pt{\p@int}{#2}}}
\ctr@ld@f\def\FigWrit@pt#1#2{\FigWp@r@m{#1}{#2}\Vc@rrect\figWp@si%
    \ifdim\wd\Gb@xSC>\z@\b@undb@xSC{\v@lX}{\v@lY}\fi\figWBB@x}
\ctr@ld@f\def\FigWp@r@m#1#2{\Figg@tXY{#1}%
    \setbox\Gb@x=\hbox{\def\t@xt@{#2}\ifx\t@xt@\empty\Figg@tT{#1}\else#2\fi}\c@lprojSP}
\ctr@ld@f\let\Vc@rrect=\relax
\ctr@ld@f\let\C@rp@r@m=\relax
\ctr@ld@f\def\figwrite[#1]#2{{\ignorespaces\def\list@num{#1}\@ecfor\p@int:=\list@num\do{%
    \setbox\Gb@x=\hbox{\def\t@xt@{#2}\ifx\t@xt@\empty\Figg@tT{\p@int}\else#2\fi}%
    \Figwrit@{\p@int}}}\ignorespaces}
\ctr@ld@f\def\Figwrit@#1{\Figg@tXY{#1}\c@lprojSP%
    \rlap{\kern\v@lX\raise\v@lY\hbox{\unhcopy\Gb@x}}\v@leur=\v@lY%
    \advance\v@lY\ht\Gb@x\b@undb@x{\v@lX}{\v@lY}\advance\v@lX\wd\Gb@x%
    \v@lY=\v@leur\advance\v@lY-\dp\Gb@x\b@undb@x{\v@lX}{\v@lY}}
\ctr@ld@f\def\figwritec[#1]#2{{\ignorespaces\def\list@num{#1}%
    \@ecfor\p@int:=\list@num\do{\Figwrit@c{\p@int}{#2}}}\ignorespaces}
\ctr@ld@f\def\Figwrit@c#1#2{\FigWp@r@m{#1}{#2}%
    \rlap{\kern\v@lX\raise\v@lY\hbox{\rlap{\kern-.5\wd\Gb@x%
    \raise-.5\ht\Gb@x\hbox{\raise.5\dp\Gb@x\hbox{\unhcopy\Gb@x}}}}}%
    \v@leur=\ht\Gb@x\advance\v@leur\dp\Gb@x%
    \advance\v@lX-.5\wd\Gb@x\advance\v@lY-.5\v@leur\b@undb@x{\v@lX}{\v@lY}%
    \advance\v@lX\wd\Gb@x\advance\v@lY\v@leur\b@undb@x{\v@lX}{\v@lY}}
\ctr@ld@f\def\figwritep[#1]{{\ignorespaces\def\list@num{#1}\setbox\Gb@x=\hbox{\c@nterpt}%
    \@ecfor\p@int:=\list@num\do{\Figwrit@{\p@int}}}\ignorespaces}
\ctr@ld@f\def\figwritew#1:#2(#3){\figwritegcw#1:{#2}(#3,0pt)}
\ctr@ld@f\def\figwritee#1:#2(#3){\figwritegce#1:{#2}(#3,0pt)}
\ctr@ld@f\def\figwriten#1:#2(#3){{\def\Vc@rrect{\v@lZ=\v@leur\advance\v@lZ\dp\Gb@x}%
    \Figwrit@NS#1:{#2}(#3)}\ignorespaces}
\ctr@ld@f\def\figwrites#1:#2(#3){{\def\Vc@rrect{\v@lZ=-\v@leur\advance\v@lZ-\ht\Gb@x}%
    \Figwrit@NS#1:{#2}(#3)}\ignorespaces}
\ctr@ld@f\def\Figwrit@NS#1:#2(#3){\let\figWp@si=\FigWp@siNS\let\figWBB@x=\FigWBB@xNS%
    \FigWrit@L#1:{#2}(#3,0pt)}
\ctr@ld@f\def\FigWp@siNS{\rlap{\kern\v@lX\raise\v@lY\hbox{\rlap{\kern-.5\wd\Gb@x%
    \raise\v@lZ\hbox{\unhcopy\Gb@x}}\c@nterpt}}}
\ctr@ld@f\def\FigWBB@xNS{\advance\v@lY\v@lZ%
    \advance\v@lY-\dp\Gb@x\advance\v@lX-.5\wd\Gb@x\b@undb@x{\v@lX}{\v@lY}%
    \advance\v@lY\ht\Gb@x\advance\v@lY\dp\Gb@x%
    \advance\v@lX\wd\Gb@x\b@undb@x{\v@lX}{\v@lY}}
\ctr@ld@f\def\figwritenw#1:#2(#3){{\let\figWp@si=\FigWp@sigW\let\figWBB@x=\FigWBB@xgWE%
    \def\C@rp@r@m{\v@leur=\unssqrttw@\v@leur\delt@=\v@leur%
    \ifdim\delt@=\z@\delt@=\epsil@n\fi}\let@xte={-}\FigWrit@L#1:{#2}(#3,0pt)}\ignorespaces}
\ctr@ld@f\def\figwritesw#1:#2(#3){{\let\figWp@si=\FigWp@sigW\let\figWBB@x=\FigWBB@xgWE%
    \def\C@rp@r@m{\v@leur=\unssqrttw@\v@leur\delt@=-\v@leur%
    \ifdim\delt@=\z@\delt@=-\epsil@n\fi}\let@xte={-}\FigWrit@L#1:{#2}(#3,0pt)}\ignorespaces}
\ctr@ld@f\def\figwritene#1:#2(#3){{\let\figWp@si=\FigWp@sigE\let\figWBB@x=\FigWBB@xgWE%
    \def\C@rp@r@m{\v@leur=\unssqrttw@\v@leur\delt@=\v@leur%
    \ifdim\delt@=\z@\delt@=\epsil@n\fi}\let@xte={}\FigWrit@L#1:{#2}(#3,0pt)}\ignorespaces}
\ctr@ld@f\def\figwritese#1:#2(#3){{\let\figWp@si=\FigWp@sigE\let\figWBB@x=\FigWBB@xgWE%
    \def\C@rp@r@m{\v@leur=\unssqrttw@\v@leur\delt@=-\v@leur%
    \ifdim\delt@=\z@\delt@=-\epsil@n\fi}\let@xte={}\FigWrit@L#1:{#2}(#3,0pt)}\ignorespaces}
\ctr@ld@f\def\figwritegw#1:#2(#3,#4){{\let\figWp@si=\FigWp@sigW\let\figWBB@x=\FigWBB@xgWE%
    \let@xte={-}\FigWrit@L#1:{#2}(#3,#4)}\ignorespaces}
\ctr@ld@f\def\figwritege#1:#2(#3,#4){{\let\figWp@si=\FigWp@sigE\let\figWBB@x=\FigWBB@xgWE%
    \let@xte={}\FigWrit@L#1:{#2}(#3,#4)}\ignorespaces}
\ctr@ld@f\def\FigWp@sigW{\v@lXa=\z@\v@lYa=\ht\Gb@x\advance\v@lYa\dp\Gb@x%
    \ifdim\delt@>\z@\relax%
    \rlap{\kern\v@lX\raise\v@lY\hbox{\rlap{\kern-\wd\Gb@x\kern-\v@leur%
          \raise\delt@\hbox{\raise\dp\Gb@x\hbox{\unhcopy\Gb@x}}}\c@nterpt}}%
    \else\ifdim\delt@<\z@\relax\v@lYa=-\v@lYa%
    \rlap{\kern\v@lX\raise\v@lY\hbox{\rlap{\kern-\wd\Gb@x\kern-\v@leur%
          \raise\delt@\hbox{\raise-\ht\Gb@x\hbox{\unhcopy\Gb@x}}}\c@nterpt}}%
    \else\v@lXa=-.5\v@lYa%
    \rlap{\kern\v@lX\raise\v@lY\hbox{\rlap{\kern-\wd\Gb@x\kern-\v@leur%
          \raise-.5\ht\Gb@x\hbox{\raise.5\dp\Gb@x\hbox{\unhcopy\Gb@x}}}\c@nterpt}}%
    \fi\fi}
\ctr@ld@f\def\FigWp@sigE{\v@lXa=\z@\v@lYa=\ht\Gb@x\advance\v@lYa\dp\Gb@x%
    \ifdim\delt@>\z@\relax%
    \rlap{\kern\v@lX\raise\v@lY\hbox{\c@nterpt\kern\v@leur%
          \raise\delt@\hbox{\raise\dp\Gb@x\hbox{\unhcopy\Gb@x}}}}%
    \else\ifdim\delt@<\z@\relax\v@lYa=-\v@lYa%
    \rlap{\kern\v@lX\raise\v@lY\hbox{\c@nterpt\kern\v@leur%
          \raise\delt@\hbox{\raise-\ht\Gb@x\hbox{\unhcopy\Gb@x}}}}%
    \else\v@lXa=-.5\v@lYa%
    \rlap{\kern\v@lX\raise\v@lY\hbox{\c@nterpt\kern\v@leur%
          \raise-.5\ht\Gb@x\hbox{\raise.5\dp\Gb@x\hbox{\unhcopy\Gb@x}}}}%
    \fi\fi}
\ctr@ld@f\def\FigWBB@xgWE{\advance\v@lY\delt@%
    \advance\v@lX\the\let@xte\v@leur\advance\v@lY\v@lXa\b@undb@x{\v@lX}{\v@lY}%
    \advance\v@lX\the\let@xte\wd\Gb@x\advance\v@lY\v@lYa\b@undb@x{\v@lX}{\v@lY}}
\ctr@ld@f\def\figwritegcw#1:#2(#3,#4){{\let\figWp@si=\FigWp@sigcW\let\figWBB@x=\FigWBB@xgcWE%
    \let@xte={-}\FigWrit@L#1:{#2}(#3,#4)}\ignorespaces}
\ctr@ld@f\def\figwritegce#1:#2(#3,#4){{\let\figWp@si=\FigWp@sigcE\let\figWBB@x=\FigWBB@xgcWE%
    \let@xte={}\FigWrit@L#1:{#2}(#3,#4)}\ignorespaces}
\ctr@ld@f\def\FigWp@sigcW{\rlap{\kern\v@lX\raise\v@lY\hbox{\rlap{\kern-\wd\Gb@x\kern-\v@leur%
     \raise-.5\ht\Gb@x\hbox{\raise\delt@\hbox{\raise.5\dp\Gb@x\hbox{\unhcopy\Gb@x}}}}%
     \c@nterpt}}}
\ctr@ld@f\def\FigWp@sigcE{\rlap{\kern\v@lX\raise\v@lY\hbox{\c@nterpt\kern\v@leur%
    \raise-.5\ht\Gb@x\hbox{\raise\delt@\hbox{\raise.5\dp\Gb@x\hbox{\unhcopy\Gb@x}}}}}}
\ctr@ld@f\def\FigWBB@xgcWE{\v@lZ=\ht\Gb@x\advance\v@lZ\dp\Gb@x%
    \advance\v@lX\the\let@xte\v@leur\advance\v@lY\delt@\advance\v@lY.5\v@lZ%
    \b@undb@x{\v@lX}{\v@lY}%
    \advance\v@lX\the\let@xte\wd\Gb@x\advance\v@lY-\v@lZ\b@undb@x{\v@lX}{\v@lY}}
\ctr@ld@f\def\figwritebn#1:#2(#3){{\def\Vc@rrect{\v@lZ=\v@leur}\Figwrit@NS#1:{#2}(#3)}\ignorespaces}
\ctr@ld@f\def\figwritebs#1:#2(#3){{\def\Vc@rrect{\v@lZ=-\v@leur}\Figwrit@NS#1:{#2}(#3)}\ignorespaces}
\ctr@ld@f\def\figwritebw#1:#2(#3){{\let\figWp@si=\FigWp@sibW\let\figWBB@x=\FigWBB@xbWE%
    \let@xte={-}\FigWrit@L#1:{#2}(#3,0pt)}\ignorespaces}
\ctr@ld@f\def\figwritebe#1:#2(#3){{\let\figWp@si=\FigWp@sibE\let\figWBB@x=\FigWBB@xbWE%
    \let@xte={}\FigWrit@L#1:{#2}(#3,0pt)}\ignorespaces}
\ctr@ld@f\def\FigWp@sibW{\rlap{\kern\v@lX\raise\v@lY\hbox{\rlap{\kern-\wd\Gb@x\kern-\v@leur%
          \hbox{\unhcopy\Gb@x}}\c@nterpt}}}
\ctr@ld@f\def\FigWp@sibE{\rlap{\kern\v@lX\raise\v@lY\hbox{\c@nterpt\kern\v@leur%
          \hbox{\unhcopy\Gb@x}}}}
\ctr@ld@f\def\FigWBB@xbWE{\v@lZ=\ht\Gb@x\advance\v@lZ\dp\Gb@x%
    \advance\v@lX\the\let@xte\v@leur\advance\v@lY\ht\Gb@x\b@undb@x{\v@lX}{\v@lY}%
    \advance\v@lX\the\let@xte\wd\Gb@x\advance\v@lY-\v@lZ\b@undb@x{\v@lX}{\v@lY}}
\ctr@ln@w{newread}\frf@g  \ctr@ln@w{newwrite}\fwf@g
\ctr@ln@w{newif}\ifcurr@ntPS
\ctr@ln@w{newif}\ifps@cri
\ctr@ln@w{newif}\ifUse@llipse
\ctr@ln@w{newif}\ifpsdebugmode \psdebugmodefalse 
\ctr@ln@w{newif}\ifPDFm@ke
\ifx\pdfliteral\undefined\else\ifnum\pdfoutput>\z@\PDFm@ketrue\fi\fi
\ctr@ld@f\def\initPDF@rDVI{%
\ifPDFm@ke
 \let\figscan=\figscan@E
 \let\newGr@FN=\newGr@FNPDF
 \ctr@ld@f\def\c@mcurveto{c}
 \ctr@ld@f\def\c@mfill{f}
 \ctr@ld@f\def\c@mgsave{q}
 \ctr@ld@f\def\c@mgrestore{Q}
 \ctr@ld@f\def\c@mlineto{l}
 \ctr@ld@f\def\c@mmoveto{m}
 \ctr@ld@f\def\c@msetgray{g}     \ctr@ld@f\def\c@msetgrayStroke{G}
 \ctr@ld@f\def\c@msetcmykcolor{k}\ctr@ld@f\def\c@msetcmykcolorStroke{K}
 \ctr@ld@f\def\c@msetrgbcolor{rg}\ctr@ld@f\def\c@msetrgbcolorStroke{RG}
 \ctr@ld@f\def\d@fprimarC@lor{\curr@ntcolor\space\curr@ntcolorc@md%
               \space\curr@ntcolor\space\curr@ntcolorc@mdStroke}
 \ctr@ld@f\def\d@fsecondC@lor{\sec@ndcolor\space\sec@ndcolorc@md%
               \space\sec@ndcolor\space\sec@ndcolorc@mdStroke}
 \ctr@ld@f\def\d@fthirdC@lor{\th@rdcolor\space\th@rdcolorc@md%
              \space\th@rdcolor\space\th@rdcolorc@mdStroke}
 \ctr@ld@f\def\c@msetdash{d}
 \ctr@ld@f\def\c@msetlinejoin{j}
 \ctr@ld@f\def\c@msetlinewidth{w}
 \ctr@ld@f\def\f@gclosestroke{\immediate\write\fwf@g{s}}
 \ctr@ld@f\def\f@gfill{\immediate\write\fwf@g{\fillc@md}}
 \ctr@ld@f\def\f@gnewpath{}
 \ctr@ld@f\def\f@gstroke{\immediate\write\fwf@g{S}}
\else
 \let\figinsertE=\figinsert
 \let\newGr@FN=\newGr@FNDVI
 \ctr@ld@f\def\c@mcurveto{curveto}
 \ctr@ld@f\def\c@mfill{fill}
 \ctr@ld@f\def\c@mgsave{gsave}
 \ctr@ld@f\def\c@mgrestore{grestore}
 \ctr@ld@f\def\c@mlineto{lineto}
 \ctr@ld@f\def\c@mmoveto{moveto}
 \ctr@ld@f\def\c@msetgray{setgray}          \ctr@ld@f\def\c@msetgrayStroke{}
 \ctr@ld@f\def\c@msetcmykcolor{setcmykcolor}\ctr@ld@f\def\c@msetcmykcolorStroke{}
 \ctr@ld@f\def\c@msetrgbcolor{setrgbcolor}  \ctr@ld@f\def\c@msetrgbcolorStroke{}
 \ctr@ld@f\def\d@fprimarC@lor{\curr@ntcolor\space\curr@ntcolorc@md}
 \ctr@ld@f\def\d@fsecondC@lor{\sec@ndcolor\space\sec@ndcolorc@md}
 \ctr@ld@f\def\d@fthirdC@lor{\th@rdcolor\space\th@rdcolorc@md}
 \ctr@ld@f\def\c@msetdash{setdash}
 \ctr@ld@f\def\c@msetlinejoin{setlinejoin}
 \ctr@ld@f\def\c@msetlinewidth{setlinewidth}
 \ctr@ld@f\def\f@gclosestroke{\immediate\write\fwf@g{closepath\space stroke}}
 \ctr@ld@f\def\f@gfill{\immediate\write\fwf@g{\fillc@md}}
 \ctr@ld@f\def\f@gnewpath{\immediate\write\fwf@g{newpath}}
 \ctr@ld@f\def\f@gstroke{\immediate\write\fwf@g{stroke}}
\fi}
\ctr@ld@f\def\c@pypsfile#1#2{\c@pyfil@{\immediate\write#1}{#2}}
\ctr@ld@f\def\Figinclud@PDF#1#2{\openin\frf@g=#1\pdfliteral{q #2 0 0 #2 0 0 cm}%
    \c@pyfil@{\pdfliteral}{\frf@g}\pdfliteral{Q}\closein\frf@g}
\ctr@ln@w{newif}\ifmored@ta
\ctr@ln@m\bl@nkline
\ctr@ld@f\def\c@pyfil@#1#2{\def\bl@nkline{\par}{\catcode`\%=12
    \loop\ifeof#2\mored@tafalse\else\mored@tatrue\immediate\read#2 to\tr@c
    \ifx\tr@c\bl@nkline\else#1{\tr@c}\fi\fi\ifmored@ta\repeat}}
\ctr@ld@f\def\keln@mun#1#2|{\def\l@debut{#1}\def\l@suite{#2}}
\ctr@ld@f\def\keln@mde#1#2#3|{\def\l@debut{#1#2}\def\l@suite{#3}}
\ctr@ld@f\def\keln@mtr#1#2#3#4|{\def\l@debut{#1#2#3}\def\l@suite{#4}}
\ctr@ld@f\def\keln@mqu#1#2#3#4#5|{\def\l@debut{#1#2#3#4}\def\l@suite{#5}}
\ctr@ld@f\let\@psffilein=\frf@g 
\ctr@ln@w{newif}\if@psffileok    
\ctr@ln@w{newif}\if@psfbbfound   
\ctr@ln@w{newif}\if@psfverbose   
\@psfverbosetrue
\ctr@ln@m\@psfllx \ctr@ln@m\@psflly
\ctr@ln@m\@psfurx \ctr@ln@m\@psfury
\ctr@ln@m\resetcolonc@tcode
\ctr@ld@f\def\@psfgetbb#1{\global\@psfbbfoundfalse%
\global\def\@psfllx{0}\global\def\@psflly{0}%
\global\def\@psfurx{30}\global\def\@psfury{30}%
\openin\@psffilein=#1\relax
\ifeof\@psffilein\errmessage{I couldn't open #1, will ignore it}\else
   \edef\resetcolonc@tcode{\catcode`\noexpand\:\the\catcode`\:\relax}%
   {\@psffileoktrue \chardef\other=12
    \def\do##1{\catcode`##1=\other}\dospecials \catcode`\ =10 \resetcolonc@tcode
    \loop
       \read\@psffilein to \@psffileline
       \ifeof\@psffilein\@psffileokfalse\else
          \expandafter\@psfaux\@psffileline:. \\%
       \fi
   \if@psffileok\repeat
   \if@psfbbfound\else
    \if@psfverbose\message{No bounding box comment in #1; using defaults}\fi\fi
   }\closein\@psffilein\fi}%
\ctr@ln@m\@psfbblit
\ctr@ln@m\@psfpercent
{\catcode`\%=12 \global\let\@psfpercent=
\ctr@ln@m\@psfaux
\long\def\@psfaux#1#2:#3\\{\ifx#1\@psfpercent
   \def\testit{#2}\ifx\testit\@psfbblit
      \@psfgrab #3 . . . \\%
      \@psffileokfalse
      \global\@psfbbfoundtrue
   \fi\else\ifx#1\par\else\@psffileokfalse\fi\fi}%
\ctr@ld@f\def\@psfempty{}%
\ctr@ld@f\def\@psfgrab #1 #2 #3 #4 #5\\{%
\global\def\@psfllx{#1}\ifx\@psfllx\@psfempty
      \@psfgrab #2 #3 #4 #5 .\\\else
   \global\def\@psflly{#2}%
   \global\def\@psfurx{#3}\global\def\@psfury{#4}\fi}%
\ctr@ld@f\def\PSwrit@cmd#1#2#3{{\Figg@tXY{#1}\c@lprojSP\b@undb@x{\v@lX}{\v@lY}%
    \v@lX=\ptT@ptps\v@lX\v@lY=\ptT@ptps\v@lY%
    \immediate\write#3{\repdecn@mb{\v@lX}\space\repdecn@mb{\v@lY}\space#2}}}
\ctr@ld@f\def\PSwrit@cmdS#1#2#3#4#5{{\Figg@tXY{#1}\c@lprojSP\b@undb@x{\v@lX}{\v@lY}%
    \global\result@t=\v@lX\global\result@@t=\v@lY%
    \v@lX=\ptT@ptps\v@lX\v@lY=\ptT@ptps\v@lY%
    \immediate\write#3{\repdecn@mb{\v@lX}\space\repdecn@mb{\v@lY}\space#2}}%
    \edef#4{\the\result@t}\edef#5{\the\result@@t}}
\ctr@ld@f\def\psaltitude#1[#2,#3,#4]{{\ifcurr@ntPS\ifps@cri%
    \PSc@mment{psaltitude Square Dim=#1, Triangle=[#2 / #3,#4]}%
    \s@uvc@ntr@l\et@tpsaltitude\resetc@ntr@l{2}\figptorthoprojline-5:=#2/#3,#4/%
    \figvectP -1[#3,#4]\n@rminf{\v@leur}{-1}\vecunit@{-3}{-1}%
    \figvectP -1[-5,#3]\n@rminf{\v@lmin}{-1}\figvectP -2[-5,#4]\n@rminf{\v@lmax}{-2}%
    \ifdim\v@lmin<\v@lmax\s@mme=#3\else\v@lmax=\v@lmin\s@mme=#4\fi%
    \figvectP -4[-5,#2]\vecunit@{-4}{-4}\delt@=#1\unit@%
    \edef\t@ille{\repdecn@mb{\delt@}}\figpttra-1:=-5/\t@ille,-3/%
    \figptstra-3=-5,-1/\t@ille,-4/\psline[#2,-5]\s@uvdash{\typ@dash}%
    \pssetdash{\defaultdash}\psline[-1,-2,-3]\pssetdash{\typ@dash}%
    \ifdim\v@leur<\v@lmax\Pss@tsecondSt\psline[-5,\the\s@mme]\Psrest@reSt\fi%
    \PSc@mment{End psaltitude}\resetc@ntr@l\et@tpsaltitude\fi\fi}}
\ctr@ld@f\def\Ps@rcerc#1;#2(#3,#4){\ellBB@x#1;#2,#2(#3,#4,0)%
    \f@gnewpath{\delt@=#2\unit@\delt@=\ptT@ptps\delt@%
    \BdingB@xfalse%
    \PSwrit@cmd{#1}{\repdecn@mb{\delt@}\space #3\space #4\space arc}{\fwf@g}}}
\ctr@ln@m\psarccirc
\ctr@ld@f\def\psarccircDD#1;#2(#3,#4){\ifcurr@ntPS\ifps@cri%
    \PSc@mment{psarccircDD Center=#1 ; Radius=#2 (Ang1=#3, Ang2=#4)}%
    \iffillm@de\Ps@rcerc#1;#2(#3,#4)%
    \f@gfill%
    \else\Ps@rcerc#1;#2(#3,#4)\f@gstroke\fi%
    \PSc@mment{End psarccircDD}\fi\fi}
\ctr@ld@f\def\psarccircTD#1,#2,#3;#4(#5,#6){{\ifcurr@ntPS\ifps@cri\s@uvc@ntr@l\et@tpsarccircTD%
    \PSc@mment{psarccircTD Center=#1,P1=#2,P2=#3 ; Radius=#4 (Ang1=#5, Ang2=#6)}%
    \setc@ntr@l{2}\c@lExtAxes#1,#2,#3(#4)\psarcellPATD#1,-4,-5(#5,#6)%
    \PSc@mment{End psarccircTD}\resetc@ntr@l\et@tpsarccircTD\fi\fi}}
\ctr@ld@f\def\c@lExtAxes#1,#2,#3(#4){%
    \figvectPTD-5[#1,#2]\vecunit@{-5}{-5}\figvectNTD-4[#1,#2,#3]\vecunit@{-4}{-4}%
    \figvectNVTD-3[-4,-5]\delt@=#4\unit@\edef\r@yon{\repdecn@mb{\delt@}}%
    \figpttra-4:=#1/\r@yon,-5/\figpttra-5:=#1/\r@yon,-3/}
\ctr@ln@m\psarccircP
\ctr@ld@f\def\psarccircPDD#1;#2[#3,#4]{{\ifcurr@ntPS\ifps@cri\s@uvc@ntr@l\et@tpsarccircPDD%
    \PSc@mment{psarccircPDD Center=#1; Radius=#2, [P1=#3, P2=#4]}%
    \Ps@ngleparam#1;#2[#3,#4]\ifdim\v@lmin>\v@lmax\advance\v@lmax\DePI@deg\fi%
    \edef\@ngdeb{\repdecn@mb{\v@lmin}}\edef\@ngfin{\repdecn@mb{\v@lmax}}%
    \psarccirc#1;\r@dius(\@ngdeb,\@ngfin)%
    \PSc@mment{End psarccircPDD}\resetc@ntr@l\et@tpsarccircPDD\fi\fi}}
\ctr@ld@f\def\psarccircPTD#1;#2[#3,#4,#5]{{\ifcurr@ntPS\ifps@cri\s@uvc@ntr@l\et@tpsarccircPTD%
    \PSc@mment{psarccircPTD Center=#1; Radius=#2, [P1=#3, P2=#4, P3=#5]}%
    \setc@ntr@l{2}\c@lExtAxes#1,#3,#5(#2)\psarcellPP#1,-4,-5[#3,#4]%
    \PSc@mment{End psarccircPTD}\resetc@ntr@l\et@tpsarccircPTD\fi\fi}}
\ctr@ld@f\def\Ps@ngleparam#1;#2[#3,#4]{\setc@ntr@l{2}%
    \figvectPDD-1[#1,#3]\vecunit@{-1}{-1}\Figg@tXY{-1}\arct@n\v@lmin(\v@lX,\v@lY)%
    \figvectPDD-2[#1,#4]\vecunit@{-2}{-2}\Figg@tXY{-2}\arct@n\v@lmax(\v@lX,\v@lY)%
    \v@lmin=\rdT@deg\v@lmin\v@lmax=\rdT@deg\v@lmax%
    \v@leur=#2pt\maxim@m{\mili@u}{-\v@leur}{\v@leur}%
    \edef\r@dius{\repdecn@mb{\mili@u}}}
\ctr@ld@f\def\Ps@rcercBz#1;#2(#3,#4){\Ps@rellBz#1;#2,#2(#3,#4,0)}
\ctr@ld@f\def\Ps@rellBz#1;#2,#3(#4,#5,#6){%
    \ellBB@x#1;#2,#3(#4,#5,#6)\BdingB@xfalse%
    \c@lNbarcs{#4}{#5}\v@leur=#4pt\setc@ntr@l{2}\figptell-13::#1;#2,#3(#4,#6)%
    \f@gnewpath\PSwrit@cmd{-13}{\c@mmoveto}{\fwf@g}%
    \s@mme=\z@\bcl@rellBz#1;#2,#3(#6)\BdingB@xtrue}
\ctr@ld@f\def\bcl@rellBz#1;#2,#3(#4){\relax%
    \ifnum\s@mme<\p@rtent\advance\s@mme\@ne%
    \advance\v@leur\delt@\edef\@ngle{\repdecn@mb\v@leur}\figptell-14::#1;#2,#3(\@ngle,#4)%
    \advance\v@leur\delt@\edef\@ngle{\repdecn@mb\v@leur}\figptell-15::#1;#2,#3(\@ngle,#4)%
    \advance\v@leur\delt@\edef\@ngle{\repdecn@mb\v@leur}\figptell-16::#1;#2,#3(\@ngle,#4)%
    \figptscontrolDD-18[-13,-14,-15,-16]%
    \PSwrit@cmd{-18}{}{\fwf@g}\PSwrit@cmd{-17}{}{\fwf@g}%
    \PSwrit@cmd{-16}{\c@mcurveto}{\fwf@g}%
    \figptcopyDD-13:/-16/\bcl@rellBz#1;#2,#3(#4)\fi}
\ctr@ld@f\def\Ps@rell#1;#2,#3(#4,#5,#6){\ellBB@x#1;#2,#3(#4,#5,#6)%
    \f@gnewpath{\v@lmin=#2\unit@\v@lmin=\ptT@ptps\v@lmin%
    \v@lmax=#3\unit@\v@lmax=\ptT@ptps\v@lmax\BdingB@xfalse%
    \PSwrit@cmd{#1}%
    {#6\space\repdecn@mb{\v@lmin}\space\repdecn@mb{\v@lmax}\space #4\space #5\space ellipse}{\fwf@g}}%
    \global\Use@llipsetrue}
\ctr@ln@m\psarcell
\ctr@ld@f\def\psarcellDD#1;#2,#3(#4,#5,#6){{\ifcurr@ntPS\ifps@cri%
    \PSc@mment{psarcellDD Center=#1 ; XRad=#2, YRad=#3 (Ang1=#4, Ang2=#5, Inclination=#6)}%
    \iffillm@de\Ps@rell#1;#2,#3(#4,#5,#6)%
    \f@gfill%
    \else\Ps@rell#1;#2,#3(#4,#5,#6)\f@gstroke\fi%
    \PSc@mment{End psarcellDD}\fi\fi}}
\ctr@ld@f\def\psarcellTD#1;#2,#3(#4,#5,#6){{\ifcurr@ntPS\ifps@cri\s@uvc@ntr@l\et@tpsarcellTD%
    \PSc@mment{psarcellTD Center=#1 ; XRad=#2, YRad=#3 (Ang1=#4, Ang2=#5, Inclination=#6)}%
    \setc@ntr@l{2}\figpttraC -8:=#1/#2,0,0/\figpttraC -7:=#1/0,#3,0/%
    \figvectC -4(0,0,1)\figptsrot -8=-8,-7/#1,#6,-4/\psarcellPATD#1,-8,-7(#4,#5)%
    \PSc@mment{End psarcellTD}\resetc@ntr@l\et@tpsarcellTD\fi\fi}}
\ctr@ln@m\psarcellPA
\ctr@ld@f\def\psarcellPADD#1,#2,#3(#4,#5){{\ifcurr@ntPS\ifps@cri\s@uvc@ntr@l\et@tpsarcellPADD%
    \PSc@mment{psarcellPADD Center=#1,PtAxis1=#2,PtAxis2=#3 (Ang1=#4, Ang2=#5)}%
    \setc@ntr@l{2}\figvectPDD-1[#1,#2]\vecunit@DD{-1}{-1}\v@lX=\ptT@unit@\result@t%
    \edef\XR@d{\repdecn@mb{\v@lX}}\Figg@tXY{-1}\arct@n\v@lmin(\v@lX,\v@lY)%
    \v@lmin=\rdT@deg\v@lmin\edef\Inclin@{\repdecn@mb{\v@lmin}}%
    \figgetdist\YR@d[#1,#3]\psarcellDD#1;\XR@d,\YR@d(#4,#5,\Inclin@)%
    \PSc@mment{End psarcellPADD}\resetc@ntr@l\et@tpsarcellPADD\fi\fi}}
\ctr@ld@f\def\psarcellPATD#1,#2,#3(#4,#5){{\ifcurr@ntPS\ifps@cri\s@uvc@ntr@l\et@tpsarcellPATD%
    \PSc@mment{psarcellPATD Center=#1,PtAxis1=#2,PtAxis2=#3 (Ang1=#4, Ang2=#5)}%
    \iffillm@de\Ps@rellPATD#1,#2,#3(#4,#5)%
    \f@gfill%
    \else\Ps@rellPATD#1,#2,#3(#4,#5)\f@gstroke\fi%
    \PSc@mment{End psarcellPATD}\resetc@ntr@l\et@tpsarcellPATD\fi\fi}}
\ctr@ld@f\def\Ps@rellPATD#1,#2,#3(#4,#5){\let\c@lprojSP=\relax%
    \setc@ntr@l{2}\figvectPTD-1[#1,#2]\figvectPTD-2[#1,#3]\c@lNbarcs{#4}{#5}%
    \v@leur=#4pt\c@lptellP{#1}{-1}{-2}\Figptpr@j-5:/-3/%
    \f@gnewpath\PSwrit@cmdS{-5}{\c@mmoveto}{\fwf@g}{\X@un}{\Y@un}%
    \edef\C@nt@r{#1}\s@mme=\z@\bcl@rellPATD}
\ctr@ld@f\def\bcl@rellPATD{\relax%
    \ifnum\s@mme<\p@rtent\advance\s@mme\@ne%
    \advance\v@leur\delt@\c@lptellP{\C@nt@r}{-1}{-2}\Figptpr@j-4:/-3/%
    \advance\v@leur\delt@\c@lptellP{\C@nt@r}{-1}{-2}\Figptpr@j-6:/-3/%
    \advance\v@leur\delt@\c@lptellP{\C@nt@r}{-1}{-2}\Figptpr@j-3:/-3/%
    \v@lX=\z@\v@lY=\z@\Figtr@nptDD{-5}{-5}\Figtr@nptDD{2}{-3}%
    \divide\v@lX\@vi\divide\v@lY\@vi%
    \Figtr@nptDD{3}{-4}\Figtr@nptDD{-1.5}{-6}\v@lmin=\v@lX\v@lmax=\v@lY%
    \v@lX=\z@\v@lY=\z@\Figtr@nptDD{2}{-5}\Figtr@nptDD{-5}{-3}%
    \divide\v@lX\@vi\divide\v@lY\@vi\Figtr@nptDD{-1.5}{-4}\Figtr@nptDD{3}{-6}%
    \BdingB@xfalse%
    \Figp@intregDD-4:(\v@lmin,\v@lmax)\PSwrit@cmdS{-4}{}{\fwf@g}{\X@de}{\Y@de}%
    \Figp@intregDD-4:(\v@lX,\v@lY)\PSwrit@cmdS{-4}{}{\fwf@g}{\X@tr}{\Y@tr}%
    \BdingB@xtrue\PSwrit@cmdS{-3}{\c@mcurveto}{\fwf@g}{\X@qu}{\Y@qu}%
    \B@zierBB@x{1}{\Y@un}(\X@un,\X@de,\X@tr,\X@qu)%
    \B@zierBB@x{2}{\X@un}(\Y@un,\Y@de,\Y@tr,\Y@qu)%
    \edef\X@un{\X@qu}\edef\Y@un{\Y@qu}\figptcopyDD-5:/-3/\bcl@rellPATD\fi}
\ctr@ld@f\def\c@lNbarcs#1#2{%
    \delt@=#2pt\advance\delt@-#1pt\maxim@m{\v@lmax}{\delt@}{-\delt@}%
    \v@leur=\v@lmax\divide\v@leur45 \p@rtentiere{\p@rtent}{\v@leur}\advance\p@rtent\@ne%
    \s@mme=\p@rtent\multiply\s@mme\thr@@\divide\delt@\s@mme}
\ctr@ld@f\def\psarcellPP#1,#2,#3[#4,#5]{{\ifcurr@ntPS\ifps@cri\s@uvc@ntr@l\et@tpsarcellPP%
    \PSc@mment{psarcellPP Center=#1,PtAxis1=#2,PtAxis2=#3 [Point1=#4, Point2=#5]}%
    \setc@ntr@l{2}\figvectP-2[#1,#3]\vecunit@{-2}{-2}\v@lmin=\result@t%
    \invers@{\v@lmax}{\v@lmin}%
    \figvectP-1[#1,#2]\vecunit@{-1}{-1}\v@leur=\result@t%
    \v@leur=\repdecn@mb{\v@lmax}\v@leur\edef\AsB@{\repdecn@mb{\v@leur}}
    \c@lAngle{#1}{#4}{\v@lmin}\edef\@ngdeb{\repdecn@mb{\v@lmin}}%
    \c@lAngle{#1}{#5}{\v@lmax}\ifdim\v@lmin>\v@lmax\advance\v@lmax\DePI@deg\fi%
    \edef\@ngfin{\repdecn@mb{\v@lmax}}\psarcellPA#1,#2,#3(\@ngdeb,\@ngfin)%
    \PSc@mment{End psarcellPP}\resetc@ntr@l\et@tpsarcellPP\fi\fi}}
\ctr@ld@f\def\c@lAngle#1#2#3{\figvectP-3[#1,#2]%
    \c@lproscal\delt@[-3,-1]\c@lproscal\v@leur[-3,-2]%
    \v@leur=\AsB@\v@leur\arct@n#3(\delt@,\v@leur)#3=\rdT@deg#3}
\ctr@ln@w{newif}\if@rrowratio\@rrowratiotrue
\ctr@ln@w{newif}\if@rrowhfill
\ctr@ln@w{newif}\if@rrowhout
\ctr@ld@f\def\Psset@rrowhe@d#1=#2|{\keln@mun#1|%
    \def\n@mref{a}\ifx\l@debut\n@mref\pssetarrowheadangle{#2}\else
    \def\n@mref{f}\ifx\l@debut\n@mref\pssetarrowheadfill{#2}\else
    \def\n@mref{l}\ifx\l@debut\n@mref\pssetarrowheadlength{#2}\else
    \def\n@mref{o}\ifx\l@debut\n@mref\pssetarrowheadout{#2}\else
    \def\n@mref{r}\ifx\l@debut\n@mref\pssetarrowheadratio{#2}\else
    \immediate\write16{*** Unknown attribute: \BS@ psset arrowhead(..., #1=...)}%
    \fi\fi\fi\fi\fi}
\ctr@ln@m\@rrowheadangle
\ctr@ln@m\C@AHANG \ctr@ln@m\S@AHANG \ctr@ln@m\UNSS@N
\ctr@ld@f\def\pssetarrowheadangle#1{\edef\@rrowheadangle{#1}{\c@ssin{\C@}{\S@}{#1}%
    \xdef\C@AHANG{\C@}\xdef\S@AHANG{\S@}\v@lmax=\S@ pt%
    \invers@{\v@leur}{\v@lmax}\maxim@m{\v@leur}{\v@leur}{-\v@leur}%
    \xdef\UNSS@N{\the\v@leur}}}
\ctr@ld@f\def\pssetarrowheadfill#1{\expandafter\set@rrowhfill#1:}
\ctr@ld@f\def\set@rrowhfill#1#2:{\if#1n\@rrowhfillfalse\else\@rrowhfilltrue\fi}
\ctr@ld@f\def\pssetarrowheadout#1{\expandafter\set@rrowhout#1:}
\ctr@ld@f\def\set@rrowhout#1#2:{\if#1n\@rrowhoutfalse\else\@rrowhouttrue\fi}
\ctr@ln@m\@rrowheadlength
\ctr@ld@f\def\pssetarrowheadlength#1{\edef\@rrowheadlength{#1}\@rrowratiofalse}
\ctr@ln@m\@rrowheadratio
\ctr@ld@f\def\pssetarrowheadratio#1{\edef\@rrowheadratio{#1}\@rrowratiotrue}
\ctr@ln@m\defaultarrowheadlength
\ctr@ld@f\def\psresetarrowhead{%
    \pssetarrowheadangle{\defaultarrowheadangle}%
    \pssetarrowheadfill{\defaultarrowheadfill}%
    \pssetarrowheadout{\defaultarrowheadout}%
    \pssetarrowheadratio{\defaultarrowheadratio}%
    \d@fm@cdim\defaultarrowheadlength{\defaulth@rdahlength}
    \pssetarrowheadlength{\defaultarrowheadlength}}
\ctr@ld@f\def\defaultarrowheadratio{0.1}
\ctr@ld@f\def\defaultarrowheadangle{20}
\ctr@ld@f\def\defaultarrowheadfill{no}
\ctr@ld@f\def\defaultarrowheadout{no}
\ctr@ld@f\def\defaulth@rdahlength{8pt}
\ctr@ln@m\psarrow
\ctr@ld@f\def\psarrowDD[#1,#2]{{\ifcurr@ntPS\ifps@cri\s@uvc@ntr@l\et@tpsarrow%
    \PSc@mment{psarrowDD [Pt1,Pt2]=[#1,#2]}\pssetfillmode{no}%
    \psarrowheadDD[#1,#2]\setc@ntr@l{2}\psline[#1,-3]%
    \PSc@mment{End psarrowDD}\resetc@ntr@l\et@tpsarrow\fi\fi}}
\ctr@ld@f\def\psarrowTD[#1,#2]{{\ifcurr@ntPS\ifps@cri\s@uvc@ntr@l\et@tpsarrowTD%
    \PSc@mment{psarrowTD [Pt1,Pt2]=[#1,#2]}\resetc@ntr@l{2}%
    \Figptpr@j-5:/#1/\Figptpr@j-6:/#2/\let\c@lprojSP=\relax\psarrowDD[-5,-6]%
    \PSc@mment{End psarrowTD}\resetc@ntr@l\et@tpsarrowTD\fi\fi}}
\ctr@ln@m\psarrowhead
\ctr@ld@f\def\psarrowheadDD[#1,#2]{{\ifcurr@ntPS\ifps@cri\s@uvc@ntr@l\et@tpsarrowheadDD%
    \if@rrowhfill\def\@hangle{-\@rrowheadangle}\else\def\@hangle{\@rrowheadangle}\fi%
    \if@rrowratio%
    \if@rrowhout\def\@hratio{-\@rrowheadratio}\else\def\@hratio{\@rrowheadratio}\fi%
    \PSc@mment{psarrowheadDD Ratio=\@hratio, Angle=\@hangle, [Pt1,Pt2]=[#1,#2]}%
    \Ps@rrowhead\@hratio,\@hangle[#1,#2]%
    \else%
    \if@rrowhout\def\@hlength{-\@rrowheadlength}\else\def\@hlength{\@rrowheadlength}\fi%
    \PSc@mment{psarrowheadDD Length=\@hlength, Angle=\@hangle, [Pt1,Pt2]=[#1,#2]}%
    \Ps@rrowheadfd\@hlength,\@hangle[#1,#2]%
    \fi%
    \PSc@mment{End psarrowheadDD}\resetc@ntr@l\et@tpsarrowheadDD\fi\fi}}
\ctr@ld@f\def\psarrowheadTD[#1,#2]{{\ifcurr@ntPS\ifps@cri\s@uvc@ntr@l\et@tpsarrowheadTD%
    \PSc@mment{psarrowheadTD [Pt1,Pt2]=[#1,#2]}\resetc@ntr@l{2}%
    \Figptpr@j-5:/#1/\Figptpr@j-6:/#2/\let\c@lprojSP=\relax\psarrowheadDD[-5,-6]%
    \PSc@mment{End psarrowheadTD}\resetc@ntr@l\et@tpsarrowheadTD\fi\fi}}
\ctr@ld@f\def\Ps@rrowhead#1,#2[#3,#4]{\v@leur=#1\p@\maxim@m{\v@leur}{\v@leur}{-\v@leur}%
    \ifdim\v@leur>\Cepsil@n{
    \PSc@mment{ps@rrowhead Ratio=#1, Angle=#2, [Pt1,Pt2]=[#3,#4]}\v@leur=\UNSS@N%
    \v@leur=\curr@ntwidth\v@leur\v@leur=\ptpsT@pt\v@leur\delt@=.5\v@leur
    \setc@ntr@l{2}\figvectPDD-3[#4,#3]%
    \Figg@tXY{-3}\v@lX=#1\v@lX\v@lY=#1\v@lY\Figv@ctCreg-3(\v@lX,\v@lY)%
    \vecunit@{-4}{-3}\mili@u=\result@t%
    \ifdim#2pt>\z@\v@lXa=-\C@AHANG\delt@%
     \edef\c@ef{\repdecn@mb{\v@lXa}}\figpttraDD-3:=-3/\c@ef,-4/\fi%
    \edef\c@ef{\repdecn@mb{\delt@}}%
    \v@lXa=\mili@u\v@lXa=\C@AHANG\v@lXa%
    \v@lYa=\ptpsT@pt\p@\v@lYa=\curr@ntwidth\v@lYa\v@lYa=\sDcc@ngle\v@lYa%
    \advance\v@lXa-\v@lYa\gdef\sDcc@ngle{0}%
    \ifdim\v@lXa>\v@leur\edef\c@efendpt{\repdecn@mb{\v@leur}}%
    \else\edef\c@efendpt{\repdecn@mb{\v@lXa}}\fi%
    \Figg@tXY{-3}\v@lmin=\v@lX\v@lmax=\v@lY%
    \v@lXa=\C@AHANG\v@lmin\v@lYa=\S@AHANG\v@lmax\advance\v@lXa\v@lYa%
    \v@lYa=-\S@AHANG\v@lmin\v@lX=\C@AHANG\v@lmax\advance\v@lYa\v@lX%
    \setc@ntr@l{1}\Figg@tXY{#4}\advance\v@lX\v@lXa\advance\v@lY\v@lYa%
    \setc@ntr@l{2}\Figp@intregDD-2:(\v@lX,\v@lY)%
    \v@lXa=\C@AHANG\v@lmin\v@lYa=-\S@AHANG\v@lmax\advance\v@lXa\v@lYa%
    \v@lYa=\S@AHANG\v@lmin\v@lX=\C@AHANG\v@lmax\advance\v@lYa\v@lX%
    \setc@ntr@l{1}\Figg@tXY{#4}\advance\v@lX\v@lXa\advance\v@lY\v@lYa%
    \setc@ntr@l{2}\Figp@intregDD-1:(\v@lX,\v@lY)%
    \ifdim#2pt<\z@\fillm@detrue\psline[-2,#4,-1]
    \else\figptstraDD-3=#4,-2,-1/\c@ef,-4/\psline[-2,-3,-1]\fi
    \ifdim#1pt>\z@\figpttraDD-3:=#4/\c@efendpt,-4/\else\figptcopyDD-3:/#4/\fi%
    \PSc@mment{End ps@rrowhead}}\fi}
\ctr@ld@f\def\sDcc@ngle{0}
\ctr@ld@f\def\Ps@rrowheadfd#1,#2[#3,#4]{{%
    \PSc@mment{ps@rrowheadfd Length=#1, Angle=#2, [Pt1,Pt2]=[#3,#4]}%
    \setc@ntr@l{2}\figvectPDD-1[#3,#4]\n@rmeucDD{\v@leur}{-1}\v@leur=\ptT@unit@\v@leur%
    \invers@{\v@leur}{\v@leur}\v@leur=#1\v@leur\edef\R@tio{\repdecn@mb{\v@leur}}%
    \Ps@rrowhead\R@tio,#2[#3,#4]\PSc@mment{End ps@rrowheadfd}}}
\ctr@ln@m\psarrowBezier
\ctr@ld@f\def\psarrowBezierDD[#1,#2,#3,#4]{{\ifcurr@ntPS\ifps@cri\s@uvc@ntr@l\et@tpsarrowBezierDD%
    \PSc@mment{psarrowBezierDD Control points=#1,#2,#3,#4}\setc@ntr@l{2}%
    \if@rrowratio\c@larclengthDD\v@leur,10[#1,#2,#3,#4]\else\v@leur=\z@\fi%
    \Ps@rrowB@zDD\v@leur[#1,#2,#3,#4]%
    \PSc@mment{End psarrowBezierDD}\resetc@ntr@l\et@tpsarrowBezierDD\fi\fi}}
\ctr@ld@f\def\psarrowBezierTD[#1,#2,#3,#4]{{\ifcurr@ntPS\ifps@cri\s@uvc@ntr@l\et@tpsarrowBezierTD%
    \PSc@mment{psarrowBezierTD Control points=#1,#2,#3,#4}\resetc@ntr@l{2}%
    \Figptpr@j-7:/#1/\Figptpr@j-8:/#2/\Figptpr@j-9:/#3/\Figptpr@j-10:/#4/%
    \let\c@lprojSP=\relax\ifnum\curr@ntproj<\tw@\psarrowBezierDD[-7,-8,-9,-10]%
    \else\f@gnewpath\PSwrit@cmd{-7}{\c@mmoveto}{\fwf@g}%
    \if@rrowratio\c@larclengthDD\mili@u,10[-7,-8,-9,-10]\else\mili@u=\z@\fi%
    \p@rtent=\NBz@rcs\advance\p@rtent\m@ne\subB@zierTD\p@rtent[#1,#2,#3,#4]%
    \f@gstroke%
    \advance\v@lmin\p@rtent\delt@
    \v@leur=\v@lmin\advance\v@leur0.33333 \delt@\edef\unti@rs{\repdecn@mb{\v@leur}}%
    \v@leur=\v@lmin\advance\v@leur0.66666 \delt@\edef\deti@rs{\repdecn@mb{\v@leur}}%
    \figptcopyDD-8:/-10/\c@lsubBzarc\unti@rs,\deti@rs[#1,#2,#3,#4]%
    \figptcopyDD-8:/-4/\figptcopyDD-9:/-3/\Ps@rrowB@zDD\mili@u[-7,-8,-9,-10]\fi%
    \PSc@mment{End psarrowBezierTD}\resetc@ntr@l\et@tpsarrowBezierTD\fi\fi}}
\ctr@ld@f\def\c@larclengthDD#1,#2[#3,#4,#5,#6]{{\p@rtent=#2\figptcopyDD-5:/#3/%
    \delt@=\p@\divide\delt@\p@rtent\c@rre=\z@\v@leur=\z@\s@mme=\z@%
    \loop\ifnum\s@mme<\p@rtent\advance\s@mme\@ne\advance\v@leur\delt@%
    \edef\T@{\repdecn@mb{\v@leur}}\figptBezierDD-6::\T@[#3,#4,#5,#6]%
    \figvectPDD-1[-5,-6]\n@rmeucDD{\mili@u}{-1}\advance\c@rre\mili@u%
    \figptcopyDD-5:/-6/\repeat\global\result@t=\ptT@unit@\c@rre}#1=\result@t}
\ctr@ld@f\def\Ps@rrowB@zDD#1[#2,#3,#4,#5]{{\pssetfillmode{no}%
    \if@rrowratio\delt@=\@rrowheadratio#1\else\delt@=\@rrowheadlength pt\fi%
    \v@leur=\C@AHANG\delt@\edef\R@dius{\repdecn@mb{\v@leur}}%
    \FigptintercircB@zDD-5::0,\R@dius[#5,#4,#3,#2]%
    \pssetarrowheadlength{\repdecn@mb{\delt@}}\psarrowheadDD[-5,#5]%
    \let\n@rmeuc=\n@rmeucDD\figgetdist\R@dius[#5,-3]%
    \FigptintercircB@zDD-6::0,\R@dius[#5,#4,#3,#2]%
    \figptBezierDD-5::0.33333[#5,#4,#3,#2]\figptBezierDD-3::0.66666[#5,#4,#3,#2]%
    \figptscontrolDD-5[-6,-5,-3,#2]\psBezierDD1[-6,-5,-4,#2]}}
\ctr@ln@m\psarrowcirc
\ctr@ld@f\def\psarrowcircDD#1;#2(#3,#4){{\ifcurr@ntPS\ifps@cri\s@uvc@ntr@l\et@tpsarrowcircDD%
    \PSc@mment{psarrowcircDD Center=#1 ; Radius=#2 (Ang1=#3,Ang2=#4)}%
    \pssetfillmode{no}\Pscirc@rrowhead#1;#2(#3,#4)%
    \setc@ntr@l{2}\figvectPDD -4[#1,-3]\vecunit@{-4}{-4}%
    \Figg@tXY{-4}\arct@n\v@lmin(\v@lX,\v@lY)%
    \v@lmin=\rdT@deg\v@lmin\v@leur=#4pt\advance\v@leur-\v@lmin%
    \maxim@m{\v@leur}{\v@leur}{-\v@leur}%
    \ifdim\v@leur>\DemiPI@deg\relax\ifdim\v@lmin<#4pt\advance\v@lmin\DePI@deg%
    \else\advance\v@lmin-\DePI@deg\fi\fi\edef\ar@ngle{\repdecn@mb{\v@lmin}}%
    \ifdim#3pt<#4pt\psarccirc#1;#2(#3,\ar@ngle)\else\psarccirc#1;#2(\ar@ngle,#3)\fi%
    \PSc@mment{End psarrowcircDD}\resetc@ntr@l\et@tpsarrowcircDD\fi\fi}}
\ctr@ld@f\def\psarrowcircTD#1,#2,#3;#4(#5,#6){{\ifcurr@ntPS\ifps@cri\s@uvc@ntr@l\et@tpsarrowcircTD%
    \PSc@mment{psarrowcircTD Center=#1,P1=#2,P2=#3 ; Radius=#4 (Ang1=#5, Ang2=#6)}%
    \resetc@ntr@l{2}\c@lExtAxes#1,#2,#3(#4)\let\c@lprojSP=\relax%
    \figvectPTD-11[#1,-4]\figvectPTD-12[#1,-5]\c@lNbarcs{#5}{#6}%
    \if@rrowratio\v@lmax=\degT@rd\v@lmax\edef\D@lpha{\repdecn@mb{\v@lmax}}\fi%
    \advance\p@rtent\m@ne\mili@u=\z@%
    \v@leur=#5pt\c@lptellP{#1}{-11}{-12}\Figptpr@j-9:/-3/%
    \f@gnewpath\PSwrit@cmdS{-9}{\c@mmoveto}{\fwf@g}{\X@un}{\Y@un}%
    \edef\C@nt@r{#1}\s@mme=\z@\bcl@rcircTD\f@gstroke%
    \advance\v@leur\delt@\c@lptellP{#1}{-11}{-12}\Figptpr@j-5:/-3/%
    \advance\v@leur\delt@\c@lptellP{#1}{-11}{-12}\Figptpr@j-6:/-3/%
    \advance\v@leur\delt@\c@lptellP{#1}{-11}{-12}\Figptpr@j-10:/-3/%
    \figptscontrolDD-8[-9,-5,-6,-10]%
    \if@rrowratio\c@lcurvradDD0.5[-9,-8,-7,-10]\advance\mili@u\result@t%
    \maxim@m{\mili@u}{\mili@u}{-\mili@u}\mili@u=\ptT@unit@\mili@u%
    \mili@u=\D@lpha\mili@u\advance\p@rtent\@ne\divide\mili@u\p@rtent\fi%
    \Ps@rrowB@zDD\mili@u[-9,-8,-7,-10]%
    \PSc@mment{End psarrowcircTD}\resetc@ntr@l\et@tpsarrowcircTD\fi\fi}}
\ctr@ld@f\def\bcl@rcircTD{\relax%
    \ifnum\s@mme<\p@rtent\advance\s@mme\@ne%
    \advance\v@leur\delt@\c@lptellP{\C@nt@r}{-11}{-12}\Figptpr@j-5:/-3/%
    \advance\v@leur\delt@\c@lptellP{\C@nt@r}{-11}{-12}\Figptpr@j-6:/-3/%
    \advance\v@leur\delt@\c@lptellP{\C@nt@r}{-11}{-12}\Figptpr@j-10:/-3/%
    \figptscontrolDD-8[-9,-5,-6,-10]\BdingB@xfalse%
    \PSwrit@cmdS{-8}{}{\fwf@g}{\X@de}{\Y@de}\PSwrit@cmdS{-7}{}{\fwf@g}{\X@tr}{\Y@tr}%
    \BdingB@xtrue\PSwrit@cmdS{-10}{\c@mcurveto}{\fwf@g}{\X@qu}{\Y@qu}%
    \if@rrowratio\c@lcurvradDD0.5[-9,-8,-7,-10]\advance\mili@u\result@t\fi%
    \B@zierBB@x{1}{\Y@un}(\X@un,\X@de,\X@tr,\X@qu)%
    \B@zierBB@x{2}{\X@un}(\Y@un,\Y@de,\Y@tr,\Y@qu)%
    \edef\X@un{\X@qu}\edef\Y@un{\Y@qu}\figptcopyDD-9:/-10/\bcl@rcircTD\fi}
\ctr@ld@f\def\Pscirc@rrowhead#1;#2(#3,#4){{%
    \PSc@mment{pscirc@rrowhead Center=#1 ; Radius=#2 (Ang1=#3,Ang2=#4)}%
    \v@leur=#2\unit@\edef\s@glen{\repdecn@mb{\v@leur}}\v@lY=\z@\v@lX=\v@leur%
    \resetc@ntr@l{2}\Figv@ctCreg-3(\v@lX,\v@lY)\figpttraDD-5:=#1/1,-3/%
    \figptrotDD-5:=-5/#1,#4/%
    \figvectPDD-3[#1,-5]\Figg@tXY{-3}\v@leur=\v@lX%
    \ifdim#3pt<#4pt\v@lX=\v@lY\v@lY=-\v@leur\else\v@lX=-\v@lY\v@lY=\v@leur\fi%
    \Figv@ctCreg-3(\v@lX,\v@lY)\vecunit@{-3}{-3}%
    \if@rrowratio\v@leur=#4pt\advance\v@leur-#3pt\maxim@m{\mili@u}{-\v@leur}{\v@leur}%
    \mili@u=\degT@rd\mili@u\v@leur=\s@glen\mili@u\edef\s@glen{\repdecn@mb{\v@leur}}%
    \mili@u=#2\mili@u\mili@u=\@rrowheadratio\mili@u\else\mili@u=\@rrowheadlength pt\fi%
    \figpttraDD-6:=-5/\s@glen,-3/\v@leur=#2pt\v@leur=2\v@leur%
    \invers@{\v@leur}{\v@leur}\c@rre=\repdecn@mb{\v@leur}\mili@u
    \mili@u=\c@rre\mili@u=\repdecn@mb{\c@rre}\mili@u%
    \v@leur=\p@\advance\v@leur-\mili@u
    \invers@{\mili@u}{2\v@leur}\delt@=\c@rre\delt@=\repdecn@mb{\mili@u}\delt@%
    \xdef\sDcc@ngle{\repdecn@mb{\delt@}}
    \sqrt@{\mili@u}{\v@leur}\arct@n\v@leur(\mili@u,\c@rre)%
    \v@leur=\rdT@deg\v@leur
    \ifdim#3pt<#4pt\v@leur=-\v@leur\fi%
    \if@rrowhout\v@leur=-\v@leur\fi\edef\cor@ngle{\repdecn@mb{\v@leur}}%
    \figptrotDD-6:=-6/-5,\cor@ngle/\psarrowheadDD[-6,-5]%
    \PSc@mment{End pscirc@rrowhead}}}
\ctr@ln@m\psarrowcircP
\ctr@ld@f\def\psarrowcircPDD#1;#2[#3,#4]{{\ifcurr@ntPS\ifps@cri%
    \PSc@mment{psarrowcircPDD Center=#1; Radius=#2, [P1=#3,P2=#4]}%
    \s@uvc@ntr@l\et@tpsarrowcircPDD\Ps@ngleparam#1;#2[#3,#4]%
    \ifdim\v@leur>\z@\ifdim\v@lmin>\v@lmax\advance\v@lmax\DePI@deg\fi%
    \else\ifdim\v@lmin<\v@lmax\advance\v@lmin\DePI@deg\fi\fi%
    \edef\@ngdeb{\repdecn@mb{\v@lmin}}\edef\@ngfin{\repdecn@mb{\v@lmax}}%
    \psarrowcirc#1;\r@dius(\@ngdeb,\@ngfin)%
    \PSc@mment{End psarrowcircPDD}\resetc@ntr@l\et@tpsarrowcircPDD\fi\fi}}
\ctr@ld@f\def\psarrowcircPTD#1;#2[#3,#4,#5]{{\ifcurr@ntPS\ifps@cri\s@uvc@ntr@l\et@tpsarrowcircPTD%
    \PSc@mment{psarrowcircPTD Center=#1; Radius=#2, [P1=#3,P2=#4,P3=#5]}%
    \figgetangleTD\@ngfin[#1,#3,#4,#5]\v@leur=#2pt%
    \maxim@m{\mili@u}{-\v@leur}{\v@leur}\edef\r@dius{\repdecn@mb{\mili@u}}%
    \ifdim\v@leur<\z@\v@lmax=\@ngfin pt\advance\v@lmax-\DePI@deg%
    \edef\@ngfin{\repdecn@mb{\v@lmax}}\fi\psarrowcircTD#1,#3,#5;\r@dius(0,\@ngfin)%
    \PSc@mment{End psarrowcircPTD}\resetc@ntr@l\et@tpsarrowcircPTD\fi\fi}}
\ctr@ld@f\def\psaxes#1(#2){{\ifcurr@ntPS\ifps@cri\s@uvc@ntr@l\et@tpsaxes%
    \PSc@mment{psaxes Origin=#1 Range=(#2)}\an@lys@xes#2,:\resetc@ntr@l{2}%
    \ifx\t@xt@\empty\ifTr@isDim\ps@xes#1(0,#2,0,#2,0,#2)\else\ps@xes#1(0,#2,0,#2)\fi%
    \else\ps@xes#1(#2)\fi\PSc@mment{End psaxes}\resetc@ntr@l\et@tpsaxes\fi\fi}}
\ctr@ld@f\def\an@lys@xes#1,#2:{\def\t@xt@{#2}}
\ctr@ln@m\ps@xes
\ctr@ld@f\def\ps@xesDD#1(#2,#3,#4,#5){%
    \figpttraC-5:=#1/#2,0/\figpttraC-6:=#1/#3,0/\psarrowDD[-5,-6]%
    \figpttraC-5:=#1/0,#4/\figpttraC-6:=#1/0,#5/\psarrowDD[-5,-6]}
\ctr@ld@f\def\ps@xesTD#1(#2,#3,#4,#5,#6,#7){%
    \figpttraC-7:=#1/#2,0,0/\figpttraC-8:=#1/#3,0,0/\psarrowTD[-7,-8]%
    \figpttraC-7:=#1/0,#4,0/\figpttraC-8:=#1/0,#5,0/\psarrowTD[-7,-8]%
    \figpttraC-7:=#1/0,0,#6/\figpttraC-8:=#1/0,0,#7/\psarrowTD[-7,-8]}
\ctr@ln@m\newGr@FN
\ctr@ld@f\def\newGr@FNPDF#1{\s@mme=\Gr@FNb\advance\s@mme\@ne\xdef\Gr@FNb{\number\s@mme}}
\ctr@ld@f\def\newGr@FNDVI#1{\newGr@FNPDF{}\xdef#1{\jobname GI\Gr@FNb.anx}}
\ctr@ld@f\def\psbeginfig#1{\newGr@FN\DefGIfilen@me\gdef\@utoFN{0}%
    \def\t@xt@{#1}\relax\ifx\t@xt@\empty\psupdatem@detrue%
    \gdef\@utoFN{1}\Psb@ginfig\DefGIfilen@me\else\expandafter\Psb@ginfigNu@#1 :\fi}
\ctr@ld@f\def\Psb@ginfigNu@#1 #2:{\def\t@xt@{#1}\relax\ifx\t@xt@\empty\def\t@xt@{#2}%
    \ifx\t@xt@\empty\psupdatem@detrue\gdef\@utoFN{1}\Psb@ginfig\DefGIfilen@me%
    \else\Psb@ginfigNu@#2:\fi\else\Psb@ginfig{#1}\fi}
\ctr@ln@m\PSfilen@me \ctr@ln@m\auxfilen@me
\ctr@ld@f\def\Psb@ginfig#1{\ifcurr@ntPS\else%
    \edef\PSfilen@me{#1}\edef\auxfilen@me{\jobname.anx}%
    \ifpsupdatem@de\ps@critrue\else\openin\frf@g=\PSfilen@me\relax%
    \ifeof\frf@g\ps@critrue\else\ps@crifalse\fi\closein\frf@g\fi%
    \curr@ntPStrue\c@ldefproj\expandafter\setupd@te\defaultupdate:%
    \ifps@cri\initb@undb@x%
    \immediate\openout\fwf@g=\auxfilen@me\initpss@ttings\fi%
    \fi}
\ctr@ld@f\def\Gr@FNb{0}
\ctr@ld@f\def\figforTeXFileno{\Gr@FNb}
\ctr@ld@f\def\figforTeXFigno{0 }
\ctr@ld@f\def\figforTeXnextFigno{1 }
\ctr@ld@f\edef\DefGIfilen@me{\jobname GI.anx}
\ctr@ld@f\def\initpss@ttings{\psreset{arrowhead,curve,first,flowchart,mesh,second,third}%
    \Use@llipsefalse}
\ctr@ld@f\def\B@zierBB@x#1#2(#3,#4,#5,#6){{\c@rre=\t@n\epsil@n
    \v@lmax=#4\advance\v@lmax-#5\v@lmax=\thr@@\v@lmax\advance\v@lmax#6\advance\v@lmax-#3%
    \mili@u=#4\mili@u=-\tw@\mili@u\advance\mili@u#3\advance\mili@u#5%
    \v@lmin=#4\advance\v@lmin-#3\maxim@m{\v@leur}{-\v@lmax}{\v@lmax}%
    \maxim@m{\delt@}{-\mili@u}{\mili@u}\maxim@m{\v@leur}{\v@leur}{\delt@}%
    \maxim@m{\delt@}{-\v@lmin}{\v@lmin}\maxim@m{\v@leur}{\v@leur}{\delt@}%
    \ifdim\v@leur>\c@rre\invers@{\v@leur}{\v@leur}\edef\Uns@rM@x{\repdecn@mb{\v@leur}}%
    \v@lmax=\Uns@rM@x\v@lmax\mili@u=\Uns@rM@x\mili@u\v@lmin=\Uns@rM@x\v@lmin%
    \maxim@m{\v@leur}{-\v@lmax}{\v@lmax}\ifdim\v@leur<\c@rre%
    \maxim@m{\v@leur}{-\mili@u}{\mili@u}\ifdim\v@leur<\c@rre\else%
    \invers@{\mili@u}{\mili@u}\v@leur=-0.5\v@lmin%
    \v@leur=\repdecn@mb{\mili@u}\v@leur\m@jBBB@x{\v@leur}{#1}{#2}(#3,#4,#5,#6)\fi%
    \else\delt@=\repdecn@mb{\mili@u}\mili@u\v@leur=\repdecn@mb{\v@lmax}\v@lmin%
    \advance\delt@-\v@leur\ifdim\delt@<\z@\else\invers@{\v@lmax}{\v@lmax}%
    \edef\Uns@rAp{\repdecn@mb{\v@lmax}}\sqrt@{\delt@}{\delt@}%
    \v@leur=-\mili@u\advance\v@leur\delt@\v@leur=\Uns@rAp\v@leur%
    \m@jBBB@x{\v@leur}{#1}{#2}(#3,#4,#5,#6)%
    \v@leur=-\mili@u\advance\v@leur-\delt@\v@leur=\Uns@rAp\v@leur%
    \m@jBBB@x{\v@leur}{#1}{#2}(#3,#4,#5,#6)\fi\fi\fi}}
\ctr@ld@f\def\m@jBBB@x#1#2#3(#4,#5,#6,#7){{\relax\ifdim#1>\z@\ifdim#1<\p@%
    \edef\T@{\repdecn@mb{#1}}\v@lX=\p@\advance\v@lX-#1\edef\UNmT@{\repdecn@mb{\v@lX}}%
    \v@lX=#4\v@lY=#5\v@lZ=#6\v@lXa=#7\v@lX=\UNmT@\v@lX\advance\v@lX\T@\v@lY%
    \v@lY=\UNmT@\v@lY\advance\v@lY\T@\v@lZ\v@lZ=\UNmT@\v@lZ\advance\v@lZ\T@\v@lXa%
    \v@lX=\UNmT@\v@lX\advance\v@lX\T@\v@lY\v@lY=\UNmT@\v@lY\advance\v@lY\T@\v@lZ%
    \v@lX=\UNmT@\v@lX\advance\v@lX\T@\v@lY%
    \ifcase#2\or\v@lY=#3\or\v@lY=\v@lX\v@lX=#3\fi\b@undb@x{\v@lX}{\v@lY}\fi\fi}}
\ctr@ld@f\def\PsB@zier#1[#2]{{\f@gnewpath%
    \s@mme=\z@\def\list@num{#2,0}\extrairelepremi@r\p@int\de\list@num%
    \PSwrit@cmdS{\p@int}{\c@mmoveto}{\fwf@g}{\X@un}{\Y@un}\p@rtent=#1\bclB@zier}}
\ctr@ld@f\def\bclB@zier{\relax%
    \ifnum\s@mme<\p@rtent\advance\s@mme\@ne\BdingB@xfalse%
    \extrairelepremi@r\p@int\de\list@num\PSwrit@cmdS{\p@int}{}{\fwf@g}{\X@de}{\Y@de}%
    \extrairelepremi@r\p@int\de\list@num\PSwrit@cmdS{\p@int}{}{\fwf@g}{\X@tr}{\Y@tr}%
    \BdingB@xtrue%
    \extrairelepremi@r\p@int\de\list@num\PSwrit@cmdS{\p@int}{\c@mcurveto}{\fwf@g}{\X@qu}{\Y@qu}%
    \B@zierBB@x{1}{\Y@un}(\X@un,\X@de,\X@tr,\X@qu)%
    \B@zierBB@x{2}{\X@un}(\Y@un,\Y@de,\Y@tr,\Y@qu)%
    \edef\X@un{\X@qu}\edef\Y@un{\Y@qu}\bclB@zier\fi}
\ctr@ln@m\psBezier
\ctr@ld@f\def\psBezierDD#1[#2]{\ifcurr@ntPS\ifps@cri%
    \PSc@mment{psBezierDD N arcs=#1, Control points=#2}%
    \iffillm@de\PsB@zier#1[#2]%
    \f@gfill%
    \else\PsB@zier#1[#2]\f@gstroke\fi%
    \PSc@mment{End psBezierDD}\fi\fi}
\ctr@ln@m\et@tpsBezierTD
\ctr@ld@f\def\psBezierTD#1[#2]{\ifcurr@ntPS\ifps@cri\s@uvc@ntr@l\et@tpsBezierTD%
    \PSc@mment{psBezierTD N arcs=#1, Control points=#2}%
    \iffillm@de\PsB@zierTD#1[#2]%
    \f@gfill%
    \else\PsB@zierTD#1[#2]\f@gstroke\fi%
    \PSc@mment{End psBezierTD}\resetc@ntr@l\et@tpsBezierTD\fi\fi}
\ctr@ld@f\def\PsB@zierTD#1[#2]{\ifnum\curr@ntproj<\tw@\PsB@zier#1[#2]\else\PsB@zier@TD#1[#2]\fi}
\ctr@ld@f\def\PsB@zier@TD#1[#2]{{\f@gnewpath%
    \s@mme=\z@\def\list@num{#2,0}\extrairelepremi@r\p@int\de\list@num%
    \let\c@lprojSP=\relax\setc@ntr@l{2}\Figptpr@j-7:/\p@int/%
    \PSwrit@cmd{-7}{\c@mmoveto}{\fwf@g}%
    \loop\ifnum\s@mme<#1\advance\s@mme\@ne\extrairelepremi@r\p@intun\de\list@num%
    \extrairelepremi@r\p@intde\de\list@num\extrairelepremi@r\p@inttr\de\list@num%
    \subB@zierTD\NBz@rcs[\p@int,\p@intun,\p@intde,\p@inttr]\edef\p@int{\p@inttr}\repeat}}
\ctr@ld@f\def\subB@zierTD#1[#2,#3,#4,#5]{\delt@=\p@\divide\delt@\NBz@rcs\v@lmin=\z@%
    {\Figg@tXY{-7}\edef\X@un{\the\v@lX}\edef\Y@un{\the\v@lY}%
    \s@mme=\z@\loop\ifnum\s@mme<#1\advance\s@mme\@ne%
    \v@leur=\v@lmin\advance\v@leur0.33333 \delt@\edef\unti@rs{\repdecn@mb{\v@leur}}%
    \v@leur=\v@lmin\advance\v@leur0.66666 \delt@\edef\deti@rs{\repdecn@mb{\v@leur}}%
    \advance\v@lmin\delt@\edef\trti@rs{\repdecn@mb{\v@lmin}}%
    \figptBezierTD-8::\trti@rs[#2,#3,#4,#5]\Figptpr@j-8:/-8/%
    \c@lsubBzarc\unti@rs,\deti@rs[#2,#3,#4,#5]\BdingB@xfalse%
    \PSwrit@cmdS{-4}{}{\fwf@g}{\X@de}{\Y@de}\PSwrit@cmdS{-3}{}{\fwf@g}{\X@tr}{\Y@tr}%
    \BdingB@xtrue\PSwrit@cmdS{-8}{\c@mcurveto}{\fwf@g}{\X@qu}{\Y@qu}%
    \B@zierBB@x{1}{\Y@un}(\X@un,\X@de,\X@tr,\X@qu)%
    \B@zierBB@x{2}{\X@un}(\Y@un,\Y@de,\Y@tr,\Y@qu)%
    \edef\X@un{\X@qu}\edef\Y@un{\Y@qu}\figptcopyDD-7:/-8/\repeat}}
\ctr@ld@f\def\NBz@rcs{2}
\ctr@ld@f\def\c@lsubBzarc#1,#2[#3,#4,#5,#6]{\figptBezierTD-5::#1[#3,#4,#5,#6]%
    \figptBezierTD-6::#2[#3,#4,#5,#6]\Figptpr@j-4:/-5/\Figptpr@j-5:/-6/%
    \figptscontrolDD-4[-7,-4,-5,-8]}
\ctr@ln@m\pscirc
\ctr@ld@f\def\pscircDD#1(#2){\ifcurr@ntPS\ifps@cri\PSc@mment{pscircDD Center=#1 (Radius=#2)}%
    \psarccircDD#1;#2(0,360)\PSc@mment{End pscircDD}\fi\fi}
\ctr@ld@f\def\pscircTD#1,#2,#3(#4){\ifcurr@ntPS\ifps@cri%
    \PSc@mment{pscircTD Center=#1,P1=#2,P2=#3 (Radius=#4)}%
    \psarccircTD#1,#2,#3;#4(0,360)\PSc@mment{End pscircTD}\fi\fi}
\ctr@ln@m\p@urcent
{\catcode`\%=12\gdef\p@urcent{
\ctr@ld@f\def\PSc@mment#1{\ifpsdebugmode\immediate\write\fwf@g{\p@urcent\space#1}\fi}
\ctr@ln@m\acc@louv \ctr@ln@m\acc@lfer
{\catcode`\[=1\catcode`\{=12\gdef\acc@louv[{}}
{\catcode`\]=2\catcode`\}=12\gdef\acc@lfer{}]]
\ctr@ld@f\def\PSdict@{\ifUse@llipse%
    \immediate\write\fwf@g{/ellipsedict 9 dict def ellipsedict /mtrx matrix put}%
    \immediate\write\fwf@g{/ellipse \acc@louv ellipsedict begin}%
    \immediate\write\fwf@g{ /endangle exch def /startangle exch def}%
    \immediate\write\fwf@g{ /yrad exch def /xrad exch def}%
    \immediate\write\fwf@g{ /rotangle exch def /y exch def /x exch def}%
    \immediate\write\fwf@g{ /savematrix mtrx currentmatrix def}%
    \immediate\write\fwf@g{ x y translate rotangle rotate xrad yrad scale}%
    \immediate\write\fwf@g{ 0 0 1 startangle endangle arc}%
    \immediate\write\fwf@g{ savematrix setmatrix end\acc@lfer def}%
    \fi\PShe@der{EndProlog}}
\ctr@ld@f\def\Pssetc@rve#1=#2|{\keln@mun#1|%
    \def\n@mref{r}\ifx\l@debut\n@mref\pssetroundness{#2}\else
    \immediate\write16{*** Unknown attribute: \BS@ psset curve(..., #1=...)}%
    \fi}
\ctr@ln@m\curv@roundness
\ctr@ld@f\def\pssetroundness#1{\edef\curv@roundness{#1}}
\ctr@ld@f\def\defaultroundness{0.2} 
\ctr@ln@m\pscurve
\ctr@ld@f\def\pscurveDD[#1]{{\ifcurr@ntPS\ifps@cri\PSc@mment{pscurveDD Points=#1}%
    \s@uvc@ntr@l\et@tpscurveDD%
    \iffillm@de\Psc@rveDD\curv@roundness[#1]%
    \f@gfill%
    \else\Psc@rveDD\curv@roundness[#1]\f@gstroke\fi%
    \PSc@mment{End pscurveDD}\resetc@ntr@l\et@tpscurveDD\fi\fi}}
\ctr@ld@f\def\pscurveTD[#1]{{\ifcurr@ntPS\ifps@cri%
    \PSc@mment{pscurveTD Points=#1}\s@uvc@ntr@l\et@tpscurveTD\let\c@lprojSP=\relax%
    \iffillm@de\Psc@rveTD\curv@roundness[#1]%
    \f@gfill%
    \else\Psc@rveTD\curv@roundness[#1]\f@gstroke\fi%
    \PSc@mment{End pscurveTD}\resetc@ntr@l\et@tpscurveTD\fi\fi}}
\ctr@ld@f\def\Psc@rveDD#1[#2]{%
    \def\list@num{#2}\extrairelepremi@r\Ak@\de\list@num%
    \extrairelepremi@r\Ai@\de\list@num\extrairelepremi@r\Aj@\de\list@num%
    \f@gnewpath\PSwrit@cmdS{\Ai@}{\c@mmoveto}{\fwf@g}{\X@un}{\Y@un}%
    \setc@ntr@l{2}\figvectPDD -1[\Ak@,\Aj@]%
    \@ecfor\Ak@:=\list@num\do{\figpttraDD-2:=\Ai@/#1,-1/\BdingB@xfalse%
       \PSwrit@cmdS{-2}{}{\fwf@g}{\X@de}{\Y@de}%
       \figvectPDD -1[\Ai@,\Ak@]\figpttraDD-2:=\Aj@/-#1,-1/%
       \PSwrit@cmdS{-2}{}{\fwf@g}{\X@tr}{\Y@tr}\BdingB@xtrue%
       \PSwrit@cmdS{\Aj@}{\c@mcurveto}{\fwf@g}{\X@qu}{\Y@qu}%
       \B@zierBB@x{1}{\Y@un}(\X@un,\X@de,\X@tr,\X@qu)%
       \B@zierBB@x{2}{\X@un}(\Y@un,\Y@de,\Y@tr,\Y@qu)%
       \edef\X@un{\X@qu}\edef\Y@un{\Y@qu}\edef\Ai@{\Aj@}\edef\Aj@{\Ak@}}}
\ctr@ld@f\def\Psc@rveTD#1[#2]{\ifnum\curr@ntproj<\tw@\Psc@rvePPTD#1[#2]\else\Psc@rveCPTD#1[#2]\fi}
\ctr@ld@f\def\Psc@rvePPTD#1[#2]{\setc@ntr@l{2}%
    \def\list@num{#2}\extrairelepremi@r\Ak@\de\list@num\Figptpr@j-5:/\Ak@/%
    \extrairelepremi@r\Ai@\de\list@num\Figptpr@j-3:/\Ai@/%
    \extrairelepremi@r\Aj@\de\list@num\Figptpr@j-4:/\Aj@/%
    \f@gnewpath\PSwrit@cmdS{-3}{\c@mmoveto}{\fwf@g}{\X@un}{\Y@un}%
    \figvectPDD -1[-5,-4]%
    \@ecfor\Ak@:=\list@num\do{\Figptpr@j-5:/\Ak@/\figpttraDD-2:=-3/#1,-1/%
       \BdingB@xfalse\PSwrit@cmdS{-2}{}{\fwf@g}{\X@de}{\Y@de}%
       \figvectPDD -1[-3,-5]\figpttraDD-2:=-4/-#1,-1/%
       \PSwrit@cmdS{-2}{}{\fwf@g}{\X@tr}{\Y@tr}\BdingB@xtrue%
       \PSwrit@cmdS{-4}{\c@mcurveto}{\fwf@g}{\X@qu}{\Y@qu}%
       \B@zierBB@x{1}{\Y@un}(\X@un,\X@de,\X@tr,\X@qu)%
       \B@zierBB@x{2}{\X@un}(\Y@un,\Y@de,\Y@tr,\Y@qu)%
       \edef\X@un{\X@qu}\edef\Y@un{\Y@qu}\figptcopyDD-3:/-4/\figptcopyDD-4:/-5/}}
\ctr@ld@f\def\Psc@rveCPTD#1[#2]{\setc@ntr@l{2}%
    \def\list@num{#2}\extrairelepremi@r\Ak@\de\list@num%
    \extrairelepremi@r\Ai@\de\list@num\extrairelepremi@r\Aj@\de\list@num%
    \Figptpr@j-7:/\Ai@/%
    \f@gnewpath\PSwrit@cmd{-7}{\c@mmoveto}{\fwf@g}%
    \figvectPTD -9[\Ak@,\Aj@]%
    \@ecfor\Ak@:=\list@num\do{\figpttraTD-10:=\Ai@/#1,-9/%
       \figvectPTD -9[\Ai@,\Ak@]\figpttraTD-11:=\Aj@/-#1,-9/%
       \subB@zierTD\NBz@rcs[\Ai@,-10,-11,\Aj@]\edef\Ai@{\Aj@}\edef\Aj@{\Ak@}}}
\ctr@ld@f\def\psendfig{\ifcurr@ntPS\ifps@cri\immediate\closeout\fwf@g%
    \immediate\openout\fwf@g=\PSfilen@me\relax%
    \ifPDFm@ke\PSBdingB@x\else%
    \immediate\write\fwf@g{\p@urcent\string!PS-Adobe-2.0 EPSF-2.0}%
    \PShe@der{Creator\string: TeX (fig4tex.tex)}%
    \PShe@der{Title\string: \PSfilen@me}%
    \PShe@der{CreationDate\string: \the\day/\the\month/\the\year}%
    \PSBdingB@x%
    \PShe@der{EndComments}\PSdict@\fi%
    \immediate\write\fwf@g{\c@mgsave}%
    \openin\frf@g=\auxfilen@me\c@pypsfile\fwf@g\frf@g\closein\frf@g%
    \immediate\write\fwf@g{\c@mgrestore}%
    \PSc@mment{End of file.}\immediate\closeout\fwf@g%
    \immediate\openout\fwf@g=\auxfilen@me\immediate\closeout\fwf@g%
    \immediate\write16{File \PSfilen@me\space created.}\fi\fi\curr@ntPSfalse\ps@critrue}
\ctr@ld@f\def\PShe@der#1{\immediate\write\fwf@g{\p@urcent\p@urcent#1}}
\ctr@ld@f\def\PSBdingB@x{{\v@lX=\ptT@ptps\c@@rdXmin\v@lY=\ptT@ptps\c@@rdYmin%
     \v@lXa=\ptT@ptps\c@@rdXmax\v@lYa=\ptT@ptps\c@@rdYmax%
     \PShe@der{BoundingBox\string: \repdecn@mb{\v@lX}\space\repdecn@mb{\v@lY}%
     \space\repdecn@mb{\v@lXa}\space\repdecn@mb{\v@lYa}}}}
\ctr@ld@f\def\psfcconnect[#1]{{\ifcurr@ntPS\ifps@cri\PSc@mment{psfcconnect Points=#1}%
    \pssetfillmode{no}\s@uvc@ntr@l\et@tpsfcconnect\resetc@ntr@l{2}%
    \fcc@nnect@[#1]\resetc@ntr@l\et@tpsfcconnect\PSc@mment{End psfcconnect}\fi\fi}}
\ctr@ld@f\def\fcc@nnect@[#1]{\let\N@rm=\n@rmeucDD\def\list@num{#1}%
    \extrairelepremi@r\Ai@\de\list@num\edef\pr@m{\Ai@}\v@leur=\z@\p@rtent=\@ne\c@llgtot%
    \ifcase\fclin@typ@\edef\list@num{[\pr@m,#1,\Ai@}\expandafter\pscurve\list@num]%
    \else\ifdim\fclin@r@d\p@>\z@\Pslin@conge[#1]\else\psline[#1]\fi\fi%
    \v@leur=\@rrowp@s\v@leur\edef\list@num{#1,\Ai@,0}%
    \extrairelepremi@r\Ai@\de\list@num\mili@u=\epsil@n\c@llgpart%
    \advance\mili@u-\epsil@n\advance\mili@u-\delt@\advance\v@leur-\mili@u%
    \ifcase\fclin@typ@\invers@\mili@u\delt@%
    \ifnum\@rrowr@fpt>\z@\advance\delt@-\v@leur\v@leur=\delt@\fi%
    \v@leur=\repdecn@mb\v@leur\mili@u\edef\v@lt{\repdecn@mb\v@leur}%
    \extrairelepremi@r\Ak@\de\list@num%
    \figvectPDD-1[\pr@m,\Aj@]\figpttraDD-6:=\Ai@/\curv@roundness,-1/%
    \figvectPDD-1[\Ak@,\Ai@]\figpttraDD-7:=\Aj@/\curv@roundness,-1/%
    \delt@=\@rrowheadlength\p@\delt@=\C@AHANG\delt@\edef\R@dius{\repdecn@mb{\delt@}}%
    \ifcase\@rrowr@fpt%
    \FigptintercircB@zDD-8::\v@lt,\R@dius[\Ai@,-6,-7,\Aj@]\psarrowheadDD[-5,-8]\else%
    \FigptintercircB@zDD-8::\v@lt,\R@dius[\Aj@,-7,-6,\Ai@]\psarrowheadDD[-8,-5]\fi%
    \else\advance\delt@-\v@leur%
    \p@rtentiere{\p@rtent}{\delt@}\edef\C@efun{\the\p@rtent}%
    \p@rtentiere{\p@rtent}{\v@leur}\edef\C@efde{\the\p@rtent}%
    \figptbaryDD-5:[\Ai@,\Aj@;\C@efun,\C@efde]\ifcase\@rrowr@fpt%
    \delt@=\@rrowheadlength\unit@\delt@=\C@AHANG\delt@\edef\t@ille{\repdecn@mb{\delt@}}%
    \figvectPDD-2[\Ai@,\Aj@]\vecunit@{-2}{-2}\figpttraDD-5:=-5/\t@ille,-2/\fi%
    \psarrowheadDD[\Ai@,-5]\fi}
\ctr@ld@f\def\c@llgtot{\@ecfor\Aj@:=\list@num\do{\figvectP-1[\Ai@,\Aj@]\N@rm\delt@{-1}%
    \advance\v@leur\delt@\advance\p@rtent\@ne\edef\Ai@{\Aj@}}}
\ctr@ld@f\def\c@llgpart{\extrairelepremi@r\Aj@\de\list@num\figvectP-1[\Ai@,\Aj@]\N@rm\delt@{-1}%
    \advance\mili@u\delt@\ifdim\mili@u<\v@leur\edef\pr@m{\Ai@}\edef\Ai@{\Aj@}\c@llgpart\fi}
\ctr@ld@f\def\Pslin@conge[#1]{\ifnum\p@rtent>\tw@{\def\list@num{#1}%
    \extrairelepremi@r\Ai@\de\list@num\extrairelepremi@r\Aj@\de\list@num%
    \figptcopy-6:/\Ai@/\figvectP-3[\Ai@,\Aj@]\vecunit@{-3}{-3}\v@lmax=\result@t%
    \@ecfor\Ak@:=\list@num\do{\figvectP-4[\Aj@,\Ak@]\vecunit@{-4}{-4}%
    \minim@m\v@lmin\v@lmax\result@t\v@lmax=\result@t%
    \det@rm\delt@[-3,-4]\maxim@m\mili@u{\delt@}{-\delt@}\ifdim\mili@u>\Cepsil@n%
    \ifdim\delt@>\z@\figgetangleDD\Angl@[\Aj@,\Ak@,\Ai@]\else%
    \figgetangleDD\Angl@[\Aj@,\Ai@,\Ak@]\fi%
    \v@leur=\PI@deg\advance\v@leur-\Angl@\p@\divide\v@leur\tw@%
    \edef\Angl@{\repdecn@mb\v@leur}\c@ssin{\C@}{\S@}{\Angl@}\v@leur=\fclin@r@d\unit@%
    \v@leur=\S@\v@leur\mili@u=\C@\p@\invers@\mili@u\mili@u%
    \v@leur=\repdecn@mb{\mili@u}\v@leur%
    \minim@m\v@leur\v@leur\v@lmin\edef\t@ille{\repdecn@mb{\v@leur}}%
    \figpttra-5:=\Aj@/-\t@ille,-3/\psline[-6,-5]\figpttra-6:=\Aj@/\t@ille,-4/%
    \figvectNVDD-3[-3]\figvectNVDD-8[-4]\inters@cDD-7:[-5,-3;-6,-8]%
    \ifdim\delt@>\z@\psarccircP-7;\fclin@r@d[-5,-6]\else\psarccircP-7;\fclin@r@d[-6,-5]\fi%
    \else\psline[-6,\Aj@]\figptcopy-6:/\Aj@/\fi
    \edef\Ai@{\Aj@}\edef\Aj@{\Ak@}\figptcopy-3:/-4/}\psline[-6,\Aj@]}\else\psline[#1]\fi}
\ctr@ld@f\def\psfcnode[#1]#2{{\ifcurr@ntPS\ifps@cri\PSc@mment{psfcnode Points=#1}%
    \s@uvc@ntr@l\et@tpsfcnode\resetc@ntr@l{2}%
    \def\t@xt@{#2}\ifx\t@xt@\empty\def\g@tt@xt{\setbox\Gb@x=\hbox{\Figg@tT{\p@int}}}%
    \else\def\g@tt@xt{\setbox\Gb@x=\hbox{#2}}\fi%
    \v@lmin=\h@rdfcXp@dd\advance\v@lmin\Xp@dd\unit@\multiply\v@lmin\tw@%
    \v@lmax=\h@rdfcYp@dd\advance\v@lmax\Yp@dd\unit@\multiply\v@lmax\tw@%
    \Figv@ctCreg-8(\unit@,-\unit@)\def\list@num{#1}%
    \delt@=\curr@ntwidth bp\divide\delt@\tw@%
    \fcn@de\PSc@mment{End psfcnode}\resetc@ntr@l\et@tpsfcnode\fi\fi}}
\ctr@ld@f\def\d@butn@de{\g@tt@xt\v@lX=\wd\Gb@x%
    \v@lY=\ht\Gb@x\advance\v@lY\dp\Gb@x\advance\v@lX\v@lmin\advance\v@lY\v@lmax}
\ctr@ld@f\def\fcn@deE{%
    \@ecfor\p@int:=\list@num\do{\d@butn@de\v@lX=\unssqrttw@\v@lX\v@lY=\unssqrttw@\v@lY%
    \ifdim\thickn@ss\p@>\z@
    \v@lXa=\v@lX\advance\v@lXa\delt@\v@lXa=\ptT@unit@\v@lXa\edef\XR@d{\repdecn@mb\v@lXa}%
    \v@lYa=\v@lY\advance\v@lYa\delt@\v@lYa=\ptT@unit@\v@lYa\edef\YR@d{\repdecn@mb\v@lYa}%
    \arct@n\v@leur(\v@lXa,\v@lYa)\v@leur=\rdT@deg\v@leur\edef\@nglde{\repdecn@mb\v@leur}%
    {\c@lptellDD-2::\p@int;\XR@d,\YR@d(\@nglde)}
    \advance\v@leur-\PI@deg\edef\@nglun{\repdecn@mb\v@leur}%
    {\c@lptellDD-3::\p@int;\XR@d,\YR@d(\@nglun)}%
    \figptstra-6=-3,-2,\p@int/\thickn@ss,-8/\pssetfillmode{yes}\us@secondC@lor%
    \psline[-2,-3,-6,-5]\psarcell-4;\XR@d,\YR@d(\@nglun,\@nglde,0)\fi
    \v@lX=\ptT@unit@\v@lX\v@lY=\ptT@unit@\v@lY%
    \edef\XR@d{\repdecn@mb\v@lX}\edef\YR@d{\repdecn@mb\v@lY}%
    \pssetfillmode{yes}\us@thirdC@lor\psarcell\p@int;\XR@d,\YR@d(0,360,0)%
    \pssetfillmode{no}\us@primarC@lor\psarcell\p@int;\XR@d,\YR@d(0,360,0)}}
\ctr@ld@f\def\fcn@deL{\delt@=\ptT@unit@\delt@\edef\t@ille{\repdecn@mb\delt@}%
    \@ecfor\p@int:=\list@num\do{\Figg@tXYa{\p@int}\d@butn@de%
    \ifdim\v@lX>\v@lY\itis@Ktrue\else\itis@Kfalse\fi%
    \advance\v@lXa-\v@lX\Figp@intreg-1:(\v@lXa,\v@lYa)%
    \advance\v@lXa\v@lX\advance\v@lYa-\v@lY\Figp@intreg-2:(\v@lXa,\v@lYa)%
    \advance\v@lXa\v@lX\advance\v@lYa\v@lY\Figp@intreg-3:(\v@lXa,\v@lYa)%
    \advance\v@lXa-\v@lX\advance\v@lYa\v@lY\Figp@intreg-4:(\v@lXa,\v@lYa)%
    \ifdim\thickn@ss\p@>\z@\Figg@tXYa{\p@int}\pssetfillmode{yes}\us@secondC@lor
    \c@lpt@xt{-1}{-4}\c@lpt@xt@\v@lXa\v@lYa\v@lX\v@lY\c@rre\delt@%
    \Figp@intregDD-9:(\v@lZ,\v@lYa)\Figp@intregDD-11:(\v@lZa,\v@lYa)%
    \c@lpt@xt{-4}{-3}\c@lpt@xt@\v@lYa\v@lXa\v@lY\v@lX\delt@\c@rre%
    \Figp@intregDD-12:(\v@lXa,\v@lZ)\Figp@intregDD-10:(\v@lXa,\v@lZa)%
    \ifitis@K\figptstra-7=-9,-10,-11/\thickn@ss,-8/\psline[-9,-11,-5,-6,-7]\else%
    \figptstra-7=-10,-11,-12/\thickn@ss,-8/\psline[-10,-12,-5,-6,-7]\fi\fi
    \pssetfillmode{yes}\us@thirdC@lor\psline[-1,-2,-3,-4]%
    \pssetfillmode{no}\us@primarC@lor\psline[-1,-2,-3,-4,-1]}}
\ctr@ld@f\def\c@lpt@xt#1#2{\figvectN-7[#1,#2]\vecunit@{-7}{-7}\figpttra-5:=#1/\t@ille,-7/%
    \figvectP-7[#1,#2]\Figg@tXY{-7}\c@rre=\v@lX\delt@=\v@lY\Figg@tXY{-5}}
\ctr@ld@f\def\c@lpt@xt@#1#2#3#4#5#6{\v@lZ=#6\invers@{\v@lZ}{\v@lZ}\v@leur=\repdecn@mb{#5}\v@lZ%
    \v@lZ=#2\advance\v@lZ-#4\mili@u=\repdecn@mb{\v@leur}\v@lZ%
    \v@lZ=#3\advance\v@lZ\mili@u\v@lZa=-\v@lZ\advance\v@lZa\tw@#1}
\ctr@ld@f\def\fcn@deR{\@ecfor\p@int:=\list@num\do{\Figg@tXYa{\p@int}\d@butn@de%
    \advance\v@lXa-0.5\v@lX\advance\v@lYa-0.5\v@lY\Figp@intreg-1:(\v@lXa,\v@lYa)%
    \advance\v@lXa\v@lX\Figp@intreg-2:(\v@lXa,\v@lYa)%
    \advance\v@lYa\v@lY\Figp@intreg-3:(\v@lXa,\v@lYa)%
    \advance\v@lXa-\v@lX\Figp@intreg-4:(\v@lXa,\v@lYa)%
    \ifdim\thickn@ss\p@>\z@\pssetfillmode{yes}\us@secondC@lor
    \Figv@ctCreg-5(-\delt@,-\delt@)\figpttra-9:=-1/1,-5/%
    \Figv@ctCreg-5(\delt@,-\delt@)\figpttra-10:=-2/1,-5/%
    \Figv@ctCreg-5(\delt@,\delt@)\figpttra-11:=-3/1,-5/%
    \figptstra-7=-9,-10,-11/\thickn@ss,-8/\psline[-9,-11,-5,-6,-7]\fi
    \pssetfillmode{yes}\us@thirdC@lor\psline[-1,-2,-3,-4]%
    \pssetfillmode{no}\us@primarC@lor\psline[-1,-2,-3,-4,-1]}}
\ctr@ln@m\@rrowp@s
\ctr@ln@m\Xp@dd     \ctr@ln@m\Yp@dd
\ctr@ln@m\fclin@r@d \ctr@ln@m\thickn@ss
\ctr@ld@f\def\Pssetfl@wchart#1=#2|{\keln@mtr#1|%
    \def\n@mref{arr}\ifx\l@debut\n@mref\expandafter\keln@mtr\l@suite|%
     \def\n@mref{owp}\ifx\l@debut\n@mref\edef\@rrowp@s{#2}\else
     \def\n@mref{owr}\ifx\l@debut\n@mref\setfcr@fpt#2|\else
     \immediate\write16{*** Unknown attribute: \BS@ psset flowchart(..., #1=...)}%
     \fi\fi\else%
    \def\n@mref{lin}\ifx\l@debut\n@mref\setfccurv@#2|\else
    \def\n@mref{pad}\ifx\l@debut\n@mref\edef\Xp@dd{#2}\edef\Yp@dd{#2}\else
    \def\n@mref{rad}\ifx\l@debut\n@mref\edef\fclin@r@d{#2}\else
    \def\n@mref{sha}\ifx\l@debut\n@mref\setfcshap@#2|\else
    \def\n@mref{thi}\ifx\l@debut\n@mref\edef\thickn@ss{#2}\else
    \def\n@mref{xpa}\ifx\l@debut\n@mref\edef\Xp@dd{#2}\else
    \def\n@mref{ypa}\ifx\l@debut\n@mref\edef\Yp@dd{#2}\else
    \immediate\write16{*** Unknown attribute: \BS@ psset flowchart(..., #1=...)}%
    \fi\fi\fi\fi\fi\fi\fi\fi}
\ctr@ln@m\@rrowr@fpt \ctr@ln@m\fclin@typ@
\ctr@ld@f\def\setfcr@fpt#1#2|{\if#1e\def\@rrowr@fpt{1}\else\def\@rrowr@fpt{0}\fi}
\ctr@ld@f\def\setfccurv@#1#2|{\if#1c\def\fclin@typ@{0}\else\def\fclin@typ@{1}\fi}
\ctr@ln@m\h@rdfcXp@dd \ctr@ln@m\h@rdfcYp@dd
\ctr@ln@m\fcn@de \ctr@ln@m\fcsh@pe
\ctr@ld@f\def\setfcshap@#1#2|{%
    \if#1e\let\fcn@de=\fcn@deE\def\h@rdfcXp@dd{4pt}\def\h@rdfcYp@dd{4pt}%
     \edef\fcsh@pe{ellipse}\else%
    \if#1l\let\fcn@de=\fcn@deL\def\h@rdfcXp@dd{4pt}\def\h@rdfcYp@dd{4pt}%
     \edef\fcsh@pe{lozenge}\else%
          \let\fcn@de=\fcn@deR\def\h@rdfcXp@dd{6pt}\def\h@rdfcYp@dd{6pt}%
     \edef\fcsh@pe{rectangle}\fi\fi}
\ctr@ld@f\def\psline[#1]{{\ifcurr@ntPS\ifps@cri\PSc@mment{psline Points=#1}%
    \let\pslign@=\Pslign@P\Pslin@{#1}\PSc@mment{End psline}\fi\fi}}
\ctr@ld@f\def\pslineF#1{{\ifcurr@ntPS\ifps@cri\PSc@mment{pslineF Filename=#1}%
    \let\pslign@=\Pslign@F\Pslin@{#1}\PSc@mment{End pslineF}\fi\fi}}
\ctr@ld@f\def\pslineC(#1){{\ifcurr@ntPS\ifps@cri\PSc@mment{pslineC}%
    \let\pslign@=\Pslign@C\Pslin@{#1}\PSc@mment{End pslineC}\fi\fi}}
\ctr@ld@f\def\Pslin@#1{\iffillm@de\pslign@{#1}%
    \f@gfill%
    \else\pslign@{#1}\ifx\derp@int\premp@int%
    \f@gclosestroke%
    \else\f@gstroke\fi\fi}
\ctr@ld@f\def\Pslign@P#1{\def\list@num{#1}\extrairelepremi@r\p@int\de\list@num%
    \edef\premp@int{\p@int}\f@gnewpath%
    \PSwrit@cmd{\p@int}{\c@mmoveto}{\fwf@g}%
    \@ecfor\p@int:=\list@num\do{\PSwrit@cmd{\p@int}{\c@mlineto}{\fwf@g}%
    \edef\derp@int{\p@int}}}
\ctr@ld@f\def\Pslign@F#1{\s@uvc@ntr@l\et@tPslign@F\setc@ntr@l{2}\openin\frf@g=#1\relax%
    \ifeof\frf@g\message{*** File #1 not found !}\end\else%
    \read\frf@g to\tr@c\edef\premp@int{\tr@c}\expandafter\extr@ctCF\tr@c:%
    \f@gnewpath\PSwrit@cmd{-1}{\c@mmoveto}{\fwf@g}%
    \loop\read\frf@g to\tr@c\ifeof\frf@g\mored@tafalse\else\mored@tatrue\fi%
    \ifmored@ta\expandafter\extr@ctCF\tr@c:\PSwrit@cmd{-1}{\c@mlineto}{\fwf@g}%
    \edef\derp@int{\tr@c}\repeat\fi\closein\frf@g\resetc@ntr@l\et@tPslign@F}
\ctr@ln@m\extr@ctCF
\ctr@ld@f\def\extr@ctCFDD#1 #2:{\v@lX=#1\unit@\v@lY=#2\unit@\Figp@intregDD-1:(\v@lX,\v@lY)}
\ctr@ld@f\def\extr@ctCFTD#1 #2 #3:{\v@lX=#1\unit@\v@lY=#2\unit@\v@lZ=#3\unit@%
    \Figp@intregTD-1:(\v@lX,\v@lY,\v@lZ)}
\ctr@ld@f\def\Pslign@C#1{\s@uvc@ntr@l\et@tPslign@C\setc@ntr@l{2}%
    \def\list@num{#1}\extrairelepremi@r\p@int\de\list@num%
    \edef\premp@int{\p@int}\f@gnewpath%
    \expandafter\Pslign@C@\p@int:\PSwrit@cmd{-1}{\c@mmoveto}{\fwf@g}%
    \@ecfor\p@int:=\list@num\do{\expandafter\Pslign@C@\p@int:%
    \PSwrit@cmd{-1}{\c@mlineto}{\fwf@g}\edef\derp@int{\p@int}}%
    \resetc@ntr@l\et@tPslign@C}
\ctr@ld@f\def\Pslign@C@#1 #2:{{\def\t@xt@{#1}\ifx\t@xt@\empty\Pslign@C@#2:
    \else\extr@ctCF#1 #2:\fi}}
\ctr@ln@m\c@ntrolmesh
\ctr@ld@f\def\Pssetm@sh#1=#2|{\keln@mun#1|%
    \def\n@mref{d}\ifx\l@debut\n@mref\pssetmeshdiag{#2}\else
    \immediate\write16{*** Unknown attribute: \BS@ psset mesh(..., #1=...)}%
    \fi}
\ctr@ld@f\def\pssetmeshdiag#1{\edef\c@ntrolmesh{#1}}
\ctr@ld@f\def\defaultmeshdiag{0}    
\ctr@ld@f\def\psmesh#1,#2[#3,#4,#5,#6]{{\ifcurr@ntPS\ifps@cri%
    \PSc@mment{psmesh N1=#1, N2=#2, Quadrangle=[#3,#4,#5,#6]}%
    \s@uvc@ntr@l\et@tpsmesh\Pss@tsecondSt\setc@ntr@l{2}%
    \ifnum#1>\@ne\Psmeshp@rt#1[#3,#4,#5,#6]\fi%
    \ifnum#2>\@ne\Psmeshp@rt#2[#4,#5,#6,#3]\fi%
    \ifnum\c@ntrolmesh>\z@\Psmeshdi@g#1,#2[#3,#4,#5,#6]\fi%
    \ifnum\c@ntrolmesh<\z@\Psmeshdi@g#2,#1[#4,#5,#6,#3]\fi\Psrest@reSt%
    \psline[#3,#4,#5,#6,#3]\PSc@mment{End psmesh}\resetc@ntr@l\et@tpsmesh\fi\fi}}
\ctr@ld@f\def\Psmeshp@rt#1[#2,#3,#4,#5]{{\l@mbd@un=\@ne\l@mbd@de=#1\loop%
    \ifnum\l@mbd@un<#1\advance\l@mbd@de\m@ne\figptbary-1:[#2,#3;\l@mbd@de,\l@mbd@un]%
    \figptbary-2:[#5,#4;\l@mbd@de,\l@mbd@un]\psline[-1,-2]\advance\l@mbd@un\@ne\repeat}}
\ctr@ld@f\def\Psmeshdi@g#1,#2[#3,#4,#5,#6]{\figptcopy-2:/#3/\figptcopy-3:/#6/%
    \l@mbd@un=\z@\l@mbd@de=#1\loop\ifnum\l@mbd@un<#1%
    \advance\l@mbd@un\@ne\advance\l@mbd@de\m@ne\figptcopy-1:/-2/\figptcopy-4:/-3/%
    \figptbary-2:[#3,#4;\l@mbd@de,\l@mbd@un]%
    \figptbary-3:[#6,#5;\l@mbd@de,\l@mbd@un]\Psmeshdi@gp@rt#2[-1,-2,-3,-4]\repeat}
\ctr@ld@f\def\Psmeshdi@gp@rt#1[#2,#3,#4,#5]{{\l@mbd@un=\z@\l@mbd@de=#1\loop%
    \ifnum\l@mbd@un<#1\figptbary-5:[#2,#5;\l@mbd@de,\l@mbd@un]%
    \advance\l@mbd@de\m@ne\advance\l@mbd@un\@ne%
    \figptbary-6:[#3,#4;\l@mbd@de,\l@mbd@un]\psline[-5,-6]\repeat}}
\ctr@ln@m\psnormal
\ctr@ld@f\def\psnormalDD#1,#2[#3,#4]{{\ifcurr@ntPS\ifps@cri%
    \PSc@mment{psnormal Length=#1, Lambda=#2 [Pt1,Pt2]=[#3,#4]}%
    \s@uvc@ntr@l\et@tpsnormal\resetc@ntr@l{2}\figptendnormal-6::#1,#2[#3,#4]%
    \figptcopyDD-5:/-1/\psarrow[-5,-6]%
    \PSc@mment{End psnormal}\resetc@ntr@l\et@tpsnormal\fi\fi}}
\ctr@ld@f\def\psreset#1{\trtlis@rg{#1}{\Psreset@}}
\ctr@ld@f\def\Psreset@#1|{\keln@mde#1|%
    \def\n@mref{ar}\ifx\l@debut\n@mref\psresetarrowhead\else
    \def\n@mref{cu}\ifx\l@debut\n@mref\psset curve(roundness=\defaultroundness)\else
    \def\n@mref{fi}\ifx\l@debut\n@mref\psset (color=\defaultcolor,dash=\defaultdash,%
         fill=\defaultfill,join=\defaultjoin,width=\defaultwidth)\else
    \def\n@mref{fl}\ifx\l@debut\n@mref\psset flowchart(arrowp=\defaultfcarrowposition,%
	arrowr=\defaultfcarrowrefpt,line=\defaultfcline,xpadd=\defaultfcxpadding,%
	ypadd=\defaultfcypadding,radius=\defaultfcradius,shape=\defaultfcshape,%
	thick=\defaultfcthickness)\else
    \def\n@mref{me}\ifx\l@debut\n@mref\psset mesh(diag=\defaultmeshdiag)\else
    \def\n@mref{se}\ifx\l@debut\n@mref\psresetsecondsettings\else
    \def\n@mref{th}\ifx\l@debut\n@mref\psset third(color=\defaultthirdcolor)\else
    \immediate\write16{*** Unknown keyword #1 (\BS@ psreset).}%
    \fi\fi\fi\fi\fi\fi\fi}
\ctr@ld@f\def\psset#1(#2){\def\t@xt@{#1}\ifx\t@xt@\empty\trtlis@rg{#2}{\Pssetf@rst}
    \else\keln@mde#1|%
    \def\n@mref{ar}\ifx\l@debut\n@mref\trtlis@rg{#2}{\Psset@rrowhe@d}\else
    \def\n@mref{cu}\ifx\l@debut\n@mref\trtlis@rg{#2}{\Pssetc@rve}\else
    \def\n@mref{fi}\ifx\l@debut\n@mref\trtlis@rg{#2}{\Pssetf@rst}\else
    \def\n@mref{fl}\ifx\l@debut\n@mref\trtlis@rg{#2}{\Pssetfl@wchart}\else
    \def\n@mref{me}\ifx\l@debut\n@mref\trtlis@rg{#2}{\Pssetm@sh}\else
    \def\n@mref{se}\ifx\l@debut\n@mref\trtlis@rg{#2}{\Pssets@cond}\else
    \def\n@mref{th}\ifx\l@debut\n@mref\trtlis@rg{#2}{\Pssetth@rd}\else
    \immediate\write16{*** Unknown keyword: \BS@ psset #1(...)}%
    \fi\fi\fi\fi\fi\fi\fi\fi}
\ctr@ld@f\def\pssetdefault#1(#2){\ifcurr@ntPS\immediate\write16{*** \BS@ pssetdefault is ignored
    inside a \BS@ psbeginfig-\BS@ psendfig block.}%
    \immediate\write16{*** It must be called before \BS@ psbeginfig.}\else%
    \def\t@xt@{#1}\ifx\t@xt@\empty\trtlis@rg{#2}{\Pssd@f@rst}\else\keln@mde#1|%
    \def\n@mref{ar}\ifx\l@debut\n@mref\trtlis@rg{#2}{\Pssd@@rrowhe@d}\else
    \def\n@mref{cu}\ifx\l@debut\n@mref\trtlis@rg{#2}{\Pssd@c@rve}\else
    \def\n@mref{fi}\ifx\l@debut\n@mref\trtlis@rg{#2}{\Pssd@f@rst}\else
    \def\n@mref{fl}\ifx\l@debut\n@mref\trtlis@rg{#2}{\Pssd@fl@wchart}\else
    \def\n@mref{me}\ifx\l@debut\n@mref\trtlis@rg{#2}{\Pssd@m@sh}\else
    \def\n@mref{se}\ifx\l@debut\n@mref\trtlis@rg{#2}{\Pssd@s@cond}\else
    \def\n@mref{th}\ifx\l@debut\n@mref\trtlis@rg{#2}{\Pssd@th@rd}\else
    \immediate\write16{*** Unknown keyword: \BS@ pssetdefault #1(...)}%
    \fi\fi\fi\fi\fi\fi\fi\fi\initpss@ttings\fi}
\ctr@ld@f\def\Pssd@f@rst#1=#2|{\keln@mun#1|%
    \def\n@mref{c}\ifx\l@debut\n@mref\edef\defaultcolor{#2}\else
    \def\n@mref{d}\ifx\l@debut\n@mref\edef\defaultdash{#2}\else
    \def\n@mref{f}\ifx\l@debut\n@mref\edef\defaultfill{#2}\else
    \def\n@mref{j}\ifx\l@debut\n@mref\edef\defaultjoin{#2}\else
    \def\n@mref{u}\ifx\l@debut\n@mref\edef\defaultupdate{#2}\pssetupdate{#2}\else
    \def\n@mref{w}\ifx\l@debut\n@mref\edef\defaultwidth{#2}\else
    \immediate\write16{*** Unknown attribute: \BS@ pssetdefault (..., #1=...)}%
    \fi\fi\fi\fi\fi\fi}
\ctr@ld@f\def\Pssd@@rrowhe@d#1=#2|{\keln@mun#1|%
    \def\n@mref{a}\ifx\l@debut\n@mref\edef\defaultarrowheadangle{#2}\else
    \def\n@mref{f}\ifx\l@debut\n@mref\edef\defaultarrowheadangle{#2}\else
    \def\n@mref{l}\ifx\l@debut\n@mref\y@tiunit{#2}\ifunitpr@sent%
     \edef\defaulth@rdahlength{#2}\else\edef\defaulth@rdahlength{#2pt}%
     \message{*** \BS@ pssetdefault (..., #1=#2, ...) : unit is missing, pt is assumed.}%
     \fi\else
    \def\n@mref{o}\ifx\l@debut\n@mref\edef\defaultarrowheadout{#2}\else
    \def\n@mref{r}\ifx\l@debut\n@mref\edef\defaultarrowheadratio{#2}\else
    \immediate\write16{*** Unknown attribute: \BS@ pssetdefault arrowhead(..., #1=...)}%
    \fi\fi\fi\fi\fi}
\ctr@ld@f\def\Pssd@c@rve#1=#2|{\keln@mun#1|%
    \def\n@mref{r}\ifx\l@debut\n@mref\edef\defaultroundness{#2}\else%
    \immediate\write16{*** Unknown attribute: \BS@ pssetdefault curve(..., #1=...)}%
    \fi}
\ctr@ld@f\def\Pssd@fl@wchart#1=#2|{\keln@mtr#1|%
    \def\n@mref{arr}\ifx\l@debut\n@mref\expandafter\keln@mtr\l@suite|%
     \def\n@mref{owp}\ifx\l@debut\n@mref\edef\defaultfcarrowposition{#2}\else
     \def\n@mref{owr}\ifx\l@debut\n@mref\edef\defaultfcarrowrefpt{#2}\else
     \immediate\write16{*** Unknown attribute: \BS@ pssetdefault flowchart(..., #1=...)}%
     \fi\fi\else%
    \def\n@mref{lin}\ifx\l@debut\n@mref\edef\defaultfcline{#2}\else
    \def\n@mref{pad}\ifx\l@debut\n@mref\edef\defaultfcxpadding{#2}%
                    \edef\defaultfcypadding{#2}\else
    \def\n@mref{rad}\ifx\l@debut\n@mref\edef\defaultfcradius{#2}\else
    \def\n@mref{sha}\ifx\l@debut\n@mref\edef\defaultfcshape{#2}\else
    \def\n@mref{thi}\ifx\l@debut\n@mref\edef\defaultfcthickness{#2}\else
    \def\n@mref{xpa}\ifx\l@debut\n@mref\edef\defaultfcxpadding{#2}\else
    \def\n@mref{ypa}\ifx\l@debut\n@mref\edef\defaultfcypadding{#2}\else
    \immediate\write16{*** Unknown attribute: \BS@ pssetdefault flowchart(..., #1=...)}%
    \fi\fi\fi\fi\fi\fi\fi\fi}
\ctr@ld@f\def\defaultfcarrowposition{0.5}
\ctr@ld@f\def\defaultfcarrowrefpt{start}
\ctr@ld@f\def\defaultfcline{polygon}
\ctr@ld@f\def\defaultfcradius{0}
\ctr@ld@f\def\defaultfcshape{rectangle}
\ctr@ld@f\def\defaultfcthickness{0}
\ctr@ld@f\def\defaultfcxpadding{0}
\ctr@ld@f\def\defaultfcypadding{0}
\ctr@ld@f\def\Pssd@m@sh#1=#2|{\keln@mun#1|%
    \def\n@mref{d}\ifx\l@debut\n@mref\edef\defaultmeshdiag{#2}\else%
    \immediate\write16{*** Unknown attribute: \BS@ pssetdefault mesh(..., #1=...)}%
    \fi}
\ctr@ld@f\def\Pssd@s@cond#1=#2|{\keln@mun#1|%
    \def\n@mref{c}\ifx\l@debut\n@mref\edef\defaultsecondcolor{#2}\else%
    \def\n@mref{d}\ifx\l@debut\n@mref\edef\defaultseconddash{#2}\else%
    \def\n@mref{w}\ifx\l@debut\n@mref\edef\defaultsecondwidth{#2}\else%
    \immediate\write16{*** Unknown attribute: \BS@ pssetdefault second(..., #1=...)}%
    \fi\fi\fi}
\ctr@ld@f\def\Pssd@th@rd#1=#2|{\keln@mun#1|%
    \def\n@mref{c}\ifx\l@debut\n@mref\edef\defaultthirdcolor{#2}\else%
    \immediate\write16{*** Unknown attribute: \BS@ pssetdefault third(..., #1=...)}%
    \fi}
\ctr@ln@w{newif}\iffillm@de
\ctr@ld@f\def\pssetfillmode#1{\expandafter\setfillm@de#1:}
\ctr@ld@f\def\setfillm@de#1#2:{\if#1n\fillm@defalse\else\fillm@detrue\fi}
\ctr@ld@f\def\defaultfill{no}     
\ctr@ln@w{newif}\ifpsupdatem@de
\ctr@ld@f\def\pssetupdate#1{\ifcurr@ntPS\immediate\write16{*** \BS@ pssetupdate is ignored inside a
     \BS@ psbeginfig-\BS@ psendfig block.}%
    \immediate\write16{*** It must be called before \BS@ psbeginfig.}%
    \else\expandafter\setupd@te#1:\fi}
\ctr@ld@f\def\setupd@te#1#2:{\if#1n\psupdatem@defalse\else\psupdatem@detrue\fi}
\ctr@ld@f\def\defaultupdate{no}     
\ctr@ln@m\curr@ntcolor \ctr@ln@m\curr@ntcolorc@md
\ctr@ld@f\def\Pssetc@lor#1{\ifps@cri\result@tent=\@ne\expandafter\c@lnbV@l#1 :%
    \def\curr@ntcolor{}\def\curr@ntcolorc@md{}%
    \ifcase\result@tent\or\pssetgray{#1}\or\or\pssetrgb{#1}\or\pssetcmyk{#1}\fi\fi}
\ctr@ln@m\curr@ntcolorc@mdStroke
\ctr@ld@f\def\pssetcmyk#1{\ifps@cri\def\curr@ntcolor{#1}\def\curr@ntcolorc@md{\c@msetcmykcolor}%
    \def\curr@ntcolorc@mdStroke{\c@msetcmykcolorStroke}%
    \ifcurr@ntPS\PSc@mment{pssetcmyk Color=#1}\us@primarC@lor\fi\fi}
\ctr@ld@f\def\pssetrgb#1{\ifps@cri\def\curr@ntcolor{#1}\def\curr@ntcolorc@md{\c@msetrgbcolor}%
    \def\curr@ntcolorc@mdStroke{\c@msetrgbcolorStroke}%
    \ifcurr@ntPS\PSc@mment{pssetrgb Color=#1}\us@primarC@lor\fi\fi}
\ctr@ld@f\def\pssetgray#1{\ifps@cri\def\curr@ntcolor{#1}\def\curr@ntcolorc@md{\c@msetgray}%
    \def\curr@ntcolorc@mdStroke{\c@msetgrayStroke}%
    \ifcurr@ntPS\PSc@mment{pssetgray Gray level=#1}\us@primarC@lor\fi\fi}
\ctr@ln@m\fillc@md
\ctr@ld@f\def\us@primarC@lor{\immediate\write\fwf@g{\d@fprimarC@lor}%
    \let\fillc@md=\prfillc@md}
\ctr@ld@f\def\prfillc@md{\d@fprimarC@lor\space\c@mfill}
\ctr@ld@f\def\defaultcolor{0}       
\ctr@ld@f\def\c@lnbV@l#1 #2:{\def\t@xt@{#1}\relax\ifx\t@xt@\empty\c@lnbV@l#2:
    \else\c@lnbV@l@#1 #2:\fi}
\ctr@ld@f\def\c@lnbV@l@#1 #2:{\def\t@xt@{#2}\ifx\t@xt@\empty%
    \def\t@xt@{#1}\ifx\t@xt@\empty\advance\result@tent\m@ne\fi
    \else\advance\result@tent\@ne\c@lnbV@l@#2:\fi}
\ctr@ld@f\def\Blackcmyk{0 0 0 1}
\ctr@ld@f\def\Whitecmyk{0 0 0 0}
\ctr@ld@f\def\Cyancmyk{1 0 0 0}
\ctr@ld@f\def\Magentacmyk{0 1 0 0}
\ctr@ld@f\def\Yellowcmyk{0 0 1 0}
\ctr@ld@f\def\Redcmyk{0 1 1 0}
\ctr@ld@f\def\Greencmyk{1 0 1 0}
\ctr@ld@f\def\Bluecmyk{1 1 0 0}
\ctr@ld@f\def\Graycmyk{0 0 0 0.50}
\ctr@ld@f\def\BrickRedcmyk{0 0.89 0.94 0.28} 
\ctr@ld@f\def\Browncmyk{0 0.81 1 0.60} 
\ctr@ld@f\def\ForestGreencmyk{0.91 0 0.88 0.12} 
\ctr@ld@f\def\Goldenrodcmyk{ 0 0.10 0.84 0} 
\ctr@ld@f\def\Marooncmyk{0 0.87 0.68 0.32} 
\ctr@ld@f\def\Orangecmyk{0 0.61 0.87 0} 
\ctr@ld@f\def\Purplecmyk{0.45 0.86 0 0} 
\ctr@ld@f\def\RoyalBluecmyk{1. 0.50 0 0} 
\ctr@ld@f\def\Violetcmyk{0.79 0.88 0 0} 
\ctr@ld@f\def\Blackrgb{0 0 0}
\ctr@ld@f\def\Whitergb{1 1 1}
\ctr@ld@f\def\Redrgb{1 0 0}
\ctr@ld@f\def\Greenrgb{0 1 0}
\ctr@ld@f\def\Bluergb{0 0 1}
\ctr@ld@f\def\Cyanrgb{0 1 1}
\ctr@ld@f\def\Magentargb{1 0 1}
\ctr@ld@f\def\Yellowrgb{1 1 0}
\ctr@ld@f\def\Grayrgb{0.5 0.5 0.5}
\ctr@ld@f\def\Chocolatergb{0.824 0.412 0.118}
\ctr@ld@f\def\DarkGoldenrodrgb{0.722 0.525 0.043}
\ctr@ld@f\def\DarkOrangergb{1 0.549 0}
\ctr@ld@f\def\Firebrickrgb{0.698 0.133 0.133}
\ctr@ld@f\def\ForestGreenrgb{0.133 0.545 0.133}
\ctr@ld@f\def\Goldrgb{1 0.843 0}
\ctr@ld@f\def\HotPinkrgb{1 0.412 0.706}
\ctr@ld@f\def\Maroonrgb{0.690 0.188 0.376}
\ctr@ld@f\def\Pinkrgb{1 0.753 0.796}
\ctr@ld@f\def\RoyalBluergb{0.255 0.412 0.882}
\ctr@ld@f\def\Pssetf@rst#1=#2|{\keln@mun#1|%
    \def\n@mref{c}\ifx\l@debut\n@mref\Pssetc@lor{#2}\else
    \def\n@mref{d}\ifx\l@debut\n@mref\pssetdash{#2}\else
    \def\n@mref{f}\ifx\l@debut\n@mref\pssetfillmode{#2}\else
    \def\n@mref{j}\ifx\l@debut\n@mref\pssetjoin{#2}\else
    \def\n@mref{u}\ifx\l@debut\n@mref\pssetupdate{#2}\else
    \def\n@mref{w}\ifx\l@debut\n@mref\pssetwidth{#2}\else
    \immediate\write16{*** Unknown attribute: \BS@ psset (..., #1=...)}%
    \fi\fi\fi\fi\fi\fi}
\ctr@ln@m\curr@ntdash
\ctr@ld@f\def\s@uvdash#1{\edef#1{\curr@ntdash}}
\ctr@ld@f\def\defaultdash{1}        
\ctr@ld@f\def\pssetdash#1{\ifps@cri\edef\curr@ntdash{#1}\ifcurr@ntPS\expandafter\Pssetd@sh#1 :\fi\fi}
\ctr@ld@f\def\Pssetd@shI#1{\PSc@mment{pssetdash Index=#1}\ifcase#1%
    \or\immediate\write\fwf@g{[] 0 \c@msetdash}
    \or\immediate\write\fwf@g{[6 2] 0 \c@msetdash}
    \or\immediate\write\fwf@g{[4 2] 0 \c@msetdash}
    \or\immediate\write\fwf@g{[2 2] 0 \c@msetdash}
    \or\immediate\write\fwf@g{[1 2] 0 \c@msetdash}
    \or\immediate\write\fwf@g{[2 4] 0 \c@msetdash}
    \or\immediate\write\fwf@g{[3 5] 0 \c@msetdash}
    \or\immediate\write\fwf@g{[3 3] 0 \c@msetdash}
    \or\immediate\write\fwf@g{[3 5 1 5] 0 \c@msetdash}
    \or\immediate\write\fwf@g{[6 4 2 4] 0 \c@msetdash}
    \fi}
\ctr@ld@f\def\Pssetd@sh#1 #2:{{\def\t@xt@{#1}\ifx\t@xt@\empty\Pssetd@sh#2:
    \else\def\t@xt@{#2}\ifx\t@xt@\empty\Pssetd@shI{#1}\else\s@mme=\@ne\def\debutp@t{#1}%
    \an@lysd@sh#2:\ifodd\s@mme\edef\debutp@t{\debutp@t\space\finp@t}\def\finp@t{0}\fi%
    \PSc@mment{pssetdash Pattern=#1 #2}%
    \immediate\write\fwf@g{[\debutp@t] \finp@t\space\c@msetdash}\fi\fi}}
\ctr@ld@f\def\an@lysd@sh#1 #2:{\def\t@xt@{#2}\ifx\t@xt@\empty\def\finp@t{#1}\else%
    \edef\debutp@t{\debutp@t\space#1}\advance\s@mme\@ne\an@lysd@sh#2:\fi}
\ctr@ln@m\curr@ntwidth
\ctr@ld@f\def\s@uvwidth#1{\edef#1{\curr@ntwidth}}
\ctr@ld@f\def\defaultwidth{0.4}     
\ctr@ld@f\def\pssetwidth#1{\ifps@cri\edef\curr@ntwidth{#1}\ifcurr@ntPS%
    \PSc@mment{pssetwidth Width=#1}\immediate\write\fwf@g{#1 \c@msetlinewidth}\fi\fi}
\ctr@ln@m\curr@ntjoin
\ctr@ld@f\def\pssetjoin#1{\ifps@cri\edef\curr@ntjoin{#1}\ifcurr@ntPS\expandafter\Pssetj@in#1:\fi\fi}
\ctr@ld@f\def\Pssetj@in#1#2:{\PSc@mment{pssetjoin join=#1}%
    \if#1r\def\t@xt@{1}\else\if#1b\def\t@xt@{2}\else\def\t@xt@{0}\fi\fi%
    \immediate\write\fwf@g{\t@xt@\space\c@msetlinejoin}}
\ctr@ld@f\def\defaultjoin{miter}   
\ctr@ld@f\def\Pssets@cond#1=#2|{\keln@mun#1|%
    \def\n@mref{c}\ifx\l@debut\n@mref\Pssets@condcolor{#2}\else%
    \def\n@mref{d}\ifx\l@debut\n@mref\pssetseconddash{#2}\else%
    \def\n@mref{w}\ifx\l@debut\n@mref\pssetsecondwidth{#2}\else%
    \immediate\write16{*** Unknown attribute: \BS@ psset second(..., #1=...)}%
    \fi\fi\fi}
\ctr@ln@m\curr@ntseconddash
\ctr@ld@f\def\pssetseconddash#1{\edef\curr@ntseconddash{#1}}
\ctr@ld@f\def\defaultseconddash{4}  
\ctr@ln@m\curr@ntsecondwidth
\ctr@ld@f\def\pssetsecondwidth#1{\edef\curr@ntsecondwidth{#1}}
\ctr@ld@f\edef\defaultsecondwidth{\defaultwidth} 
\ctr@ld@f\def\psresetsecondsettings{%
    \pssetseconddash{\defaultseconddash}\pssetsecondwidth{\defaultsecondwidth}%
    \Pssets@condcolor{\defaultsecondcolor}}
\ctr@ln@m\sec@ndcolor \ctr@ln@m\sec@ndcolorc@md
\ctr@ld@f\def\Pssets@condcolor#1{\ifps@cri\result@tent=\@ne\expandafter\c@lnbV@l#1 :%
    \def\sec@ndcolor{}\def\sec@ndcolorc@md{}%
    \ifcase\result@tent\or\pssetsecondgray{#1}\or\or\pssetsecondrgb{#1}%
    \or\pssetsecondcmyk{#1}\fi\fi}
\ctr@ln@m\sec@ndcolorc@mdStroke
\ctr@ld@f\def\pssetsecondcmyk#1{\def\sec@ndcolor{#1}\def\sec@ndcolorc@md{\c@msetcmykcolor}%
    \def\sec@ndcolorc@mdStroke{\c@msetcmykcolorStroke}}
\ctr@ld@f\def\pssetsecondrgb#1{\def\sec@ndcolor{#1}\def\sec@ndcolorc@md{\c@msetrgbcolor}%
    \def\sec@ndcolorc@mdStroke{\c@msetrgbcolorStroke}}
\ctr@ld@f\def\pssetsecondgray#1{\def\sec@ndcolor{#1}\def\sec@ndcolorc@md{\c@msetgray}%
    \def\sec@ndcolorc@mdStroke{\c@msetgrayStroke}}
\ctr@ld@f\def\us@secondC@lor{\immediate\write\fwf@g{\d@fsecondC@lor}%
    \let\fillc@md=\sdfillc@md}
\ctr@ld@f\def\sdfillc@md{\d@fsecondC@lor\space\c@mfill}
\ctr@ld@f\edef\defaultsecondcolor{\defaultcolor} 
\ctr@ld@f\def\Pss@tsecondSt{%
    \s@uvdash{\typ@dash}\pssetdash{\curr@ntseconddash}%
    \s@uvwidth{\typ@width}\pssetwidth{\curr@ntsecondwidth}\us@secondC@lor}
\ctr@ld@f\def\Psrest@reSt{\pssetwidth{\typ@width}\pssetdash{\typ@dash}\us@primarC@lor}
\ctr@ld@f\def\Pssetth@rd#1=#2|{\keln@mun#1|%
    \def\n@mref{c}\ifx\l@debut\n@mref\Pssetth@rdcolor{#2}\else%
    \immediate\write16{*** Unknown attribute: \BS@ psset third(..., #1=...)}%
    \fi}
\ctr@ln@m\th@rdcolor \ctr@ln@m\th@rdcolorc@md
\ctr@ld@f\def\Pssetth@rdcolor#1{\ifps@cri\result@tent=\@ne\expandafter\c@lnbV@l#1 :%
    \def\th@rdcolor{}\def\th@rdcolorc@md{}%
    \ifcase\result@tent\or\Pssetth@rdgray{#1}\or\or\Pssetth@rdrgb{#1}%
    \or\Pssetth@rdcmyk{#1}\fi\fi}
\ctr@ln@m\th@rdcolorc@mdStroke
\ctr@ld@f\def\Pssetth@rdcmyk#1{\def\th@rdcolor{#1}\def\th@rdcolorc@md{\c@msetcmykcolor}%
    \def\th@rdcolorc@mdStroke{\c@msetcmykcolorStroke}}
\ctr@ld@f\def\Pssetth@rdrgb#1{\def\th@rdcolor{#1}\def\th@rdcolorc@md{\c@msetrgbcolor}%
    \def\th@rdcolorc@mdStroke{\c@msetrgbcolorStroke}}
\ctr@ld@f\def\Pssetth@rdgray#1{\def\th@rdcolor{#1}\def\th@rdcolorc@md{\c@msetgray}%
    \def\th@rdcolorc@mdStroke{\c@msetgrayStroke}}
\ctr@ld@f\def\us@thirdC@lor{\immediate\write\fwf@g{\d@fthirdC@lor}%
    \let\fillc@md=\thfillc@md}
\ctr@ld@f\def\thfillc@md{\d@fthirdC@lor\space\c@mfill}
\ctr@ld@f\def\defaultthirdcolor{1}  
\ctr@ld@f\def\pstrimesh#1[#2,#3,#4]{{\ifcurr@ntPS\ifps@cri%
    \PSc@mment{pstrimesh Type=#1, Triangle=[#2,#3,#4]}%
    \s@uvc@ntr@l\et@tpstrimesh\ifnum#1>\@ne\Pss@tsecondSt\setc@ntr@l{2}%
    \Pstrimeshp@rt#1[#2,#3,#4]\Pstrimeshp@rt#1[#3,#4,#2]%
    \Pstrimeshp@rt#1[#4,#2,#3]\Psrest@reSt\fi\psline[#2,#3,#4,#2]%
    \PSc@mment{End pstrimesh}\resetc@ntr@l\et@tpstrimesh\fi\fi}}
\ctr@ld@f\def\Pstrimeshp@rt#1[#2,#3,#4]{{\l@mbd@un=\@ne\l@mbd@de=#1\loop\ifnum\l@mbd@de>\@ne%
    \advance\l@mbd@de\m@ne\figptbary-1:[#2,#3;\l@mbd@de,\l@mbd@un]%
    \figptbary-2:[#2,#4;\l@mbd@de,\l@mbd@un]\psline[-1,-2]%
    \advance\l@mbd@un\@ne\repeat}}
\initpr@lim\initpss@ttings\initPDF@rDVI
\ctr@ln@w{newbox}\figBoxA
\ctr@ln@w{newbox}\figBoxB
\ctr@ln@w{newbox}\figBoxC
\catcode`\@=12

\begin{abstract}
We consider the quintic nonlinear Schr\"odinger equation (NLS) on the circle
$$  i\partial_t u+\partial_{x}^{2}u =\pm \nu \ |u|^4u,\quad \nu\ll1, \ x\in \S^{1},\ t\in \R.$$
We prove that there exist  solutions corresponding to an initial datum built on four Fourier modes which form a resonant set (see definition \ref{def.res}),  which have a non trivial dynamic that involves periodic energy exchanges between the modes initially excited. It is noticeable that this nonlinear phenomena does not depend on the choice of the resonant set. \\
The dynamical result is obtained by calculating a resonant normal form up to order 10 of the Hamiltonian of the quintic NLS and then by isolating an effective term of order 6. Notice that this phenomena can not occur in the cubic NLS case for which the amplitudes of the Fourier modes are almost actions, i.e. they are almost constant.\\
    \end{abstract}

\begin{altabstract}
 Nous considérons l'équation de Schr\"odinger non linéaire (NLS) quintique sur le cercle
 $$  i\partial_t u+\partial_{x}^{2}u =\pm \nu \ |u|^4u,\quad  \nu\ll1, \ x\in \S^{1},\ t\in \R.$$
Nous montrons qu'il existe des solutions issues d'une condition initiale construite sur quatre modes de Fourier formant un ensemble résonant (voir définition \ref{def.res}) ont une dynamique non triviale mettant en jeu des échanges périodiques d'énergie entre ces quatre modes initialement excités. Il est remarquable que ce phénomène non linéaire soit indépendant du choix de l'ensemble résonant. \\ Le résultat  dynamique est obtenu en mettant d'abord sous forme normale résonante jusqu'à l'ordre 10 l'Hamiltonien de NLS quintique puis en isolant un terme effectif d'ordre 6. 
 Il est à noter que ce phénomène ne peut pas se produire pour NLS cubique  pour lequel les amplitudes des modes de Fourier sont des presque-actions et donc ne varient quasiment pas au cours du temps.   \end{altabstract}
\altkeywords{Forme Normale, Equation de Schr\"odinger non linéaire,
résonances, échange d'énergie}\frontmatter
\thanks{
\noindent The first author was supported in part by the  grant ANR-10-BLAN-DynPDE.\\
The second author was supported in part by the  grant ANR-07-BLAN-0250.\\
Both authors were supported in part by the  grant ANR-10-JCJC 0109.}

\maketitle

\section{Introduction}
\subsection{General introduction} Denote by  $\S^{1}=\R/2\pi\Z$ the circle, and let $\nu>0$ be a small parameter. In this paper we are concerned with the following quintic non linear Schr\"odinger equation
\begin{equation}\label{cauchy} 
\left\{
\begin{aligned}
&i\partial_t u+\partial_{x}^{2}u  =\pm \nu|u|^{4}u,\quad
(t,x)\in\R\times \S^{1},\\
&u(0,x)= u_{0}(x).
\end{aligned}
\right.
\end{equation} 
 If $u_{0}\in H^{1}(\S^{1})$, thanks to the conservation of the energy, we  show that the equation admits a unique global solution $u\in H^{1}(S^{1})$. In this work we want to describe some particular examples of nonlinear dynamics which can be generated by \eqref{cauchy}.\\

 For the linear Schr\"odinger equation ($\nu=0$ in \eqref{cauchy}) we can compute the solution explicitly in the Fourier basis: Assume that $\dis u_{0}(x)=\sum_{j\in \Z}\xi_{j}^{0}\e^{ijx}$, then $\dis u(t,x)=\sum_{j\in \Z}\xi_{j}(t)\e^{ijx}$ with $\dis \xi_{j}(t)=\xi_{j}^{0}\e^{-ij^{2}t}$. In particular, for all $j\in \Z$, the quantity $|\xi_{j}|$ remains constant. Now, let $\nu>0$, then a natural question is: do there exist solutions so that the $|\xi_{j}|$ have a nontrivial dynamic. First we review some known results. \\
 
Consider a general Hamiltonian perturbation where we add a linear term and a nonlinear term:
\begin{equation}\label{NLS}   i\partial_t u+\partial_{x}^{2}u + V\star u= \nu\partial_{\bar u}g(x,u,\bar u),\quad x\in \S^{1},\ t\in \R\end{equation} 
 where $V$ is a smooth periodic potential and $g$ is analytic and at least of order three. In that case the frequencies are $\om_j=j^2+\hat V(j), \ j\in \Z$ where $\hat V(j)$ denote the Fourier coefficients of $V$. Under a non resonant condition on these frequencies, it has been established by D. Bambusi and the first author \cite{BG06} (see also \cite{Gre07}) that the linear actions $|\xi_j|^2,\ j\in\Z$ are almost invariant during very long time, or more precisely, that for all $N\geq 1$ 
 $$ |\xi_j(t)|^2= |\xi_j(0)|^2+\mathcal{O}(\nu), \quad \text{for}\; \;|t|\leq \nu^{-N}.$$
 Therefore in this non resonant case, the dynamics of NLS are very close to the linear dynamics.\\
Another very interesting case is    the classical cubic NLS
\begin{equation} \label{NLS-cubic} 
 i\partial_t u+\partial_{x}^{2}u  =\pm \nu  |u|^{2}u,\quad
(t,x)\in\R\times \S^{1}\end{equation} 
and for this equation  again  nothing moves:
$$ |\xi_j(t)|^2= |\xi_j(0)|^2+\mathcal{O}(\nu), \quad \text{for all } t\in\R.$$
 This last result is a consequence of the existence of  action angle variables $(I,\theta)$ for the cubic NLS equation (there are globally defined in the defocusing case and locally defined around the origin in the focusing case, see respectively \cite{GK,GKP} and \cite{KT}) and that the actions are close to the  Fourier mode amplitudes to the square: $I_j=|\xi_j|^2(1+\mathcal{O}(\nu))$. \\
 Thus, in these two examples, the linear actions $|\xi_{j}|^{2}$ are almost constant in time, but for different reasons.\\
 Notice that in both previous cases, the Sobolev norms of the solutions, \linebreak[4]$ \left(\sum_{j\in \Z} j^{2s}|\xi_j(t)|^2\right)^{1/2}$ are almost constant for all $s\geq 0$.\\
 On the other hand, recently C. Villegas-Blas and the first author consider the following cubic NLS equation
\begin{equation} \label{NLS-GV} 
 i\partial_t u+\partial_{x}^{2}u  =\pm \nu \cos 2x\ |u|^{2}u,\quad
(t,x)\in\R\times \S^{1}\end{equation} 
 and prove that this special nonlinearity generates a nonlinear effect: if $u_0(x)=A e^{ix}+\bar A e^{-ix}$ then the modes $1$ and $-1$  exchange energy periodically (see \cite{GV}). For instance if
 $u_0(x)=\cos x
+\sin x$, a total beating is proved for $|t|\leq \nu^{-5/4}$:
$$|\xi_1(t)|^2=  \frac{1\pm \sin 2\nu t}{2} + \mathcal{O}(\nu^{3/4}), \quad |\xi_{-1}(t)|^2=  \frac{1\mp \sin 2\nu t}{2} + \mathcal{O}(\nu^{3/4}).$$
Of course in \eqref{NLS-GV} the interaction between the mode 1 and the mode $-1$ is induced by the $\cos 2x$ in front of the nonlinearity. \\
 
 In the present work we consider the quintic NLS equation \eqref{cauchy}. Notice that  Liang and You have proved in  \cite{LY} that, in the neighborhood of the origin, there exist many quasi periodic solutions of  \eqref{cauchy}. The basic approach is to apply the KAM method and vary the amplitude of the solutions in order to avoid resonances  in the spirit of the pioneer work of Kuksin-P\"oschel (\cite{KP}).  Here we want to take advantage of the resonances in the linear part of the equation to construct solutions that exchange energy between different Fourier modes.\\   Formally, by the Duhamel formula
 \begin{equation*}
 u(t)=\e^{it\partial_{x}^{2}}u_{0}-i\nu\int_{0}^{t}\e^{i(t-s)\partial_{x}^{2}}\big(|u|^{4}u\big)(s)\text{d}s,
 \end{equation*}
and  we deduce that $|\xi_{j}|^{2}$ cannot move as long as  $t\ll \nu^{-1}$. In this paper we  prove that for a large class of convenient initial data,  certain of the $|\xi_j|^2$ effectively move after a time of order $t\sim \nu^{-1}$. 
\begin{defi}\label{def.res}
A set $\mathcal{A}$ of the form 
\begin{equation*} 
\mathcal{A}=\big\{n,n+3k,n+4k,n+k\big\}, \quad k\in \Z\backslash\{0\}\; \text{ and } n\in\Z,
\end{equation*}
is called a resonant set. In the sequel we will use the notation
 \begin{equation*}
a_{2}=n,\;\;a_{1}=n+3k,\;\;b_{2}=n+4k,\;\;b_{1}=n+k.
\end{equation*}
\end{defi}
We are interested in these resonant sets, since they correspond to resonant monomials of order 6 in the normal form of the Hamiltonian \eqref{cauchy}, namely $\xi_{a_1}^2\xi_{a_2}\bar\xi_{b_1}^2\bar\xi_{b_2}$.
See Sections \ref{Sect1} and \ref{Sect2} for more details.
 \begin{exem}
For $(n,k)=(-2,1)$,  we obtain $(a_{2},a_{1},b_{2},b_{1})=(-2,1,2,-1)$; for $(n,k)=(-2,1)$,  we obtain $(a_{2},a_{1},b_{2},b_{1})=(-1,5,7,1)$.
 \end{exem}

\subsection{The main result}
Our first result is the following:

  \begin{theo}\label{thm1}
  There exist $T>0$, $\nu_{0}>0$ and a $2T-$periodic function $K_{\star}:\R\longmapsto ]0,1[$ which satisfies $K_{\star}(0)\leq 1/4$ and $K_{\star}(T)\geq3/4$ so that if $\mathcal A$ is a resonant set and if $0<\nu<\nu_{0}$, there exists a solution to \eqref{cauchy} satisfying for all $0\leq t\leq \nu^{-3/2}$
\begin{equation*}
 u(t,x)=\sum_{j\in \mathcal{A}}u_{j}(t)\e^{ijx}+\nu^{1/4} q_{1}(t,x)+\nu^{3/2}tq_{2}(t,x),
  \end{equation*}
\vspace{-0,4 cm}with 
\begin{equation*} 
\begin{array}{ccccc}
 |u_{a_{1}}(t)|^{2}&=&2 |u_{a_{2}}(t)|^{2}&=&K_{\star}(\nu t)\\[5pt]
 |u_{b_{1}}(t)|^{2}&=&2 |u_{b_{2}}(t)|^{2}&=&1-K_{\star}(\nu t),
 \end{array}
 \end{equation*}
 and where for all $s\in \R$,  $\|q_{1}(t,\cdot)\|_{H^{s}(\T)}\leq C_{s}$ for all $t\in \R_{+}$, and $\|q_{2}(t,\cdot)\|_{H^{s}(\T)}\leq C_{s}$ for all $0\leq t\leq  \nu^{-3/2}$.
 \end{theo}
 Theorem \ref{thm1}  shows that there is an exchange between the two modes $a_1$  and $a_2$ and the two modes $b_1$ and $b_2$.   It is remarkable that this nonlinear effect is universal in the sense that this dynamic does not depend on the choice of the resonant set $\mathcal A$.
 
 In Section \ref{Sect1}, we will see that such a result does not hold for any set $\mathcal{A}$ with $\#\mathcal{A}\leq 3$. However, three modes of a resonant set $\A$ can excite the fourth mode of $\A$ if this one was initially arbitrary small but not zero. More precisely : 
   \begin{theo}\label{thm2}
   For all $0<\gamma<1/10$, there exist $T_{\gamma}>0$, a $2T_{\gamma}-$periodic function $K_{\gamma}:\R\longmapsto ]0,1[$ which satisfies $K_{\gamma}(0)=\gamma$ and $K_{\gamma}(T_{\gamma})\geq1/10$, and there exists $\nu_{0}>0$ so that if $\mathcal A$ is a resonant set and if $0<\nu<\nu_{0}$, there exists a solution to \eqref{cauchy} satisfying for all $0\leq t\leq \nu^{-3/2}$
 \begin{equation*}
 u(t,x)=\sum_{j\in \mathcal{A}}u_{j}(t)\e^{ijx}+\nu^{1/4} q_{1}(t,x)+\nu^{3/2}tq_{2}(t,x),
  \end{equation*}
\vspace{-0,4 cm}with 
\begin{equation*} 
\begin{array}{cccccc}
 |u_{a_{1}}(t)|^{2}&=&K_{\gamma}(\nu t)\,; \;&  2|u_{a_{2}}(t)|^{2}&=&7+{K_{\gamma}(\nu t)}\\[5pt]
 |u_{b_{1}}(t)|^{2}&=&1-K_{\gamma}(\nu t)\,;\;&2 |u_{b_{2}}(t)|^{2}&=&1-K_{\gamma}(\nu t),
 \end{array}
 \end{equation*}
 and where for all $s\in \R$,  $\|q_{1}(t,\cdot)\|_{H^{s}(\T)}\leq C_{s}$ for all $t\in \R_{+}$, and $\|q_{2}(t,\cdot)\|_{H^{s}(\T)}\leq C_{s}$ for all $0\leq t\leq  \nu^{-3/2}$.
 \end{theo}

Of course the solutions satisfy  the three conservation laws : the mass, the momentum and the energy are constant quantities. Denote by $L_{j}=|u_{j}|^{2}$, then we have \\
$\bullet$ Conservation of the mass: $\dis \int |u|^{2}$
\begin{equation}\label{mass}
L_{a_{1}}+L_{a_{2}}+L_{b_{1}}+I_{b_{2}}=cst.
\end{equation}
$\bullet$ Conservation of the momentum: $\dis \text{Im}\,\int \ov{u}\partial_{x} u$
\begin{equation}\label{moment}
a_{1}L_{a_{1}}+a_{2}L_{a_{2}}+b_{1}L_{b_{1}}+b_{2}L_{b_{2}}=cst .
\end{equation}
$\bullet$ Conservation of the energy :  $\dis \int |\partial_{x}u|^{2}+\frac{\nu}3\int|u|^{6}$  
\begin{equation}\label{energy}
a^{2}_{1}L_{a_{1}}+a^{2}_{2}L_{a_{2}}+b^{2}_{1}L_{b_{1}}+b^{2}_{2}L_{b_{2}}=cst.
\end{equation}
On the other hand, the solutions given by Theorem \ref{thm1} satisfies for $0\leq t\leq \nu^{-5/4}$ and $s\geq 0$
 \begin{equation}\label{uHs}
\|u(t,\cdot)\|^{2}_{\dot{H}^{s}}=\frac{K_{\star}(\nu t)}2\big(2|a_{1}|^{2s}+|a_{2}|^{2s}-2|b_{1}|^{2s}-|b_{2}|^{2s}\big)+|b_{1}|^{2s}+\frac12|b_{2}|^{2s}+\mathcal{O}(\nu^{1/4}).
\end{equation}
Remark that \eqref{uHs} for $s=0,1$ is compatible with respectively \eqref{mass} and \eqref{energy}, since, for these   values of $s$, the coefficient  $\big(2|a_{1}|^{2s}+|a_{2}|^{2s}-2|b_{1}|^{2s}-|b_{2}|^{2s}\big)$ vanishes for $(a_1,a_2,b_1,b_2)\in\mathcal A$.\\
But  for $s\geq 2$, this coefficient is no more zero, except for some symmetric choices of $\mathcal A$ like $(-2,1,2,-1)$. Thus in the other cases $\|u(t,\cdot)\|^{2}_{\dot{H}^{s}}$ is not constant. Actually, a computation shows that, choosing $n=-k$ in the definition of $\mathcal A$,   the ratio between $\|u(T,\cdot)\|^{2}_{H^{s}}$ and $\|u(0,\cdot)\|^{2}_{H^{s}}$ is larger than 2 for $s\geq 4$. \\
Very recently, Colliander, Keel,  Staffilani,  Takaoka and Tao \cite{CKSTT} have proved a very nice result on the transfer of energy to high frequencies in the cubic defocusing nonlinear Schrödinger equation on the 2 dimensional torus. Of course their result is  more powerful ; in particular they allow a ratio between the initial $H^s$-norm and the $H^s$ norm for long time arbitrarily large. On the contrary our result only allows transfers of energy from modes $\{n,n+3k\}$ to modes $\{n+4k,n+k\}$ and thus the possibility of growing of the $H^s$-norm is bounded by $c^{s}$ for some constant $c$.  Nevertheless our approach is much more simple, it applies in 1-d and it is somehow universal (the dynamics we describe are not at all exceptional).

\medskip

\begin{rema}
Consider a resonant  set  $\mathcal{A}$, and    let $u$ be given by Theorem \ref{thm1}. Then by the scaling properties of the equation, for all $N\in \N^{*}$, $u_{N}$ defined by $ u_{N}(t,x)=N^{\frac12}u(N^{2}t,Nx)$ is also a solution of \eqref{cauchy} and we have
\begin{equation*}
u_{N}(t,x)=N^{1/2}\sum_{j\in \mathcal{A}}u_{j}(N^{2}t)\e^{ijNx}+\nu^{1/4} q_{1}(N^{2}t,Nx)+\nu^{3/2}tq_{2}(N^{2}t,Nx).
\end{equation*}
Next, for any $N\in \N^{*}$, the set $N\mathcal{A}$ is also a resonant set, and thus we can apply Theorem \ref{thm1}, which gives the existence of a solution to \eqref{cauchy} which reads
\begin{equation*}
\wt{u}_{N}(t,x)=\sum_{j\in \mathcal{A}}\wt{u}_{j}(t)\e^{ijNx}+\nu^{1/4} \wt{q}_{1}(t,x)+\nu^{3/2}t\wt{q}_{2}(t,x).
\end{equation*}
        Observe however that there are not the same.
\end{rema}

Theorem \ref{thm1} is obtained by calculating a resonant normal form up to order 10 of the Hamiltonian of the quintic NLS and then by isolating an effective term of order 6. Roughly speaking we obtain in the new variables $H= N+Z^{i}+Z_6^e+R$ where $N+Z^{i}$ depends only on the actions, $Z_6^e$, the effective part, is a polynomial homogeneous of order 6 which depends on one angle and $R$ is a remainder term.\\
 We first prove that, reduced to the resonant set, $N+  Z^{i}+Z_6^e$ generates the nonlinear dynamic that we expect. Then we have to prove that adding the remainder term $R$ and considering all the modes, this nonlinear dynamic persists beyond the local time (here $t\gtrsim \nu^{-1}$). In general this is a hard problem. Nevertheless in our case, the nonlinear dynamic corresponds to a stable orbit around a elliptic equilibrium point. So we explicitly calculate the action-angle variables $(K,\phi)\in \R^4\times \mathbb{T}^4$ for the finite dimensional system in such way that our nonlinear dynamics reads $\dot K=0$. Then for the complete system, we obtain $\dot K= O(\nu^{5/2})$ and we are essentially done.\\
In \cite{GV}, this construction was much more simpler since the finite dimensional nonlinear dynamics was in fact linear  (after a change of variable) and linear dynamics are more stable by perturbation than nonlinear ones.
 
 \subsection{Plan of the paper}
 

We begin in Section \ref{Sect1} with some arithmetical preliminaries. In Section  \ref{Sect2} we reinterpret equation \eqref{cauchy} as a Hamiltonian equation and we compute a completely resonant normal form at order 6. In Section \ref{Sect3} we study the equation (the model equation) obtained by the previous normal form after truncation of the   error terms. In Section \ref{Sect4} we show that the model equation gives  a  good approximation of some particular solutions of \eqref{cauchy}.

\section{Preliminaries: Arithmetic}\label{Sect1}
We are interested in sets  $\mathcal{A}$ of small cardinality so that there exist $(j_{1},j_{2},j_{3},\ell_{1},\ell_{2},\ell_{3})\in \mathcal{A}^{6}$ satisfying the following resonance condition
 \begin{equation} \label{eq.res}
\left\{
\begin{aligned}
& j_{1}^{2}+j_{2}^{2}+j^{2}_{3}  = \ell^{2}_1+\ell^{2}_{2}+\ell^{2}_{3} ,\\
& j_{1}+j_{2}+j_{3}=\ell_1+\ell_{2}+\ell_{3},
\end{aligned}
\right. \quad \text{and}\quad \big\{j_{1},j_{2},j_{3}\big\}\neq \big\{\ell_{1},\ell_{2},\ell_{3}\big\}.
\end{equation}
To begin with, let us recall a classical result 
\begin{lemm}\label{lem.actions}
Assume that $(j_{1},j_{2},j_{3},\ell_{1},\ell_{2},\ell_{3})\in \mathbb{Z}^{6}$ satisfy \eqref{eq.res}. Then $\big\{j_{1},j_{2},j_{3}\big\}\cap\big\{\ell_{1},\ell_{2},\ell_{3}\big\}=\emptyset$.
\end{lemm}

\begin{proof}
 If, say $j_{1}=\ell_{1}$, then we have the relation 
\begin{equation*}
j_2+j_3= \ell_{2}+\ell_{3} \;\; {\rm and } \;\;  j_2^2+j_3^2= \ell^{2}_{2}+\ell^{2}_{3},
\end{equation*}
and this implies that $(j_{2},j_{3})=(l_{2},l_{3})$ or $(j_{2},j_{3})=(l_{3},l_{2})$.
Squaring the first equality yields  $(j_{2}+j_{3})^{2}=(l_{2}+l_{3})^{2}$. To this equality we subtract $j_{2}^{2}+j_{3}^{2}=\ell_{2}^{2}+\ell_{3}^{2}$, which implies $j_{2}j_{3}=\ell_{2}\ell_{3}$. Now compute $$(\ell_{2}-j_{2})(\ell_{2}-j_{3})=\ell_{2}^{2}+j_{2}j_{3}-j_{2}\ell_{2}-j_{3}\ell_{2}=\ell_{2}(\ell_{2}+\ell_{3}-j_{2}-j_{3})=0,$$
hence the result.
\end{proof}

\begin {lemm}
Assume that there exist integers $(j_{1},j_{2},j_{3},\ell_{1},\ell_{2},\ell_{3})\in \mathcal{A}^{6}$ which satisfy \eqref{eq.res}.
Then the cardinal of $\mathcal{A}$ is greater or equal than 4.
\end{lemm}
 
 \begin{proof}
 Assume that $\#\mathcal{A}\leq 3$. Then by Lemma \ref{lem.actions} we can assume that      $\mathcal{A}=\{j_{1},j_{2},\ell_{1}\}$ and that 
 \begin{equation*}
 2j_{1}+j_{2}=3\ell_{1}\,;\;\; 2j^{2}_{1}+j^{2}_{2}=3\ell^{2}_{1}.
 \end{equation*}
 Let $k\in\Z$ so that $j_{1}=\ell_{1}+k$, then from the first equation we deduce that $j_{2}=\ell_{1}-2k$. Finally, inserting the last relation   in the second equation, we deduce that $k=0$ which  implies that $j_{1}=j_{2}=\ell_{1}$.
 \end{proof}
 The next result describes the sets $\mathcal{A}$ of cardinal 4 and which contain non trivial solutions to \eqref{eq.res}. According to definition \ref{def.res}, these sets are called resonant sets.
\begin {lemm}[Description of the resonant sets]\label{lem.res}
The resonance sets are the 
\begin{equation*} 
\mathcal{A}=\big\{n,n+3k,n+4k,n+k\big\}, \quad k\in \Z\backslash\{0\}\; \text{ and } n\in\Z.
\end{equation*}
\end{lemm}

\begin{proof}
By Lemma \ref{lem.actions}, we know that either $\big\{j_{1},j_{2},j_{3}\big\}=\big\{\ell_{1},\ell_{2},\ell_{3}\big\}$ or $\big\{j_{1},j_{2},j_{3}\big\}\cap\big\{\ell_{1},\ell_{2},\ell_{3}\big\}=\emptyset$. We consider the second case. \\
$\bullet$ First we exclude the case $j_{1}=j_{2}=j_{3}=j$. In that case we have to solve
 \begin{equation} \label{eq.*}
\left\{
\begin{aligned}
& 3j^{2}  = \ell^{2}_1+\ell^{2}_{2}+\ell^{2}_{3}  , \\   
& 3j =\ell_1+\ell_{2}+\ell_{3}. 
\end{aligned}
\right.
\end{equation}
We will show that  \eqref{eq.*}  implies $\ell_{1}=\ell_{2}=\ell_{3}=j$. Set $\ell_{1}=j+p$ and $\ell_{2}=j+q$. Then by the second line $\ell_{3}=j-p-q$. Now, we plug in the first line and get $p^{2}+q^{2}+pq=0$. This in turn implies that $p=q=0$ thanks to the inequality $p^{2}+q^{2}\geq 2|pq|$.\\
$\bullet$ Then we can assume that $j_{2}=j_{3}$ and $\ell_{2}=\ell_{3}$, and $\sharp\{j_1,j_2,\ell_1,\ell_2\}=4$.  Thus we have to solve
 \begin{equation*} 
\left\{
\begin{aligned}
& j_{1}^{2}+2j_{2}^{2}  = \ell^{2}_1+2\ell^{2}_{2}  , \\   
& j_{1}+2j_{2} =\ell_1+2\ell_{2}. 
\end{aligned}
\right.
\end{equation*}
From the first line, we infer that $ (j_{1}-\ell_{1})(j_{1}+\ell_{1})=2(\ell_{2}-j_{2})(\ell_{2}+j_{2})$. The second gives $j_{1}-\ell_{1}=2(\ell_{2}-j_{2})$, thus $j_{1}+\ell_{1}=j_{2}+\ell_{2}$. Hence we are led to solve the system
 \begin{equation*} 
\left\{
\begin{aligned}
&\ell_1-\ell_{2}= -j_{1}+j_{2} , \\
&\ell_1+2\ell_{2}= j_{1}+2j_{2}
\end{aligned}
\right.
\end{equation*}
where the integers $j_{1}$ and $j_{2}$ are considered as parameters. The solutions are
 \begin{equation*}
 \ell_{1}=\frac{1}3(-j_{1}+4j_{2}),\quad \quad  \ell_{2}=\frac{1}3(2j_{1}+j_{2})
 \end{equation*}
with the restriction, $j_{1}\equiv j_{2}$  mod $ 3$,  in order to obtain integer solutions. Let $n\in \Z$, $k\in \Z^{*}$ so that $j_{1}=n$ and $j_{2}=n+3k$, the solutions then reads $\ell_{1}=n+4k$ and $\ell_{2}=n+k$, as claimed.
\end{proof}
Define the set 
 \begin{multline*}
{\mathcal R}=\lbrace(j_1,j_2,j_3,\ell_1,\ell_{2},\ell_{3})\in\mathbb{Z}^6\;s.t. \\
 j_1+j_2+j_3=\ell_{1}+\ell_{2}+\ell_{3} \;\; {\rm and } \;\; j_1^2+j_2^2+j_3^2=\ell^{2}_{1}+\ell^{2}_{2}+\ell^{2}_{3}\rbrace.
\end{multline*}
The following result will be useful in the sequel 
 
 \begin{lemm}\label{lem.couple}
 Let  $(j_{1},j_{2},j_{3},\ell_{1},p_{1},p_{2})\in \mathcal{R}$. Assume that $j_{1},j_{2},j_{3},\ell_{1}\in \mathcal{A}$. Then $p_{1},p_{2}\in \mathcal{A}$.
 \end{lemm}

 \begin{proof}
 Let $j_{1},j_{2},j_{3},\ell_{1}\in \mathcal{A}$ and $p_{1},p_{2}\in \mathbb{N}$ so that 
  \begin{equation}\label{eq.double}
\left\{
\begin{aligned}
& p_{1}+p_{2}=j_{1}+j_{2}+j_{3}-\ell_{1}   , \\   
& p^{2}_{1}+p^{2}_{2}=j^{2}_{1}+j^{2}_{2}+j^{2}_{3}-\ell^{2}_{1}  . 
\end{aligned}
\right.
\end{equation}
By Lemma \ref{lem.res}, there exist $n,k\in \Z$ and $(m_{s})_{1\leq s\leq 4}$ with $m_{s}\in \big\{0,1,3,4\big\}$ so that $j_{s}=m_{s}k$ and $\ell_{1}=m_{4}k$. We also write $p_{1}=n+q_{1}$ and $p_{2}=n+q_{2}$. We plug these expressions in \eqref{eq.double} which gives
 \begin{equation*} 
\left\{
\begin{aligned}
& q_{1}+q_{2}=(m_{1}+m_{2}+m_{3}-m_{4})k   , \\   
& q^{2}_{1}+q^{2}_{2}+2n(q_{1}+q_{2})=2n(m_{1}+m_{2}+m_{3}-m_{4})k+(m^{2}_{1}+m^{2}_{2}+m^{2}_{3}-m^{2}_{4})k^{2},  \end{aligned}
\right.
\end{equation*}
and is equivalent to 
 \begin{equation*} 
\left\{
\begin{aligned}
& q_{1}+q_{2}=(m_{1}+m_{2}+m_{3}-m_{4})k   , \\   
& q^{2}_{1}+q^{2}_{2}=(m^{2}_{1}+m^{2}_{2}+m^{2}_{3}-m^{2}_{4} )k^{2} . 
\end{aligned}
\right.
\end{equation*}
We write $q_{1}=r_{1}k$ and $q_{2}=r_{2}k$, then $r_{1},r_{2}\in \Q$ satisfy
  \begin{equation} \label{eq.16}
\left\{
\begin{aligned}
& r_{1}+r_{2}=m_{1}+m_{2}+m_{3}-m_{4}:= S  , \\   
& r^{2}_{1}+r^{2}_{2}=m^{2}_{1}+m^{2}_{2}+m^{2}_{3}-m^{2}_{4} : =T.
\end{aligned}
\right.
\end{equation}
Next, we observe that indeed $r_{1},r_{2}\in \Z$ : In fact \eqref{eq.16} is equivalent to 
 \begin{equation} 
r_{1}+r_{2}=S  , \qquad  r_{1}r_{2}=\frac12(S^{2}-T):=U,
\end{equation}
($U\in \Z$ since $S$ and $T$ have same parity) and $r_{1}, r_{2}$ are the roots of the polynomial $X^{2}-SX+U$. Thus if $r=\alpha/\beta$ with $\alpha\wedge\beta=1$, we have that $\beta|1$ and then $r\in \Z$.\\
We are finally reduced to solve \eqref{eq.16}
 where $m_{s}\in \big\{0,1,3,4\big\}$. We list all possible cases in the following array : By symmetry we only need to consider the cases $m_{1}\geq m_{2}\geq m_{3}$. We denote by $m_{1}m_{2}m_{3}m_{4}$ a possible choice and by $T=m^{2}_{1}+m^{2}_{2}+m^{2}_{3}-m^{2}_{4}$.\\

\begin{tabular}{|c|c|c|c|c|c|}
  \hline
Values of $m_{s}$ & Value of $T$  &  Values of $m_{s}$ & Value of $T$& $m_{s}$ & Value of $T$  \\
  \hline
  4440 & 48 & 4441 & 47& 4443&39\\
  4430 & 41 & 4431 & 40=36+4& 4410&33\\
    4413 & 24 & 4401 & 31& 4403&23\\
    4330 & 34=25+9 & 4331 & 33& 4310&26=25+1\\
    4301 & 24 & 4110 & 18=9+9& 4113&9=9+0\\
      4103 & 8=4+4 & 4001 & 15& 4003&7\\
        3330 & 27 & 3331 & 26=25+1& 3334&11\\
        3310 & 19 & 3314 & 3& 3301&17=16+1\\
        3304 & 2 & 3110 & 11& 3114&-5\\
        3104 & -6 & 3001 & 8=4+4& 3004&-7\\
        1110 & 3 & 1113 & -6& 1114&-13\\
  \hline
\end{tabular}
~\\[10pt]
In this array, we read all the possible solutions to \eqref{eq.16} which are (assuming that $m_{1}\geq m_{2}\geq m_{3}$ and $r_{1}\geq r_{2}$)
\begin{equation}
 (r_{1},r_{2},m_{1},m_{2},m_{3},m_{4})=(3,3,4,1,1,0),\quad (3,0,4,1,1,3),\quad (4,1,3,3,0,1).
\end{equation}
 Now we observe that we always have $r_{1},r_{2}\in \big\{0,1,3,4\big\}$, so that if we come back to \eqref{eq.double}, $p_{1}=n+r_{1}k$, $p_{2}=n+r_{2}k$ and $p_{1},p_{2}\in \mathcal{A}$.
\end{proof}

\section{The normal form}\label{Sect2}
\subsection{Hamiltonian formulation}
From now, and until the end of the paper, we set $\eps=\nu^{1/4}$. In the sequel, it will be more convenient to deal with small initial conditions to \eqref{cauchy}, thus we make the change of unknown $v=\eps u$ and we obtain 
\begin{equation}\label{cauchy*} 
\left\{
\begin{aligned}
&i\partial_t v+\partial_{x}^{2}v  =|v|^{4}v,\quad
(t,x)\in\R\times {\T}^{},\\
&v(0,x)= v_{0}(x)=\eps u_{0}(x).
\end{aligned}
\right.
\end{equation} 

Let us expand $v$ and $\bar v$ in Fourier modes:
$$v(x)=\sum_{j\in \Z}\xi_j e^{ijx},\quad \bar v(x)=\sum_{j\in\Z} \eta_j e^{-ijx}.$$
We define
\begin{equation*}
P(\xi,\eta)= \frac13\int_{\T} |v(x)|^6 \text{d}x = \frac13\sum_{\substack{j,\ell\in \Z^3\\ \mathcal M(j,\ell)=0}}\xi_{j_1}\xi_{j_2}\xi_{j_3}\eta_{\ell_1}\eta_{\ell_2}\eta_{\ell_3},
 \end{equation*}
where $\Mc(j,\ell)= j_1+j_2+\cdots+j_{p}-\ell_1-\ell_2-\cdots-\ell_{p}$ denotes the momentum of the multi-index $(j,l)\in \Z^{2p}$ or equivalently the momentum of the monomial $\xi_{j_1}\xi_{j_2}\cdots\xi_{j_p}\eta_{\ell_1}\eta_{\ell_2}\cdots\eta_{\ell_p}$. \\
In this Fourier setting the equation \eqref{cauchy*} reads as an infinite Hamiltonian system
\begin{equation}\label{hamsys}
\left\{\begin{array}{rll}
i\dot \xi_j&=j^2\xi_j+\frac{\partial P}{\partial \eta_j} &\quad j\in \Z,\\
-i\dot \eta_j&=j^2\eta_j+\frac{\partial P}{\partial \xi_j} &\quad j\in \Z  .
\end{array}\right.
\end{equation}

 Since the regularity is not an issue in this work, we will work in the following analytic phase space ($\rho\geq 0$)
$$\Ac_\rho=\{(\xi,\eta)\in \ell^1(\Z)\times  \ell^1(\Z)\mid ||(\xi,\eta)||_\rho:=\sum_{j\in \Z}e^{\rho |j|}(|\xi_j|+|\eta_j|)<\infty\}
$$
which  we endow with the canonical symplectic structure $-i\sum_j d\xi_j\wedge \eta_j$. Notice that this Fourier space corresponds to  functions $u(z)$ analytic on a strip $|\Im z| <\rho$  around the real axis.\\
According to this symplectic structure, the Poisson bracket between two functions $f$ and $g$ of $(\xi,\eta)$ is defined by
$$\lbrace{f},{g}\rbrace={-i}\sum_{j\in\mathbb{Z}}\frac{\partial{f}}{\partial{\xi_j}}\frac{\partial{g}}{\partial{\eta_j}}-
\frac{\partial{f}}{\partial{\eta_j}}\frac{\partial{g}}{\partial{\xi_j}}.$$
In particular, if $(\xi(t),\eta(t))$ is a solution of \eqref{hamsys} and $F$ is some regular Hamiltonian function, we have
$$
\frac{d}{\text{d}t}F(\xi(t),\eta(t))=\{F,H\}(\xi(t),\eta(t))
$$
where
\begin{equation*}
H=N+P= \sum_{j\in \Z}j^2\xi_{j}\eta_{j} +\frac13\sum_{\substack{j,\ell\in \Z^3\\ \mathcal M(j,\ell)=0}}\xi_{j_1}\xi_{j_2}\xi_{j_3}\eta_{\ell_1}\eta_{\ell_2}\eta_{\ell_3},
\end{equation*}
is the total Hamiltonian of the system. It is  convenient to work in the symplectic polar coordinates $\dis \big(\xi_{j}=\sqrt{I_{j}}\e^{i\theta_{j}},\eta_{j}=\sqrt{I_{j}}\e^{-i\theta_{j}})_{j\in \Z}$. Since we have $\text{d}\xi \wedge\text{d}\eta=i\text{d}\theta\wedge \text{d}I$,  the system \eqref{cauchy*} is equivalent to 
\begin{equation*}
\left\{\begin{array}{rrl}
\dot \theta_j=&-\frac{\partial H}{\partial I_j} &\quad j\in \Z,\\[4pt]
\dot I_j=&\frac{\partial H}{\partial \theta_j} &\quad j\in \Z.
\end{array}\right.
\end{equation*}
Finally, we define 
\begin{equation}\label{J}
J=\sum_{j\in \Z}I_{j}=\sum_{j\in \Z}\xi_{j}\eta_{j}=\|v\|^{2}_{L^{2}(\T)},
\end{equation}
 which is a constant of motion for \eqref{cauchy*} and \eqref{hamsys}.

\subsection{The Birkhoff normal form procedure}
We denote by $B_\rho(r)$ the ball of radius $r$ centred at the origin in $\Ac_\rho$. 
 Recall the definition  
 \begin{multline*}
{\mathcal R}=\lbrace(j_1,j_2,j_3,\ell_1,\ell_{2},\ell_{3})\in\mathbb{Z}^6\;s.t. \\
 j_1+j_2+j_3=\ell_{1}+\ell_{2}+\ell_{3} \;\; {\rm and } \;\; j_1^2+j_2^2+j_3^2=\ell^{2}_{1}+\ell^{2}_{2}+\ell^{2}_{3}\rbrace
\end{multline*}
and its subset
\begin{equation*}
\mathcal{R}_{0}=\mathcal{R} \cap \Big\{ \big\{j_{1},j_{2},j_{3}\big\}=\big\{\ell_{1},\ell_{2},\ell_{3}\big\}\Big\}.
\end{equation*}
We are now able to state the main result of this section, which is a normal form result at order 10 for the Hamiltonian  $H$.

\begin{prop}\label{NF}
There exists a canonical change of variable $\tau$ from $B_\rho(\eps)$ into  $B_\rho(2\eps)$ with $\eps$ small enough such that
\begin{equation}\label{ham*}
\ov{H}:=H\circ \tau = N +Z_{6}+R_{10},
\end{equation}
where
\begin{enumerate}[(i)]
\item $N$ is the term $\dis N(I) = \sum_{j\in \Z}j^2I_{j}$;
\item $\dis Z_6$  is the homogeneous polynomial of degree 6     
\begin{equation*}
Z_{6}=\sum_{\mathcal{R}}\xi_{j_{1}}\xi_{j_{2}}\xi_{j_{3}}\eta_{\ell_{1}}\eta_{\ell_{2}}\eta_{\ell_{3}}.
\end{equation*}
\item  $R_{10}$ is the remainder of order 10, i.e. a Hamiltonian satisfying\\
 $||X_{R_{10}}(z)||_\rho \leq C ||z||^9_\rho$   for $z=(\xi,\eta)\in B_\rho(\eps)$;
 \item  $\tau$ is close to the identity: there exist a constant $C_\rho$ such that $||\tau(z)-z||_\rho\le{C_\rho}||z||^2_\rho$  for all $z\in B_\rho(\eps).$
\end{enumerate}
\end{prop}

By abuse of notation, in the proposition and in the sequel, the new variables $(\xi',\eta')=\tau^{-1}(\xi,\eta)$ are still  denoted by $(\xi,\eta)$.

\begin{proof} 
For convenience of the reader, we briefly recall the Birkhoff normal form method.
Let us search $\tau$ as time one flow of $\chi$ a polynomial Hamiltonian of order 6, 
\begin{equation*} 
\chi = \sum_{\substack{j,\ell\in \Z^3\\ \mathcal M(j,\ell)=0}} a_{j,\ell}\   \xi_{j_1}\xi_{j_2}\xi_{j_3}\eta_{\ell_1}\eta_{\ell_2}\eta_{\ell_3}.
\end{equation*}
For any smooth function $F$,   the Taylor expansion of $F\circ \Phi^t_\chi$ between $t=0$ and $t=1$ gives
$$F\circ \tau = F+ \{F ,\chi \}+\frac 1 2 \int_0^1(1-t)\{\{F,\chi\},\chi\}\circ \Phi^t_\chi \text{d}t.$$
Applying this formula to $H=N+P$ we get
\begin{equation*} 
H\circ \tau =N+P+ \{N ,\chi \}+\{P,\chi\}+\frac 1 2 \int_0^1(1-t)\{\{H,\chi\},\chi\}\circ \Phi^t_\chi \text{d}t.
\end{equation*}
 Therefore in order to obtain $H\circ \tau = N +Z_{6}+R_{10}$ we define  
\begin{equation} \label{homo}
Z_{6}=P+ \{N ,\chi \}
\end{equation}
 and 
\begin{equation}\label{R}
R_{10}=\{P ,\chi \}+ \frac 1 2 \int_0^1(1-t)\{\{H,\chi\},\chi\}\circ \Phi^t_\chi \text{d}t.
\end{equation}
For $j,\ell\in \Z^3$ we define the associated divisor by
$$\Omega(j,\ell)=j_1^2+j_2^2+j_3^2-\ell_1^2-\ell_2^2-\ell_3^2.
$$
The homological equation \ref{homo} is solved by defining
\begin{equation*} 
\chi := \sum_{\substack{j,\ell\in \Z^3\\ \mathcal M(j,\ell)=0, \Om(j,\ell)\neq 0}} \frac 1 {i\Om(j,\ell)} \xi_{j_1}\xi_{j_2} \xi_{j_3}\eta_{\ell_1}\eta_{\ell_2}\eta_{\ell_3}
\end{equation*}
and thus 
$\dis Z_6 = \sum_{\substack{j,\ell\in \Z^3\\ \mathcal M(j,\ell)=0, \Om(j,\ell)= 0}} \xi_{j_1}\xi_{j_2}\xi_{j_3}\eta_{\ell_1}\eta_{\ell_2}\eta_{\ell_3}.$
At this stage we define the class $\Pc_p$ of formal polynomial $$\dis Q=\sum_{\substack{j,l\in \Z^p\\ \mathcal M(j,\ell)=0}} a_{j,\ell}  \xi_{j_1}\xi_{j_2}\cdots\xi_{j_p}\eta_{\ell_1}\eta_{\ell_2}\cdots\eta_{\ell_p}$$
where the $a_{j\ell}$ form a bounded family and we define $[Q]=\sup_{j,\ell}|a_{j\ell}|$. We recall the following result from \cite{FG}
\begin{lemm}\label{FG} Let $P\in \Pc_p$. Then
\begin{itemize}
\item[(i)] $P$ is well defined and continuous (and thus analytic) on $\Ac_\rho$ and 
$$|P(\xi,\eta)|\leq [P]||(\xi,\eta)||_0^{2p}\leq  [P]||(\xi,\eta)||_\rho^{2p}.$$
\item[(ii)] The associated vector field $X_P$ is bounded (and thus smooth) from $\Ac_\rho$ to $\Ac_\rho$ and 
$$ ||X_P(\xi,\eta)||_\rho\leq 2p[P]||(\xi,\eta)||_\rho^{2p-1}.$$
\item[(iii)] Let $Q\in \Pc_q$ then $\{P,Q\}\in \Pc_{p+q-2}$ and
$$[\{P,Q\}]\leq 2qp [P][Q].$$
\end{itemize}
\end{lemm}
For convenience of the reader the proof of this lemma is recalled in the appendix A.

By using this Lemma and since there are no small divisors in this resonant case, $Z_6$ and $\chi$ have analytic vector fields on $\Ac_\rho$. On the other hand, since $\chi$ is homogeneous of order 6,  for $\eps$ sufficiently small, the time one flow generated by $\chi$  maps the ball $B_\rho(\eps)$ into the  ball $B_\rho(2\eps)$ and is close to the identity in the sense of assertion $(iv)$. \\
Concerning $R_{10}$, by construction it is a Hamiltonian function which is of order at least 10. To obtain assertion $(iii)$ it remains to prove that the vector field $X_{R_{10}}$ is smooth from $B_\rho(\eps)$ into $\Ac_\rho$ in such a way we can Taylor expand $X_{R_{10}}$ at the origin. This is clear for the first term of \eqref{R}: $\{P,\chi\}$ have a smooth vector field as a consequence of Lemma \ref{FG} assertions $(ii)$ and $(iii)$. For the second term, notice that $\{H,\chi\}=Z_{6}-P +\{P,\chi\}$ which is  a polynomial on $\Ac_\rho$ having bounded coefficients and the same is true for $Q=\cp{\{H,\chi\},\chi}$. Therefore, in view of Lemma \ref{FG},  $X_Q$ is smooth. Now, since for $\eps$ small enough $\Phi^t_\chi$ maps smoothly the ball $B_\rho(\eps)$ into the  ball $B_\rho(2\eps)$  for all $0\leq t\leq 1$, we conclude that $\int_0^1(1-t)\{\{H,\chi\},\chi\}\circ \Phi^t_\chi \text{d}t$ has a smooth vector field.
\end{proof}

\subsection{Description of the resonant normal form}

In this subsection we  study the resonant part of the normal form  given by Proposition \ref{NF}
\begin{equation*}
Z_{6}=\sum_{\mathcal{R}}\xi_{j_{1}}\xi_{j_{2}}\xi_{j_{3}}\eta_{\ell_{1}}\eta_{\ell_{2}}\eta_{\ell_{3}}.
\end{equation*}
We have 
\begin{prop}\label{prop.Z6}
The polynomial $Z_{6}$ reads 
 \begin{equation}\label{Z6}
Z_{6}= Z^{i}_{6}+Z^{e}_{6}+Z_{6,2}+Z_{6,3},
\end{equation}
where
\begin{enumerate}[(i)]
 \item $\dis Z^{i}_6$  is a homogeneous polynomial of degree 6 which only depends on the actions (recall the definition \eqref{J} of $J$):
\begin{equation*}
Z^{i}_{6}(I)=\sum_{\mathcal{R}_{0}}\xi_{j_{1}}\xi_{j_{2}}\xi_{j_{3}}\eta_{\ell_{1}}\eta_{\ell_{2}}\eta_{\ell_{3}}=6J^{3}-9J\sum_{k\in\Z}I^{2}_{k}+4\sum_{k\in\Z}I^{3}_{k};
\end{equation*}
\item  $\dis Z^{e}_6$    is the effective Hamiltonian, it is a homogeneous polynomial of degree 6 which  involves only modes in the resonant set $\A$:
 $$Z^{e}_{6}(\xi,\eta)= 9(\xi_{a_{2}}\xi^{2}_{a_{1}}\eta_{b_{2}}\eta^{2}_{b_{1}}+\xi_{b_{2}}\xi^{2}_{b_{1}}\eta_{a_{2}}\eta^{2}_{a_{1}});$$ 
\item $Z_{6,2}$ is an homogeneous polynomial of degree 6 which contains all the terms   involving exactly  two modes which are not in $\mathcal{A}$;
\item  $Z_{6,3}$ is an homogeneous polynomial of degree 6 which contains all the terms   involving at least  three modes which are not in $\mathcal{A}$.
\end{enumerate}
\end{prop}

\begin{exem} Assume that $\mathcal{A}=\{-2,1,2,-1\}$. Then we have  $Z^{e}_{6}(\xi,\eta)= 9(\xi_{-2}\xi^{2}_{1}\eta_{2}\eta^{2}_{-1}+\xi_2\xi^{2}_{-1}\eta_{-2}\eta^{2}_{1})$, and we can compute (see Example \eqref{exemple})
\begin{eqnarray*}
Z_{6,2}(\xi,\eta)&=& 36(\xi_{3}\xi_{-2}\xi_{-1}\eta_{-3}\eta_{2}\eta_{1} + \xi_{-3}\xi_{2}\xi_{1}\eta_{3}\eta_{-2}\eta_{-1})\\
&& +9(\xi_{4}\xi^{2}_{-2}\eta_{-4}\eta^{2}_{2} + \xi_{-4}\xi^{2}_{2}\eta_{4}\eta^{2}_{-2}).
\end{eqnarray*} 
If $\mathcal{A}=\{-1,5,7,1\}$, the term $Z_{6,2}$ is much more complicated (see Example \eqref{exemple}).
\end{exem}

\proof ({\it of Proposition \ref{prop.Z6}})
A priori, in \eqref{Z6} there should also be a polynomial $Z_{6,1}$  composed of the terms involving exactly  one mode which is  not in $\mathcal{A}$. An important fact of Proposition \ref{prop.Z6} is that $Z_{6,1}=0$, and this is a consequence of Lemma \ref{lem.couple}. \\
The specific form of the effective Hamiltonian announced in (ii) follows from the proof of Lemma \ref{lem.res}.\\ It remains to compute $Z_6^i$. This is done in the two following lemmas.\\
 Denote by 
 \begin{equation*}
{\mathcal Q}=\lbrace(j_1,j_2,\ell_1,\ell_{2})\in\mathbb{Z}^4\;s.t. \;\;
 j_1+j_2=\ell_{1}+\ell_{2} \;\; {\rm and } \;\; j_1^2+j_2^2=\ell^{2}_{1}+\ell^{2}_{2}\rbrace.
\end{equation*}
Observe that if $(j_{1},j_{2},\ell_{1},\ell_{2})\in \mathcal{Q}$, then $\{j_{1},j_{2}\}=\{\ell_{1},\ell_{2}\}$ (see the proof of Lemma \ref{lem.actions}). Next, we can state 
\begin{lemm} The two following identities hold true
\begin{equation}\label{Z4}
Z_{4}(I):=\sum_{(j_1,j_2,\ell_1,\ell_{2})\in \mathcal Q}\xi_{j_{1}}\xi_{j_{2}}\eta_{\ell_{1}}\eta_{\ell_{2}}=2J^{2}-\sum_{j\in\Z}I_{j}^{2},
\end{equation}

\begin{equation}\label{Z44}
W^{(k)}_{4}(I):=\sum_{(j_1,\ell_2,\ell_{3})\in \Omega^{(k)}} \xi_{k}\xi_{j_{1}}\eta_{\ell_{2}}\eta_{\ell_{3}}=2I_{k}(J-I_{k}),
\end{equation}
where $\Omega^{(k)}=\big\{   (j_1,\ell_2,\ell_{3})\in \Z^{3}\;\;\text{s.t.}\;\;   (k,j_1,\ell_2,\ell_{3})\in \mathcal Q\;\;\text{and}\;\;j_{1}\neq k  \big\}$.
\end{lemm}

\begin{proof}
First we prove \eqref{Z4}. Thanks to the previous  remark and the fact  that $\xi_{j}\eta_{j}=I_{j}$, we have
\begin{eqnarray*}
Z_{4}(I)&=&\sum_{\mathcal Q,j_{1}=\ell_{1}}\xi_{j_{1}}\xi_{j_{2}}\eta_{\ell_{1}}\eta_{\ell_{2}}+\sum_{\mathcal Q,j_{1}\neq\ell_{1}}\xi_{j_{1}}\xi_{j_{2}}\eta_{\ell_{1}}\eta_{\ell_{2}}\nonumber\\
&=&\sum_{(j_{1},j_{2})\in \Z^{2}}I_{j_{1}}I_{j_{2}}+\sum_{\substack{   (j_{1},j_{2})\in \Z^{2}\\  j_{1}\neq j_{2}}}I_{j_{1}}I_{j_{2}}\nonumber\\
&=&2\big(\sum_{j\in\Z}I_{j}\big)^{2}-\sum_{j\in\Z}I_{j}^{2}=2J^{2}-\sum_{j\in\Z}I_{j}^{2},
\end{eqnarray*}
which was the claim.\\
We now turn to \eqref{Z44}. Again we split the sum in two
\begin{eqnarray*}
W_{4}^{(k)}(I)&=&\sum_{\substack{(j_1,\ell_2,\ell_{3})\in \Omega^{(k)}\\j_{1}=\ell_{2}}} \xi_{k}\xi_{j_{1}}\eta_{\ell_{2}}\eta_{\ell_{3}}+\sum_{\substack{(j_1,\ell_2,\ell_{3})\in \Omega^{(k)}\\j_{1}\neq\ell_{2}}} \xi_{k}\xi_{j_{1}}\eta_{\ell_{2}}\eta_{\ell_{3}}\\
&=&I_{k}\sum_{j_{1}\in\Z\backslash \{k\}} I_{j_{1}}+I_{k}\sum_{j_{1}\in\Z\backslash \{k\}} I_{j_{1}}\nonumber\\
&=& 2I_{k}(J-I_{k}),
\end{eqnarray*}
hence the result.
\end{proof}

\begin{lemm} The following identity holds true
\begin{equation*}
Z^{i}_{6}(I):=\sum_{\mathcal{R}_{0}}\xi_{j_{1}}\xi_{j_{2}}\xi_{j_{3}}\eta_{\ell_{1}}\eta_{\ell_{2}}\eta_{\ell_{3}}=6J^{3}-9J\sum_{k\in \Z}I^{2}_{k}+4\sum_{k\in \Z}I^{3}_{k}.
\end{equation*}
\end{lemm}

\begin{proof} First we split the sum into three parts
\begin{eqnarray*}
Z^{i}_{6}(I)&=&\sum_{\mathcal{R}_{0},j_{1}=\ell_{1}}\xi_{j_{1}}\xi_{j_{2}}\xi_{j_{3}}\eta_{\ell_{1}}\eta_{\ell_{2}}\eta_{\ell_{3}}+\sum_{\mathcal{R}_{0},j_{1}\neq \ell_{1},j_{2}=\ell_{1}}\xi_{j_{1}}\xi_{j_{2}}\xi_{j_{3}}\eta_{\ell_{1}}\eta_{\ell_{2}}\eta_{\ell_{3}}\nonumber\\
&&+\sum_{\substack{\mathcal{R}_{0},j_{1}\neq\ell_{1},\\j_{2}\neq \ell_{1},j_{3}=\ell_{1}}}\xi_{j_{1}}\xi_{j_{2}}\xi_{j_{3}}\eta_{\ell_{1}}\eta_{\ell_{2}}\eta_{\ell_{3}}:=\Sigma_{1}+\Sigma_{2}+\Sigma_{3}.
 \end{eqnarray*}
For the first sum, we use \eqref{Z4} to write
\begin{equation}\label{sig1}
\Sigma_{1}=\sum_{\substack{  (j_2,j_3,\ell_2,\ell_{3})\in \mathcal Q\\ j_{1}\in\Z }}I_{j_{1}}\xi_{j_{2}}\xi_{j_{3}}\eta_{\ell_{2}}\eta_{\ell_{3}}=JZ_{4}(I)=2J^{3}-J\sum_{k\in \Z}I^{2}_{k}.
\end{equation}
Now we deal with the sum $\Sigma_{3}$. Denote by 
 \begin{equation*}
{\mathcal Q^{(k)}}=\big\{(j_1,j_2,\ell_1,\ell_{2})\in\big(\mathbb{Z}\backslash \{k\}\big)^4\;s.t. \;\;
 j_1+j_2=\ell_{1}+\ell_{2} \;\; {\rm and } \;\; j_1^2+j_2^2=\ell^{2}_{1}+\ell^{2}_{2}\big\},
\end{equation*}
then from \eqref{Z4} we deduce that 
\begin{equation*}
Z^{(k)}_{4}(I):=\sum_{(j_1,j_2,\ell_1,\ell_{2})\in \mathcal Q^{(k)}}\xi_{j_{1}}\xi_{j_{2}}\eta_{\ell_{1}}\eta_{\ell_{2}}=2(J-I_{{k}})^{2}-\sum_{j\in\Z}I_{j}^{2}+I^{2}_{k}.
\end{equation*}
Therefore by the previous equality
\begin{eqnarray}
\Sigma_{3}&=&\sum_{\substack{  (j_1,j_2,\ell_2,\ell_{3})\in \mathcal Q^{(\ell_{1})}\\ \ell_{1}\in\Z }}I_{\ell_{1}}\xi_{j_{1}}\xi_{j_{2}}\eta_{\ell_{2}}\eta_{\ell_{3}}=\sum_{k\in \Z}I_{k}Z^{(k)}_{4}(I)\nonumber\\
&=&\sum_{k\in \Z}I_{k}\Big( 2J^{2}-4JI_{k}+2I_{k}^{2}-\sum_{j\in \Z} I_{j}^{2}+I_{k}^{2}    \Big)\nonumber \\
&=&2J^{3}-5J\sum_{k\in \Z}I_{k}^{2}+3\sum_{k\in \Z}I_{k}^{3}.\label{sig3}
\end{eqnarray}
Now we consider $\Sigma_{2}$. By \eqref{Z44} and \eqref{sig3}
\begin{eqnarray}
\Sigma_{2}&=&\sum_{\Rc,j_{1}\neq j_{2}}I_{j_{2}}\xi_{j_{1}}\xi_{j_{3}}\eta_{\ell_{2}}\eta_{\ell_{3}}\nonumber\\
&=&\sum_{\Rc,j_{1}\neq j_{2}, j_{3}\neq j_{2}}I_{j_{2}}\xi_{j_{1}}\xi_{j_{3}}\eta_{\ell_{2}}\eta_{\ell_{3}}+\sum_{\Rc,j_{1}\neq j_{2}, j_{3}= j_{2}}I_{j_{2}}\xi_{j_{1}}\xi_{j_{2}}\eta_{\ell_{2}}\eta_{\ell_{3}}\nonumber\\
&=&\Sigma_{3}+\sum_{j_{2}\in \Z}I_{j_{2}}W^{j_{2}}(I)\nonumber\\
&=&2J^{3}-5J\sum_{k\in\Z}I_{k}^{2}+3\sum_{k\in \Z}I_{k}^{3}+2J\sum_{k\in \Z}I_{k}^{2}-2\sum_{k\in \Z}I_{k}^{3}\nonumber\\
&=&2J^{3}-3J\sum_{k\in \Z}I_{k}^{2}+\sum_{k\in \Z}I_{k}^{3} \label{sig2}.
\end{eqnarray}
Finally, \eqref{sig1}, \eqref{sig3} and \eqref{sig2} yield the result.
\end{proof}

\section{The model equation}\label{Sect3}
We want to describe the dynamic of a solution to \eqref{hamsys} so that $\xi_{j}^{0}=\eta_{j}^{0}=0$ when $j\not \in \mathcal{A}$. In view of the result of Propositions \ref{NF} and \ref{prop.Z6}  we hope that such a solution will be close to the solution (with same initial condition) of the Hamiltonian flow of 
$N+Z^i_{6}+Z_{6}^{e}$ reduced to the four modes of the resonant set, i.e.
\begin{equation}\label{Hat}
\wh{H}=\sum_{j\in \mathcal{A}}j^{2}I_{j}+6J^{3}-9J\sum_{k \in \mathcal{A}}I^{2}_{k}+4\sum_{k\in \mathcal{A}}I^{3}_{k}+18I_{a_{2}}^{1/2}I_{b_{2}}^{1/2}I_{a_{1}}I_{b_{1}}\cos (2\phi_{0}), 
\end{equation}
with $\phi_{0} =\theta_{a_{1}}-\theta_{b_{1}}+\frac12\theta_{a_{2}}-\frac12\theta_{b_{2}}$. \\
The Hamiltonian system associated to $\wh{H}$ is defined on the phase space $\Tc^4\times\R^4\ni (\theta_{a_1},\theta_{a_2},\theta_{a_3},\theta_{a_4};I_{a_1},I_{a_2},I_{a_3},I_{a_4})$ by
\begin{equation} \label{ham1}
\left\{\begin{array}{rrl}
\dot \theta_{a_j}=&-\frac{\partial \wh{H}}{\partial I_{a_j}} &\quad j=1,2,\\[4pt]
\dot I_{a_j}=&\frac{\partial \wh{H}}{\partial \theta_{a_j}} &\quad j=1,2,\\[4pt]
\dot \theta_{b_j}=&-\frac{\partial \wh{H}}{\partial I_{b_j}} &\quad j=1,2,\\[4pt]
\dot I_{b_j}=&\frac{\partial \wh{H}}{\partial \theta_{b_j}} &\quad j=1,2,\\[4pt]
\end{array}\right.
\end{equation}

This finite dimensional system turns out to be completely integrable.
\begin{lemm}
The system \eqref{ham1} is completely integrable.
\end{lemm}

\begin{proof}
it is straightforward to check that
\begin{equation*}
K_{1}=I_{a_{1}}+I_{b_{1}},\quad K_{2}=I_{a_{2}}+I_{b_{2}} \;\;\text{and}\;\;K_{1/2}=I_{b_{2}}+\frac12 I_{a_{1}},
\end{equation*} 
are constants of motion. 
Furthermore we verify 
\begin{equation*}
\cp{K_{1},\wh{H}}=\cp{K_{2},\wh{H}}=\cp{K_{1/2},\wh{H}}=0,
\end{equation*}
as well as
\begin{equation*}
\cp{K_{1},K_{2}}=\cp{K_{2},K_{1/2}}=\cp{K_{1/2},K_{1}}=0.
\end{equation*}
Moreover the previous quantities are independent. So $\wh{H}$ admits four integrals of motions that are independent and in involution and thus $\wh{H}$ is completely  integrable.
\end{proof}

\subsection{Action angle variables for $\wh{H}$}
In this section we construct action angle variables for  $\wh{H}$ in two particular regimes corresponding to two particular set of initial data. \\ We begin with a partial construction common to both cases.
The previous considerations suggest that we make the following symplectic  change of variables: Denote by
\begin{equation*}
\theta={^{t}}(\theta_{a_{1}},\theta_{b_{1}},\theta_{b_{2}},\theta_{a_{2}}),\quad I={^{t}}(I_{a_{1}},I_{b_{1}},I_{b_{2}},I_{a_{2}}).
\end{equation*}
Then we define the new variables 
\begin{equation*}\label{cht1}
\phi={^{t}}(\phi_{0},\phi_{1},\phi_{2},\phi_{1/2}),\quad K={^{t}}(K_{0},K_{1},K_{2},K_{1/2}),
\end{equation*}
by the linear transform
\begin{equation}\label{cht2}
{\left(\begin{array}{c} 
{\phi} \\ {K} 
\end{array} \right)}=\left(\begin{array}{cc} 
{^{t}}B^{-1} & 0 \\ 0 &  B
\end{array} \right){\left(\begin{array}{c} 
{\theta} \\ {I} 
\end{array} \right)},
\end{equation}
where the matrix $B$ is given by
\begin{equation*}\label{cht3}
B=\begin{pmatrix}1&0&0&0\\
1&1&0&0\\
0&0&1&1\\
\frac12&0&1&0\\
\end{pmatrix} \quad \text{and thus}\quad 
{}^{t}B^{-1}=\begin{pmatrix}1&-1&-\frac12&\frac12\\
0&1&0&0\\
0&0&0&1\\
0&0&1&-1\\
\end{pmatrix} .
\end{equation*}
In the new variables \eqref{ham1} reads
\begin{equation} \label{syst@}
\left\{\begin{array}{rr}
\dot \phi_{0}=&-\frac{\partial \wh{H}}{\partial K_{0}} \\[4pt]
\dot K_{0}=&\frac{\partial \wh{H}}{\partial \phi_{0}}   
\end{array}\right., \qquad 
\left\{\begin{array}{rc}
\dot \phi_{j}=&-\frac{\partial \wh{H}}{\partial K_{j}} \\[4pt]
\dot K_{j}=&0  
\end{array}\right.,\quad \text{for} \quad j =1,2,3.
\end{equation}

In the sequel, we will need the explicit expression of $\wh{H}$ in these new coordinates. Observe that for $j=1,2$ we have
\begin{equation*}
I_{a_{j}}^{2}+I_{b_{j}}^{2}=K_{j}^{2}-2I_{a_{j}}I_{b_{j}} \;\;\text{and}\;\; I_{a_{j}}^{3}+I_{b_{j}}^{3}=K_{j}(K_{j}^{2}-3I_{a_{j}}I_{b_{j}}),
\end{equation*}
then if we introduce the notation
 \begin{equation*}\label{def.f}
F(K_{1},K_{2})=K_{1}+4K_{2}+(K_{1}+K_{2})(K_{1}^{2}+K_{2}^{2}+8K_{1}K_{2}),
\end{equation*}
the Hamiltonian $\wh{H}$ reads 
\begin{multline}
\wh{H}=\wh{H}(\phi_{0},K_{0},K_{1},K_{2},K_{1/2})\label{Hat.gen}\\
\begin{aligned}
&=F(K_{1},K_{2})+6\big[(K_{1}+3K_{2})I_{a_{1}}I_{b_{1}}+(K_{2}+3K_{1})I_{a_{2}}I_{b_{2}}+3I_{a_{2}}^{\frac12}I_{b_{2}}^{\frac12}I_{a_{1}}I_{b_{1}}\cos(2\phi_{0})\big],
\end{aligned}
\end{multline}
where 
\begin{equation*}
I_{a_{1}}=K_{0},\;\;I_{b_{1}}=K_{1}-K_{0},\;\;I_{b_{2}}=K_{1/2}-\frac12K_{0},\;\; I_{a_{2}}=K_{2}-K_{1/2}+\frac12K_{0}.
\end{equation*}
We now want to exhibit some particular trajectories $(\phi_{0},K_{0})$, actually periodic orbits around stable equilibrium. For that we particularise the coefficients $K_{j}$ for $j\neq 0$.

\medskip

Let $A\geq 1/2$. We set $K_{1}=\eps^{2}$, $K_{2}=A\eps^{2}$ and $K_{1/2}=\frac12\eps^{2}$, and we denote by  
\begin{equation*}\label{HA}
\wh{H}_{0}(\phi_0,K_0):= \wh{H}(\phi_{0},K_{0},\eps^{2},A\eps^{2},\frac12\eps^{2}).
\end{equation*}
The evolution of $(\phi_0,K_0)$ is given by
\begin{equation*} \label{Hhat0}
\left\{\begin{array}{rr}
\dot \phi_{0}=&-\frac{\partial \wh{H}_0}{\partial K_{0}} \\[4pt]
\dot K_{0}=&\frac{\partial \wh{H}_0}{\partial \phi_{0}}  . 
\end{array}\right.
\end{equation*}
Then, we make the change of unknown 
\begin{equation*}
\phi_{0}(t)=\phi(\eps^{4}t)\quad \text{and}\quad K_{0}(t)=\eps^{2}K(\eps^{4}t).
\end{equation*}
An elementary computation shows that, the evolution of $(\phi,K)$ is given by 
\begin{equation*}  
\left\{\begin{array}{rl}
\dot \phi=&-\frac{\partial {H_{\star}}}{\partial K}\\[10pt]
\dot K=&\frac{\partial {H_{\star}}}{\partial \phi}.
\end{array}\right.
\end{equation*}
where
\begin{multline}\label{formule}
H_{\star}=H_{\star}(\phi,K) =\\
\begin{aligned} 
&\frac32(1-K)\Big[(A+3)(2A-1)+(7+13A)K+6(1-K)^{\frac12}(2A-1+K)^{\frac12}K\cos(2\phi)\Big].
\end{aligned}
\end{multline}

 \medskip

 \subsubsection{First regime: $A=1/2$}\label{sec.32}

In that case we have
\begin{equation*}\label{hat*}
H_{\star}=H_{\star}(\phi,K) =\frac94K(1-K)\Big[9+4K^{\frac12}(1-K)^{\frac12}\cos(2\phi)\Big],
\end{equation*}
and the  evolution of $(\phi,K)$ is given by
\begin{equation} \label{syst*}
\left\{\begin{array}{rl}
\dot \phi&=-\frac{27}4(1-2K)\Big[3+2K^{\frac12}(1-K)^{\frac12}\cos(2\phi)\Big] \\[10pt]
\dot K   &=-18K^{\frac32}(1-K)^{\frac32}\sin(2\phi).
\end{array}\right.
\end{equation}

 The dynamical system \eqref{syst*} is of pendulum type. Let us define
\begin{equation}\label{def.k}
\kappa_{\star}=\frac12-\frac18 \Big[2(7\sqrt{105}-69)\Big]^{1/2}\approx 0.208... ,
\end{equation}
we have
\begin{prop}\label{prop.mar}~
 Let $\kappa_{\star}$ be given by \eqref{def.k}.
If $\kappa_{\star}<K(0)<1-\kappa_{\star}$ and $\phi(0)=0$, then there is $T>0$ so that $(\phi,K)$ is a $2T-$periodic solution of \eqref{syst*} and 
\begin{equation*} 
K(0)+K(T)=1.
\end{equation*}  
\end{prop}

We denote by $(\phi_{\star},K_{\star})$ such a trajectory.

\figinit{pt}
\figpt 1:(-45,8)\figpt 2:(-20,15)\figpt 3:(45,-8)\figpt 4:(20,-15)
\figpt 5:(45,8)\figpt 6:(20,15)\figpt 7:(-45,-8)\figpt 8:(-20,-15)
\figpt 71:(-55,26)\figpt 72:(-27,32)\figpt 73:(55,-26)\figpt 74:(27,-32)
\figpt 75:(55,26)\figpt 76:(27,32)\figpt 77:(-55,-26)\figpt 78:(-27,-32)
\figpt 80:(-95,22)\figpt 81:(-75,22)\figpt 82:(75,22)\figpt 83:(95,22)
\figpt 84:(-95,-22)\figpt 85:(-75,-22)\figpt 86:(75,-22)\figpt 87:(95,-22)
\figpt 9:(20,-20)\figpt 10:(0,-20)\figpt 11:(-20,-20)\figpt 12:(-40,-10)
\figpt 13:(-120,-40)\figpt 14:(120,-40)
\figpt 15:(-120,40)\figpt 16:(120,40)
\figpt 17:(0,-45)\figpt 18:(0,60)
\figpt 19:(-85,-45)\figpt 20:(-85,45)
\figpt 21:(85,-45)\figpt 22:(85,45)
\figpt 23:(-120,0)\figpt 24:(120,0)
\figpt 31:(-85,0)\figpt 32:(-75,15)\figpt 33:(-25,25)\figpt 34:(0,25)
\figpt 35:(25,25)\figpt 36:(75,15)\figpt 37:(85,0)\figpt 38:(75,-15)
\figpt 39:(25,-25)\figpt 40:(0,-25)\figpt 41:(-25,-25)\figpt 42:(-75,-15)
\figpt 50:(0,-40) \figpt 51:(120,-40) \figpt 52:(5,60)  \figpt 53:(0,0)
 \figpt 54:(-85,0)  \figpt 55:(87,0) \figpt 56:(85,0) \figpt 57:(0,0)
 \figpt 100:(0,-25)  \figpt 101:(85,-40)  \figpt 102:(-85,-40)  \figpt 200:(-65,30)\figpt 220:(0,46)
 
\psbeginfig{}
 \psset(width=1)
 \psarrow[1,2]
  \psarrow[3,4]
   \psarrow[6,5]
   \psarrow[8,7]
   \psarrow[71,72]
  \psarrow[73,74]
   \psarrow[76,75]
   \psarrow[78,77]
    \psarrow[80,81]
        \psarrow[82,83]
         \psarrow[85,84]
        \psarrow[87,86]
\psset(dash=2)
\psBezier 4[31,32,33,34,35,36,37,38,39,40,41,42,31]
  \psset(width=\defaultwidth)
\psline[15,16]
\psline[19,20]
\psline[21,22]
\psset(dash=5)
\psline[23,24]
\psset(dash=\defaultdash) 
\psarrow[13,14]
\psarrow[17,18]
\psendfig
\figvisu{\figBoxA}{ {\bf Figure 1 : }The phase portrait of system \eqref{syst*}}{
\figwritese 50:$O$(3pt) \figwritese 51:$\phi$(3pt)
\figwritese 52:$K$(3pt)
\figwritese 53:$\omega_{0}$(3pt)
\figwritesw 54:$\omega_{1}$(3pt)
\figwritese 55:$\omega_{2}$(3pt)
\figwritese 100:$\kappa_{\star}$(3pt)
\figwrites 101:$\pi/2$(3pt)
\figwrites 102:$-\pi/2$(3pt)
\figwritese 220:$1$(3pt)
\figsetmark{$\times$}
\figwritep[57,54,56,100]
}
\centerline{\box\figBoxA}

\begin{proof} The line $K=0$ and $K=1$ are barriers  and  the phase portrait is $\pi$-periodic in $\varphi$ so
 we  restrict our study to the region  $-\frac{\pi}2\leq \phi\leq \frac{\pi}2$, $0<K<1$. In this domain, there are exactly three equilibrium points : $\omega_{0}=(0,1/2)$ which is a centre and $\omega_{1}=(-\pi/2,1/2)$ and $\omega_{2}=(\pi/2,1/2)$ which are saddle points. 
The level set $H_{\star}(\phi,K)=H_{\star}(\omega_{1})=H_{\star}(\omega_{2})={63}/{16}$, which corresponds to the equation
\begin{equation*} 
K(1-K)\big(9+4K^{\frac12}(1-K)^{\frac12}\cos(2\phi)\big)=\frac{7}{4},
\end{equation*}
 defines two heteroclinic orbits which link the points $\omega_{1}$ and $\omega_{2}$ : $\mathcal{C}_{1}$ in the region $\{K<1/2\}$ and $\mathcal{C}_{2}$ in the region  $\{K>1/2\}$ (see the dashed curves in Figures 1\&2). Moreover, we can explicitly compute the intersection $(0,\kappa_{\star})$ of the curve $\mathcal{C}_{1}$ with the $K-$axis, and we obtain \eqref{def.k}.\\
Let $\widetilde{U}\subset ]-\pi/2,\pi/2[\times ]\kappa_{\star},1-\kappa_{\star}[\backslash \{\omega_{0}\}$ be the open domain delimited by the curves $\mathcal{C}_{1}$ and $\mathcal{C}_{2}$ minus the point $\omega_{0}$. Any solution  issued from a point inside $\widetilde{U}$ is periodic and turns around the centre $\omega_0$. Furthermore, let $2T$ be the period, by symmetry we have $(\phi(T),K(T))=(0,1-K(0))$.

\end{proof}
\begin{figure}[h]
\end{figure}

\figinit{pt}
\figpt 1:(-40,0)\figpt 2:(-40,10)\figpt 3:(-20,20)\figpt 4:(0,20)
\figpt 5:(20,20)\figpt 6:(40,10)\figpt 7:(40,0)\figpt 8:(40,-10)
\figpt 9:(20,-20)\figpt 10:(0,-20)\figpt 11:(-20,-20)\figpt 12:(-40,-10)
\figpt 13:(-120,-40)\figpt 14:(120,-40)
\figpt 15:(-120,40)\figpt 16:(120,40)
\figpt 17:(0,-45)\figpt 18:(0,60)
\figpt 19:(-85,-45)\figpt 20:(-85,45)
\figpt 21:(85,-45)\figpt 22:(85,45)
\figpt 23:(-120,0)\figpt 24:(120,0)
\figpt 31:(-85,0)\figpt 32:(-75,15)\figpt 33:(-25,25)\figpt 34:(0,25)
\figpt 35:(25,25)\figpt 36:(75,15)\figpt 37:(85,0)\figpt 38:(75,-15)
\figpt 39:(25,-25)\figpt 40:(0,-25)\figpt 41:(-25,-25)\figpt 42:(-75,-15)
\figpt 50:(0,-40) \figpt 51:(120,-40) \figpt 52:(5,60)  \figpt 53:(0,0)
 \figpt 54:(-85,0)  \figpt 55:(87,0) \figpt 56:(85,0) \figpt 57:(0,0)
  \figpt 100:(0,-25)   \figpt 110:(26,13)   \figpt 120:(95,28) 
   \figpt 101:(85,-40)  \figpt 102:(-85,-40) \figpt 200:(-65,30) \figpt 210:(-65,-30) \figpt 220:(0,46)
 
\psbeginfig{}
 \psset(width=1)
\psBezier 4[1,2,3,4,5,6,7,8,9,10,11,12,1]
\psset(dash=2)
\psBezier 4[31,32,33,34,35,36,37,38,39,40,41,42,31]
  \psset(width=\defaultwidth)
\psline[15,16]
\psline[19,20]
\psline[21,22]
\psset(dash=5)
\psline[23,24]
\psset(dash=\defaultdash) 
\psline[110,120]
\psarrow[13,14]
\psarrow[17,18]
\psendfig
\figvisu{\figBoxA}{ {\bf Figure 2 : }An example of trajectory $(\phi_{\star},K_{\star})$}{
\figwritese 50:$O$(3pt) \figwritese 51:$\phi$(3pt)
\figwritese 52:$K$(3pt)
\figwritese 53:$\omega_{0}$(3pt)
\figwritesw 54:$\omega_{1}$(3pt)
\figwritese 55:$\omega_{2}$(3pt)
\figwritese 100:$\kappa_{\star}$(3pt)
\figwritee 120:$\(\phi_{\star},K_{\star}\)$(3pt)
\figwrites 101:$\pi/2$(3pt)
\figwrites 102:$-\pi/2$(3pt)
\figwrites 200:$\mathcal{C}_{2}$(3pt)
\figwriten 210:$\mathcal{C}_{1}$(3pt)
\figwritese 220:$1$(3pt)
\figsetmark{$\times$}
\figwritep[57,54,56]
}
\centerline{\box\figBoxA}
~\\
By applying the Arnold-Liouville theorem (see e.g. \cite{Arn}) inside  $\widetilde{U}$ we obtain
\begin{lemm}\label{Arnold}
Let $U\subset \subset \widetilde{U}$, then  there exists a symplectic change of variables 
$
\Phi:U\ni(K,\phi)  \longmapsto (L,\alpha)\in \R_{>0}\times\T$
which defines  action angle coordinates for \eqref{syst*} i.e., \eqref{syst*} is equivalent to the system 
\begin{equation*}
\dot L=-\frac{\partial {H_{\star}}}{\partial \alpha}=0,\quad 
\dot \alpha=\frac{\partial {H_{\star}}}{\partial L}.
\end{equation*}
Moreover $\Phi$ is a $\mathcal{C}^{1}$-diffeomorphism, and there exists $C>0$ depending on $U$ so that 
$$\|\text{d}\Phi\|\leq C,\quad\|\text{d}\Phi^{-1}\|\leq C.$$
\end{lemm}

 \subsubsection{Second regime: A=4}
 In that case we obtain
 \begin{equation*}
 H_{\star}=\frac32(1-K)\Big[49+59K+6(1-K)^{\frac12}(7+K)^{\frac12}K\cos(2\phi)\Big],
 \end{equation*}
 and the evolution of $(\phi,K)$ is given by
 \begin{equation}\label{nouv.syst}
\left\{\begin{array}{rl}
\dot \phi=&3\Big[59K-5-3(K+7)^{-\frac12}(1-K)^{\frac12}(-3K^{2}-16K+7)\cos(2\phi)\Big] \\[10pt]
\dot K=&-18(1-K)^{\frac32} (7+K)^{\frac12}K\sin(2\phi).
\end{array}\right.
\end{equation}
\begin{prop}~
\item Let $\gamma>0$ arbitrary small, and set $(\phi(0),K(0))=(0,\gamma)$. Then there is $T_{\gamma}>0$ so that $(\phi,K)$ is $2T_{\gamma}-$periodic and 
\begin{equation*} 
K(T_{\gamma})>\frac1{10}.
\end{equation*}  
\end{prop}
We denote by $(\phi_{\star},K_{\star})$ such a trajectory.\\

\bigskip

\figinit{pt}
\figpt 1:(-45,8)\figpt 2:(-20,15)\figpt 3:(45,-8)\figpt 4:(20,-15)
\figpt 5:(40,16)\figpt 6:(12,24)\figpt 7:(-40,16)\figpt 8:(-12,24)
\figpt 71:(40,-18)\figpt 72:(12,-6)\figpt 73:(-40,-18)\figpt 74:(-12,-6)
\figpt 75:(40,-30)\figpt 76:(12,-32)\figpt 77:(-40,-30)\figpt 78:(-12,-32)
\figpt 80:(-95,22)\figpt 81:(-75,22)\figpt 82:(75,22)\figpt 83:(95,22)
\figpt 84:(-95,-22)\figpt 85:(-75,-22)\figpt 86:(75,-22)\figpt 87:(95,-22)
\figpt 9:(20,-20)\figpt 10:(0,-20)\figpt 11:(-20,-20)\figpt 12:(-40,-10)
\figpt 13:(-120,-40)\figpt 14:(120,-40)
\figpt 15:(-120,40)\figpt 16:(120,40)
\figpt 17:(0,-45)\figpt 18:(0,60)
\figpt 19:(-85,-45)\figpt 20:(-85,45)
\figpt 21:(85,-45)\figpt 22:(85,45)
\figpt 23:(-120,0)\figpt 24:(120,0)
\figpt 31:(-70,-40)\figpt 32:(-50,-20)\figpt 33:(-8,-18)\figpt 34:(0,-18)
\figpt 35:(8,-18)\figpt 36:(50,-20)\figpt 37:(70,-40)
\figpt 131:(-70,-40)\figpt 132:(-45,8)\figpt 133:(-7,8)\figpt 134:(0,8)
\figpt 135:(7,8)\figpt 136:(45,8)\figpt 137:(70,-40)
\figpt 50:(0,-40) \figpt 51:(120,-40) \figpt 52:(5,60)  \figpt 53:(0,-18)
 \figpt 54:(-70,-40)  \figpt 55:(78,-43) \figpt 56:(70,-40) \figpt 57:(0,-18) \figpt 58:(-70,-44)
  \figpt 101:(90,-40)  \figpt 102:(-92,-40)  \figpt 200:(-65,30)\figpt 220:(0,46)
 
\psbeginfig{}
 \psset(width=1)
   \psarrow[6,5]
   \psarrow[7,8]
   \psarrow[72,71]
 \psarrow[73,74]
   \psarrow[75,76]
   \psarrow[78,77]
    \psarrow[80,81]
        \psarrow[82,83]
         \psarrow[84,85]
        \psarrow[86,87]
\psset(dash=5)
\psBezier 2[31,32,33,34,35,36,37]
\psset(dash=2)
\psBezier 2[131,132,133,134,135,136,137]
\psline[54,56]
  \psset(width=\defaultwidth)
\psline[15,16]
\psline[19,20]
\psline[21,22]
\psset(dash=\defaultdash) 
\psarrow[13,14]
\psarrow[17,18]
\psendfig
\figvisu{\figBoxA}{ {\bf Figure 3 : }The phase portrait of system \eqref{nouv.syst}}{
\figwritese 50:$O$(3pt) \figwritese 51:$\phi$(3pt)
\figwritese 52:$K$(3pt)
\figwritese 53:$\omega_{0}$(3pt)
\figwrites 58:$\omega_{1}$(3pt)
\figwritesw 55:$\omega_{2}$(3pt)
\figwrites 101:$\pi/2$(3pt)
\figwrites 102:$-\pi/2$(3pt)
\figwritese 220:$1$(3pt)
\figsetmark{$\times$}
\figwritep[57,54,56]
}
\centerline{\box\figBoxA}
~\\

\begin{proof}
 We  restrict our study to the region  $0\leq \phi\leq \frac{\pi}2$, $0<K<1$. First, we study the sign of $\dot{\phi}$. To begin with, observe that $\dot{\phi}$ has exactly the sign of  $f(K)-\cos(2\phi)$ where
 \begin{equation*}
 f(K)=\frac13(59K-5)(K+7)^{\frac12}(1-K)^{-\frac12}(-3K^{2}-16K+7)^{-1}.
 \end{equation*}
 We verify  that there exists $1/10<\kappa_{0}<1/5$ so that the function $f$ is increasing and one to one $f:[0,\kappa_{0}]\longrightarrow [-5\sqrt{7}/21,1]$.  Thus, the curve  $\mathcal{C}_{0}:=\{\dot{\phi}=0\}$ can be expressed as a decreasing function $K(\phi)=f^{-1}\big(\cos(2\phi)\big)$.\\
 Thanks to this study, and the expression of $\dot{K}$, we deduce that the phase portrait has exactly three equilibrium points :  $\omega_{0}=(0,\kappa_{0})$ which is a centre and $\omega_{1}=(-\phi_{0},0)$ and $\omega_{2}=(\phi_{0},0)$ which are saddle points (here $0<\phi_{0}<\pi/2$ is defined by the equation $f(0)=-5\sqrt{7}/21=\cos(2\phi_{0})$). 
The level set $H_{\star}(\phi,K)=H_{\star}(\omega_{1})=H_{\star}(\omega_{2})=\frac{3}2 \cdot49,$   which is defined by the equation 
 \begin{equation*}
 10-59K+6(1-K)^{\frac32}(7+K)^{\frac12}\cos(2\phi)=0,
 \end{equation*}
 defines two heteroclinic orbits $\mathcal{C}_{1}:=\{K=0\}$ and $\mathcal{C}_{2}$ that link the two saddle points (see the dashed curves in Figures 3\&4).  \\
 Let $\widetilde{U}_2\subset ]-\pi/2,\pi/2[\times ]0,1[\backslash \{\omega_{0}\}$ be the open domain delimited by the curves $\mathcal{C}_{1}$ and $\mathcal{C}_{2}$ minus the point $\omega_{0}$. Any solution  issued from a point inside $\widetilde{U}_2$ is periodic and turns around the centre $\omega_0$. Furthermore, let $2T$ be the period, by symmetry we have $\phi(T)=0$ and $K(T)>\kappa_0$.

\end{proof}
\figinit{pt}
\figpt 1:(-45,8)\figpt 2:(-20,15)\figpt 3:(45,-8)\figpt 4:(20,-15)
\figpt 5:(40,16)\figpt 6:(12,24)\figpt 7:(-40,16)\figpt 8:(-12,24)
\figpt 71:(40,-18)\figpt 72:(12,-6)\figpt 73:(-40,-18)\figpt 74:(-12,-6)
\figpt 75:(40,-30)\figpt 76:(12,-32)\figpt 77:(-40,-30)\figpt 78:(-12,-32)
\figpt 80:(-95,22)\figpt 81:(-75,22)\figpt 82:(75,22)\figpt 83:(95,22)
\figpt 84:(-95,-22)\figpt 85:(-75,-22)\figpt 86:(75,-22)\figpt 87:(95,-22)
\figpt 9:(20,-20)\figpt 10:(0,-20)\figpt 11:(-20,-20)\figpt 12:(-40,-10)
\figpt 13:(-120,-40)\figpt 14:(120,-40)
\figpt 15:(-120,40)\figpt 16:(120,40)
\figpt 17:(0,-45)\figpt 18:(0,60)
\figpt 19:(-85,-45)\figpt 20:(-85,45)
\figpt 21:(85,-45)\figpt 22:(85,45)
\figpt 23:(-120,0)\figpt 24:(120,0)
\figpt 31:(-70,-40)\figpt 32:(-50,-20)\figpt 33:(-8,-18)\figpt 34:(0,-18)
\figpt 35:(8,-18)\figpt 36:(50,-20)\figpt 37:(70,-40)
\figpt 131:(-70,-40)\figpt 132:(-45,8)\figpt 133:(-7,8)\figpt 134:(0,8)
\figpt 135:(7,8)\figpt 136:(45,8)\figpt 137:(70,-40)
\figpt 600:(-40,-40)\figpt 210:(-48,-9)    \figpt 700:(-20,-15)
\figpt 231:(-62,-32)\figpt 232:(-48,-9)\figpt 233:(-25,5)\figpt 234:(0,5)
\figpt 235:(25,5)\figpt 236:(48,-9)\figpt 237:(62,-32) \figpt 238:(-15,-38) \figpt 239:(0,-38) \figpt 240:(15,-38)
\figpt 331:(-60,-37) \figpt 337:(60,-37)
\figpt 241:(-50,-38) \figpt 242:(50,-38)
\figpt 50:(0,-40) \figpt 51:(120,-40) \figpt 52:(5,60)  \figpt 53:(0,-18)
 \figpt 54:(-70,-40)  \figpt 55:(78,-43) \figpt 56:(70,-40) \figpt 57:(0,-18) \figpt 58:(-70,-44)
  \figpt 101:(90,-40)  \figpt 102:(-92,-40)  \figpt 200:(-65,30)\figpt 220:(0,46)
   \figpt 400:(-66,-38)  \figpt 401:(65,-40)   \figpt 299:(-65,-40) \figpt 402:(66,-38) \figpt 437:(0,-38) 
   \figpt 500:(27,0) 
 
\psbeginfig{}
 \psset(width=1)
\psset(dash=5)
\psBezier 2[31,32,33,34,35,36,37]
\psset(dash=2)
\psBezier 2[131,132,133,134,135,136,137]
\psline[54,56]
\psset(dash=\defaultdash)
\psBezier 2[231,232,233,234,235,236,237]
\psBezier 1[231,299,400,437]
\psBezier 1[237,401,402,437]
\psset(dash=2)
\psset(width=\defaultwidth)
\psline[15,16]
\psline[19,20]
\psline[21,22]
\psset(dash=\defaultdash) 
\psline[120,500]
\psarrow[13,14]
\psarrow[17,18]
\psendfig
\figvisu{\figBoxA}{ {\bf Figure 4 : }An example of trajectory $(\phi_{\star},K_{\star})$}{
\figwritese 50:$O$(3pt) \figwritese 51:$\phi$(3pt)
\figwritese 52:$K$(3pt)
\figwritese 53:$\omega_{0}$(3pt)
\figwrites 58:$\omega_{1}$(3pt)
\figwritesw 55:$\omega_{2}$(3pt)
\figwrites 101:$\pi/2$(3pt)
\figwrites 102:$-\pi/2$(3pt)
\figwrites 700:$\mathcal{C}_{0}$(3pt)
\figwrites 600:$\mathcal{C}_{1}$(3pt)
\figwriten 210:$\mathcal{C}_{2}$(3pt)
\figwritee 120:$\(\phi_{\star},K_{\star}\)$(3pt)
\figwritese 220:$1$(3pt)
\figsetmark{$\times$}
\figwritep[57,54,56]
}
\centerline{\box\figBoxA}
~\\
As in the first case, applying the Arnold-Liouville theorem inside  $\widetilde{U}$ we obtain
\begin{lemm}\label{Arnold2}
Let  $U\subset \subset \widetilde{U}_2$, then  there exists a symplectic change of variables 
$
\Phi:U\ni(K,\phi)  \longmapsto (L,\alpha)\in \R_{>0}\times\T$
which defines  action angle coordinates for \eqref{nouv.syst} i.e., \eqref{nouv.syst} is equivalent to the system 
\begin{equation*}
\dot L=-\frac{\partial {H_{\star}}}{\partial \alpha}=0,\quad 
\dot \alpha=\frac{\partial {H_{\star}}}{\partial L}.
\end{equation*}
Moreover $\Phi$ is a $\mathcal{C}^{1}$-diffeomorphism, and there exists $C>0$ depending on $U$ so that 
$$\|\text{d}\Phi\|\leq C,\quad\|\text{d}\Phi^{-1}\|\leq C.$$
\end{lemm}

 \subsubsection{On the other cases.} More generally, we can consider the case $K_{1}=\eps^{2}$, $K_{2}=A\eps^{2}$ and $K_{1/2}=B\eps^{2}$, where the constants $A,B>0$ satisfy the natural conditions $A\geq B$ and $B\geq 1/2$. Roughly speaking, the mechanism is the following. Consider the curve $\mathcal{C}_{0}=\{\dot{\phi}=0\}$. If $\mathcal{C}_{0}$ has no intersection with $\{K=0\}$ and $\{K=1\}$, then the dynamic is essentially the one of $A=1/2$, $B=1/2$. On the contrary, if $\mathcal{C}_{0}$ has two intersections with $\{K=0\}$ or $\{K=1\}$, then the dynamic is essentially the one of $A=4$, $B=1/2$.

\section{Proof of Theorems \ref{thm1} and \ref{thm2} }\label{Sect4}
Consider the Hamiltonian $\ov{H}$ given by \eqref{ham*}, which is a function of $\big(\xi_{j},\eta_{j}\big)_{j\in \Z}$. We make the linear change of variables given by  \eqref{cht2} (the variables $\xi_{j},\eta_{j}$ remain unchanged for $p\notin \mathcal{A}$). In the sequel, the Hamiltonian in the new variables is still denoted by $\ov{H}$. Then $\ov{H}$ induces the system
\begin{equation} \label{Bourg}
\left\{\begin{array}{rr}
\dot \phi_{j}=&-\frac{\partial \ov{H}}{\partial K_{j}} \\[4pt]
\dot K_{j}=&\frac{\partial \ov{H}}{\partial \phi_{j}}   
\end{array}\right., \qquad 
 \left\{\begin{array}{rr}
i\dot {\xi}_{p}=&\frac{\partial \ov{H}}{\partial \eta_{p}} \\[4pt]
i\dot{\eta}_{p}=&-\frac{\partial \ov{H}}{\partial \xi_{p}}   
\end{array}\right.,p\notin\mathcal{A}. \qquad 
\end{equation}
Next, we take some initial conditions to \eqref{Bourg} which will be close to the initial conditions chosen for \eqref{syst@}.

Observe that the $K_{j}$'s aren't constants of motion of \eqref{Bourg}. However, they are almost preserved, and this is the result of the next lemma.
Recall that  $\mathcal{A}=\{a_{2},a_{1},b_{2},b_{1}\}$, 
$$K_{1}=I_{a_{1}}+I_{b_{1}},\;\;K_{2}=I_{a_{2}}+I_{b_{2}},\;\; \text{and}\;\;K_{1/2}=I_{b_{2}}+\frac12 I_{a_{1}},$$
and recall the notations of Proposition \eqref{prop.Z6}.

 \begin{lemm}\label{lem.dyn}
 Assume that 
 \begin{equation}\label{assp}
 \xi_{j}(0),\eta_{j}(0)=\mathcal{O}(\eps),\;\forall \,j\in \mathcal{A}\;\;\text{and}\;\;\xi_{p}(0),\eta_{p}(0)=\mathcal{O}(\eps^{3}),\;\forall \,p\not\in \mathcal{A}.
 \end{equation}
Then for all $0\leq t \leq C\eps^{-6}$, 
 \begin{equation}\label{pp}
 I_{p}(t)=\mathcal{O}(\eps^{6})\;\;\text{when}\;\; p\not\in \mathcal{A},
 \end{equation}
and
  \begin{eqnarray}
 K_{1}(t)&=&K_{1}(0)+\mathcal{O}(\eps^{10 })t\label{ip1} \\
 K_{2}(t)&=&K_{2}(0)+\mathcal{O}(\eps^{10 })t\label{ip2}\\
 K_{1/2}(t)&=&K_{1/2}(0)+\mathcal{O}(\eps^{10 })t.\label{ip3}
 \end{eqnarray}
 \end{lemm}

 \begin{proof}
 We first remark that by the preservation of the $L^2$ norm in the NLS equation, we have
 $$\sum_{p\in\Z}I_p(t) =\sum_{p\in\Z}I_p(0)\quad \text{ for all }t\in \R,$$
 and therefore by using \eqref{assp}
\begin{equation}\label{eps2}
{I_{p}}(t)=\mathcal O (\eps^2) \quad \text{for all } p\in\Z \text{ and for all }  t.
\end{equation}
 On the other hand by Propositions \ref{NF} and \ref{prop.Z6}, we have for $p\in \Z$
\begin{equation}\label{eq0}
\dot{I_{p}}=\cp{I_{p},\ov{H}}=\cp{I_{p},Z_{6}^{e}}+\cp{I_{p},Z_{6,2}}+\cp{I_{p},Z_{6,3}}+\cp{I_{p},R_{10}}.
\end{equation}

$\bullet$ To prove \eqref{pp}, we first verify that for $p\notin \A$, $\cp{I_{p},Z_{6}^{e}}=0$. Then,  we remark that, as a consequence of Lemma \ref{lem.couple}, all the monomials appearing in $Z_{6,2}$ have the form
 \begin{equation*}
 \xi_{j_{1}}\xi_{j_{2}}\xi_{p_{1}}\eta_{\ell_{1}}\eta_{\ell_{2}}\eta_{p_{2}}\quad \text{ or }\quad \xi_{\ell_{1}}\xi_{\ell_{2}}\xi_{p_{2}}\eta_{j_{1}}\eta_{j_{2}}\eta_{p_{1}}
 \end{equation*}
where $(j_1,j_2,p_1,\ell_1,\ell_2,p_2)\in \mathcal{R}$, $j_1,j_2,\ell_1,\ell_2\in\A$ and $p_1,p_2\notin \A$.
Furthermore, by straightforward computation,
 \begin{equation}\label{sym}\cp{I_{p_{1}}+I_{p_{2}},\xi_{j_{1}}\xi_{j_{2}}\xi_{p_{1}}\eta_{\ell_{1}}\eta_{\ell_{2}}\eta_{p_{2}}}=
 \cp{I_{p_{1}}+I_{p_{2}},\xi_{\ell_{1}}\xi_{\ell_{2}}\xi_{p_{2}}\eta_{j_{1}}\eta_{j_{2}}\eta_{p_{1}}}=0.\end{equation}
 Then we define an equivalence relation : Let $p,\widetilde{p}\not \in \mathcal{A}$. We say that $p$ and $\widetilde{p}$ are linked and write $p\leftrightarrow\wt{p}$ if there exist $k\in \N^{*}$, a sequence $(q^{(i)})_{1\leq i\leq k}\not \in \mathcal{A}$ so that $q^{(1)}=p$, ${q}^{(k)}=\wt{p}$ and $j_{1}^{(i)},j_{2}^{(i)},\ell_{1}^{(i)},\ell_{2}^{(i)}\in \mathcal{A}$ satisfying 
 \begin{equation*}
 (j_{1}^{(i)},j_{2}^{(i)},q^{(i)},\ell_{1}^{(i)},\ell_{2}^{(i)},{q}^{(i+1)})\in \mathcal{R},\quad \text{for all }\quad 1\leq i\leq k-1.
 \end{equation*}
 For $p\not\in \mathcal{A}$, we define $\dis J_{p}=\sum_{q\leftrightarrow p}I_{q}$. We stress out that $J_p$ is a sum of positive quantities, one of them being $I_p$. So the control of $J_p$ induces the control of $I_p$.\\  In view of \eqref{sym} we have
 \begin{equation*}
 \cp{J_{p},Z_{6,2}}=0
\end{equation*}
and thus 
\begin{equation*}
\dot{J}_{p}=\cp{J_{p},Z_{6,3}}+\cp{J_{p},R_{10}},\;\;\text{when}\;\;p\not \in \mathcal{A}.
\end{equation*}
Furthermore all the monomials appearing in $\cp{J_{p},R_{10}}$ are of order 10 and contains at least one mode out of $\mathcal{A}$.
Therefore as soon as \eqref{pp} remains valid, we have $$\dot{J}_{p}(t)=\mathcal O(\eps^{3+3\times 3})+\mathcal O(\eps^{9+3})$$ and thus $$|{J}_{p}(t)| = \mathcal O(\eps^{6})+t\ \mathcal O(\eps^{12}).$$ We then conclude by a classical bootstrap argument that \eqref{pp} holds true for $t\leq C \eps^{-6}$.

$\bullet$ It remains to prove \eqref{ip1}-\eqref{ip3}. Again this is proved by a bootstrap argument. To begin with, we verify by direct calculation that for all $p\in \big\{ 1/2,1,2\big\}$, 
$$
 \cp{K_{p},Z_{6}^{e}}=0.
$$
 Therefore, by using \eqref{eq0}  we deduce that for all $p\in \big\{ 1/2,1,2 \big\}$
 \begin{equation}\label{eqx}
 \dot{K}_{p}=\cp{K_{p},Z_{6,2}}+\cp{K_{p},Z_{6,3}}+\cp{K_{p},R_{10}}.
 \end{equation}
 Then we use that each monomial of $Z_{6,2}$ contains at least two terms with indices $p'\not\in \mathcal{A}$ (see Proposition \ref{prop.Z6}). Therefore, as soon as \eqref{assp} holds, we have $|\cp{K_{p},Z_{6,2}}|\leq C\eps^{10}$. Furthermore $|\cp{K_{p},R_{10}}|\leq C\eps^{10}$. Therefore, by \eqref{eqx},
 $$K_{p}(t)= K_{p}(0)+t\ \mathcal{O}(\eps^{10}).$$
  Finally, to recover the bounds \eqref{assp}, we have to demand that $t$ is  so that  $0\leq t\leq \eps^{-6}$, which was the claim.
 \end{proof} 
 
From now, we   fix the initial conditions
 \begin{equation} \label{inic}
\begin{array}{l}
K_{1}(0)=\eps^{2},\;\;K_{2}(0)=A\eps^{2},\;\;K_{1/2}(0)=\eps^{2}/2, \\[4pt]
\text{and} \;\;|\xi_{j}(0)|, |\eta_{j}(0)| \leq C\eps^{3}\;\; \text{for} \;\;j\notin \mathcal{A}.
\end{array} 
\end{equation}

 Let  $\ov{H}$  be given by \eqref{ham*}. Then according to the result of Lemma \ref{lem.dyn} which says that for a suitable long time we remain close to the regime of Section \ref{Sect3}, we hope that we can write $\ov{H}=\wh{H}_{0}+R$, where $R$ is an error term which remains small for times $0\leq t\leq \eps^{-6}$.
 
 We focus on the motion of $(\phi_{0},K_{0})$  and as in the previous section, we make the change of unknown
\begin{equation}\label{renormal}
\phi_{0}(t)=\phi(\eps^{4}t)\quad \text{and}\quad K_{0}(t)=\eps^{2}K(\eps^{4}t),
\end{equation}
and we work with the scaled time variable $\tau=\eps^{4}t$. Then we can state 
\begin{prop} Consider the solution \eqref{Bourg} with the initial conditions \eqref{inic}. Then $(\phi,K)$ defined by \eqref{renormal} satisfies for $0\leq \tau \leq \eps^{-2}$
\begin{equation} \label{syst.prop}
\left\{\begin{array}{rr}
\dot \phi=&-\frac{\partial H_{\star}}{\partial K}+\mathcal{O}(\eps^{2})\\[5pt]
\dot K=&\frac{\partial H_{\star}}{\partial \phi} +\mathcal{O}(\eps^{2})  ,
\end{array}\right. 
\end{equation}
where $H_{\star}$ is the Hamiltonian \eqref{formule}
$$ H_{\star}=\frac32(1-K)\Big[(A+3)(2A-1)+(7+13A)K+6(1-K)^{\frac12}(2A-1+K)^{\frac12}K\cos(2\phi)\Big]. $$
\end{prop}

 \begin{proof}
  First recall that   $\dis \wh{H}=\wh{H}(\phi_{0},K_{0},K_{1},K_{2},K_{1/2})$ is the reduced Hamiltonian given by \eqref{Hat.gen}. By Propositions \ref{NF} and \ref{prop.Z6} we  have
\begin{equation}\label{413}
\ov{H}=\wh{H}+R_{I}+Z_{6,2}+Z_{6,3}+R_{10},
\end{equation}
where $R_{I}$ is the polynomial function of the actions $I_{j}$ defined by (recall that $J=\sum_{k\in \N}K_{p}$)
\begin{eqnarray*}
R_{I}&=&6\big(J^{3}-(K_{1}+K_{2})^{3}\big)-9J\sum_{k\in \Z}I_{k}^{2}+9(K_{1}+K_{2})\sum_{k\in \mathcal{A}}I_{k}^{2}+\nonumber \\
&&+\sum_{j\notin \mathcal{A}}j^{2}I_{j}+4\sum_{k\not\in \mathcal{A}}I_{k}^{3}.
\end{eqnarray*}
Notice that $R_I$ vanishes when $I_k=0$ for all $k\notin \A$ since $R_I$ is in fact the part of $N+Z_6^i$ that does not depend only on the internal variables $(I_k)_{k\in\A}$.\\
 Thanks to  the Taylor formula there is $Q$ so that 
\begin{eqnarray}\label{415}
\wh{H}(\phi_{0},K_{0},K_{1},K_{2},K_{1/2})&=&\wh{H}(\phi_{0},K_{0},\eps^{2},A\eps^{2},\eps^{2}/2)+Q\nonumber \\
&=& \wh{H}_{0}+Q.
\end{eqnarray}
Thus, by \eqref{413} and \eqref{415} we have $\ov{H}=\wh{H}_{0}+R$ with 
$$R=Q+R_{I}+Z_{6,2}+Z_{6,3}+R_{10}.$$
By \eqref{Bourg}, $(\phi_{0},K_{0})$ satisfies the system
\begin{equation*} 
\left\{\begin{array}{rr}
\dot \phi_{0}(t)=&-\frac{\partial \ov{H}}{\partial K_{0}}(\phi_{0}(t),K_{0}(t),\dots) \\[5pt]
\dot {K}_{0}(t)=&\frac{\partial \ov{H}}{\partial \phi_{0}}(\phi_{0}(t),K_{0}(t),\dots)   ,
\end{array}\right. 
\end{equation*}
where the dots stand  for the dependance of the Hamiltonian on the other coordinates.  Then, after the change of variables \eqref{renormal} we obtain
\begin{equation*} 
\left\{\begin{array}{rr}
\dot \phi(\tau)=&-\frac1{\eps^{6}}\frac{\partial \ov{H}}{\partial K}(\phi(\tau),\eps^{2}K(\tau),\dots) \\[5pt]
\dot {K}(\tau)=&\frac1{\eps^{6}}\frac{\partial \ov{H}}{\partial \phi}(\phi(\tau),\eps^{2}K(\tau),\dots).
\end{array}\right. 
\end{equation*}
Now write  $\ov{H}=\wh{H}_{0}+R$ and observe that $\wh{H}_{0}(\phi,\eps^{2}K)=C_{\eps}+\eps^{6}H_{\star}(\phi,K)$. As a consequence, $(\phi,K)$ satisfies 
\begin{equation*} 
\left\{\begin{array}{rr}
\dot {\phi} =&-\frac{\partial H_{\star}}{\partial K}-\frac1{\eps^{6}}\frac{\partial R(\phi,\eps^{2}K,\dots)}{\partial K}\\[5pt]
\dot {K} =&\frac{\partial H_{\star}}{\partial \phi} +\frac1{\eps^{6}}\frac{\partial R(\phi,\eps^{2}K,\dots)}{\partial \phi}.
\end{array}\right. 
\end{equation*}
Thus it remains to estimate $\partial_{\phi} R(\phi,\eps^{2}K,\dots)$ and $\partial_{K} R(\phi,\eps^{2}K,\dots)$. Remark that $\phi$ and $K$ are dimensionless variables. Thus, if $P$ is a polynomial involving $p$ internal modes, $(\xi_j,\eta_j)_{j\in\A}$, and $q$ external modes, $(\xi_j,\eta_j)_{j\notin\A}$, we have by using  Lemma \ref{lem.dyn}
$$\partial_{\phi} P(\phi,\eps^{2}K,\dots)=\mathcal{O}(\eps^{p+3q}), \quad \partial_{K} P(\phi,\eps^{2}K,\dots)=\mathcal{O}(\eps^{p+3q}).$$
Then notice that  $R_I$ contains only monomials involving at least one external actions $(I_k)_{k\notin \A}$. Therefore we get
\begin{eqnarray*}
\partial_{\phi}R_{I}(\phi,\eps^{2}K,\dots)&=\mathcal{O}(\eps^{10}),\quad \partial_{K}R_{I}(\phi,\eps^{2}K,\dots)&=\mathcal{O}(\eps^{10}),\\
\partial_{\phi}Z_{6,2}(\phi,\eps^{2}K,\dots)&=\mathcal{O}(\eps^{10}),\quad  \partial_{K}Z_{6,2}(\phi,\eps^{2}K,\dots)&=\mathcal{O}(\eps^{10}),\\
\partial_{\phi}Z_{6,3}(\phi,\eps^{2}K,\dots)&=\mathcal{O}(\eps^{12}),\quad  \partial_{K}Z_{6,3}(\phi,\eps^{2}K,\dots)&=\mathcal{O}(\eps^{12}),\\
\quad \partial_{\phi}R_{10}(\phi,\eps^{2}K,\dots)&=\mathcal{O}(\eps^{10}),\quad \partial_{K}R_{10}(\phi,\eps^{2}K,\dots)&=\mathcal{O}(\eps^{10}).
\end{eqnarray*}
On the other hand, by construction  $Q$ reads $P_1\Delta K_1+ P_2 \Delta K_2+P_{1/2}\Delta_{1/2}$ where $P_1$, $P_2$ and $P_{1/2}$ are polynomials of order 2 in $K_0$, $K_1$, $K_2$, $K_{1/2}$ and $\eps^2$ while $\Delta K_j$ denotes the variation of $K_j$: $\Delta K_j= K_j-K_j(0)$. Using again Lemma \ref{lem.dyn}, we check that for $0\leq \tau\leq \eps^{-2}$
\begin{equation*}
 \quad \partial_{\phi}Q=\mathcal{O}(\eps^{8}),\quad \partial_{K}Q=\mathcal{O}(\eps^{8}),
\end{equation*}
hence the result.
\end{proof}

Now we choose some precise initial conditions for $(\phi,K)$. We take $\phi(0)=0$ and $\kappa_{\star}<K(0)<1-\kappa_{\star}$ as in Theorem \ref{thm1} or $K(0)=\gamma\ll1$ as in Theorem \ref{thm2}. We also  consider the solution $(\phi_{\star},K_{\star})$ to \eqref{syst*} with initial condition  $(\phi_{\star},K_{\star})(0)= (\phi,K)(0)$. Then 
\begin{lemm}\label{lem.compa}
For all $0\leq \tau \leq \eps^{-2}$ we have
\begin{equation*} 
 ({\phi},{K})(\tau)=(\phi_{\star},K_{\star})(\tau)+ \mathcal{O}(\eps^2)\tau,
\end{equation*}
\end{lemm}

\begin{proof} Consider the system \eqref{syst.prop}, and apply  the change of variable $(L,\alpha)=\Phi(K,\phi)$ defined in Lemma \ref{Arnold}. Using \eqref{syst.prop} and the fact that $d\Phi$ is bounded (cf. Lemma \ref{Arnold}), we obtain that for $0\leq \tau\leq \eps^{-2}$
\begin{eqnarray*} 
\frac{d}{d\tau}(L,\alpha)&=&\frac{d}{d\tau}\Phi(K,\phi)= d\Phi(K,\phi). (\dot K,\dot \phi)\\&=&d\Phi (K,\phi). (\frac{\partial H_{\star}}{\partial \phi},-\frac{\partial H_{\star}}{\partial K})+\mathcal{O}(\eps^{2})\\ &=&(\frac{\partial H_{\star}}{\partial \alpha},-\frac{\partial H_{\star}}{\partial L})+\mathcal{O}(\eps^{2})\\
&=&(0,-\frac{\partial H_{\star}}{\partial L})+\mathcal{O}(\eps^{2}).
\end{eqnarray*}
Therefore there exists $L_{\star}\in \R$ so that $L(\tau)=L_{\star}+\mathcal{O}(\eps^{2})\tau$ and if we define  
 $\omega_{\star}=-\frac{\partial H_{\star}}{\partial L}(L_{\star}) $, we obtain  $\dis \alpha(\tau)=\omega_{\star}\tau+ \mathcal{O}(\eps^{2})\tau$. Next, as $d\Phi^{-1}$ is bounded, we get 
\begin{eqnarray*} 
 ({\phi},{K})(\tau)=\Phi^{-1}\big(L(\tau),\alpha(\tau)\big) &=&\Phi^{-1}\big(L_{\star},\omega_{\star}\tau\big)+ \mathcal{O}(\eps^{2})\tau\nonumber \\
 &=& (\phi_{\star},K_{\star})(\tau)+ \mathcal{O}(\eps^{2})\tau,
\end{eqnarray*}
where $(\phi_{\star},K_{\star})(\tau)$ is the solution of \eqref{syst*} so that  $(\phi_{\star},K_{\star})(0) =(\phi,K)(0)$.
\end{proof}

\begin{proof}[Proof of Theorems  \ref{thm1} and \ref{thm2}]
As a consequence of Lemma \ref{lem.compa}, the solution of \eqref{Bourg} satisfies for $0\leq t\leq \eps^{-6}$
\begin{eqnarray*}
K_{0}(t)&=&\eps^{2}K_{\star}(\eps^{4}t)+ \mathcal{O}(\eps^{8})t  \\
\phi_{0}(t)&=&\phi_{\star}(\eps^{4}t)+ \mathcal{O}(\eps^{6})t.
\end{eqnarray*}
This completes the proof of the main results : The error term $q_{1}$ comes from the normal form reduction (see Proposition \ref{NF}), and the error term $q_{2}$ comes from the $\mathcal{O}(\eps^{6})$ above (recall that $\nu=\eps^{4}$). 
\end{proof}

\appendix

\section{}
We prove Lemma \ref{FG}:\\
The first assertion is trivial. Concerning the second one we have
    \begin{multline*} 
 ||X_P(\xi,\eta)||_\rho=\sum_{k\in \Z} e^{\rho|k|}\left( \left|\frac{\partial Q}{\partial \xi_k}\right| +\left|\frac{\partial Q}{\partial \eta_k}\right|\right)\\
 \begin{aligned}
   &\leq p[P]\sum_{k\in \Z} \e^{\rho|k|}\sum_{\substack{j_1,\cdots,j_{p-1},\ell_1,\cdots,\ell_{p}\in \Z\\ \mathcal M(j_1,\cdots,j_{p-1},k;\ell_1,\dots,\ell_{p})=0}} \left|\xi_{j_1}\cdots\xi_{j_{p-1}}\eta_{\ell_1}\cdots\eta_{\ell_p}\right|+\left|\xi_{\ell_1}\cdots\xi_{\ell_{p}}\eta_{j_1}\cdots\eta_{j_{p-1}}\right|\\
    &\leq  p[P]  \sum_{j_1,\cdots,j_{p-1},\ell_1,\cdots,\ell_{p}\in \Z} \left|\xi_{j_1}e^{\rho|j_1|}\cdots\xi_{j_{p-1}}\e^{\rho|j_{p-1}|}\eta_{\ell_1}\e^{\rho|\ell_1|}\cdots\eta_{\ell_p}\e^{\rho|\ell_p|}\right|+\\
     &\qquad+ p[P]  \sum_{j_1,\cdots,j_{p-1},\ell_1,\cdots,\ell_{p}\in \Z} 
     \left|\xi_{\ell_1}\e^{\rho|\ell_1|}\cdots\xi_{\ell_p}\e^{\rho|\ell_p|}\eta_{j_1}\e^{\rho|j_1|}\cdots\eta_{j_{p-1}}\e^{\rho|j_{p-1}|}\right|\\
    &\leq 2p[P]||(\xi,\eta)||_\rho^{2p-1},
  \end{aligned}
\end{multline*}
where we used, $$\mathcal M(j_1,\cdots,j_{p-1},k;\ell_1,\dots,\ell_{p})=0 \Rightarrow |k|\leq |j_1|+\cdots+|j_{p-1}|+|\ell_1|+\cdots+|\ell_{p}|.$$
Assume now that $P\in\Pc_p$ and $Q\in\Pc_q$  with coefficients $a_{j\ell}$ and $b_{j\ell}$. It is clear that $\{P,Q\}$ is a monomial of degree $2p+2q - 2$ satisfying the zero momentum condition. Furthermore writing
$$
\{P,Q\}(\xi,\eta) = \sum_{ (j,\ell)\in\Z^{2p+2q-2} } c_{j\ell} \xi_{j_1}\cdots\xi_{j_{p+q-1}}\eta_{\ell_1}\cdots\eta_{\ell_{p+q-1}},
$$
where $c_{j\ell}$ is expressed  as a sum of coefficients $a_{ik}b_{nm}$ for which there exists $s\in \Z$ such that
$$i \cup n\ \setminus\{s\}=j \text{ and } k \cup m\ \setminus\{s\}=\ell.$$
For instance if $s=i_1=m_1$ then necessarily $j=(i_2,\cdots,i_{p},n_1,\cdots,n_q)$ and $\ell=k_1,\dots,k_p,m_2,\cdots,m_{q}$. Thus for fixed $(j,\ell)$, you just have to choose which of the indices $i$ you excise and which of indices $m$ you excise or, symmetrically, which of the indices $n$ you excise and which of indices $k$ you excise. Note that the value of $s$ is automatically fixed by the zero momentum condition on $(i,k)$ and on $(n,m)$. So $$|c_{j\ell}|\leq 2pq[P][Q].$$

\section{} 
We give here a method to compute the terms which appear in $Z_{6,2}$ (see Proposition \ref{prop.Z6}).  Let $\mathcal{A}$ be a resonant set.\\
Let  $(j_{1},j_{2},j_{3},\ell_{1},p_{1},p_{2})\in \mathcal{R}$. Assume that $j_{1},j_{2},j_{3},\ell_{1}\in \mathcal{A}$. Then by Lemma \ref{lem.couple}, we deduce that  $p_{1},p_{2}\in \mathcal{A}$. As a consequence, the only terms which will give a nontrivial contribution to $Z_{6,2}$ are of the form $(j_{1},j_{2},p_{1},j_{3},j_{4},p_{2})\in \mathcal{R}$, with $j_{1},j_{2},j_{3},j_{4}\in \mathcal{A}$ and $p_{1},p_{2}\not \in \mathcal{A}$.        \\

 Let $j_{1},j_{2},\ell_{1},\ell_{2}\in \mathcal{A}$ and $p_{1},p_{2}\in \mathbb{N}$ so that 
  \begin{equation}\label{eq.double*}
\left\{
\begin{aligned}
& p_{2}-p_{1}=j_{1}+j_{2}-\ell_{1}-\ell_{2}   , \\   
& p^{2}_{2}-p^{2}_{1}=j^{2}_{1}+j^{2}_{2}-\ell^{2}_{1}-\ell^{2}_{2}  . 
\end{aligned}
\right.
\end{equation}
 By Lemma \ref{lem.res},   there exist $k\in \Z^{*}$ and $n\in \N$ so that 
$\dis \mathcal{A}=\big\{n,n+3k,n+4k,n+k\big\}$. 
Hence, there exist $n,k\in \Z$ and $(m_{s})_{1\leq j\leq 4}$ with $m_{s}\in \big\{0,1,3,4\big\}$ so that $j_{s}=n+m_{s}k$ and $\ell_{1}=n+m_{3}k$, $\ell_{2}=n+m_{4}k$. We then define $q_{1},q_{2}\in \Q$ by  $p_{1}=n+q_{1}k$ and $p_{2}=n+q_{2}k$. We plug these expressions in \eqref{eq.double*} which gives
  \begin{equation*}  
\left\{
\begin{aligned}
& q_{2}-q_{1}=m_{1}+m_{2}-m_{3}-m_{4}:= U  , \\   
& q^{2}_{2}-q^{2}_{1}=m^{2}_{1}+m^{2}_{2}-m^{2}_{3}-m^{2}_{4} : =V.
\end{aligned}
\right.
\end{equation*}
When $U\neq 0$, we can solve this latter equation and we obtain
\begin{equation*}
q_{2}=\frac12(\frac{V}U+U),\quad q_{1}=\frac12(\frac{V}U-U).
\end{equation*}
By symmetry, we can assume that $m_{1}\geq m_{2}$, $m_{3}\geq m_{4}$. We also observe that $(m_{1},m_{2},p_{1},m_{3},m_{4},p_{2})$ is a solution iff  $(m_{3},m_{4},p_{2},m_{1},m_{2},p_{1})$ is a solution.\\

\begin{tabular}{|c|c|c|c|c|}
  \hline
Values of $m_{s}$ & Value of $V$  &  Values of $U$ & Value of $q_{2}$& Value of $q_{1}$   \\
  \hline
  4400 & 32 & 8 & 6& -2\\
  4401 & 31 & 7 & 40/7& -9/7\\
  4411 & 30 & 6 & 11/2& -1/2\\
  4431 & 22 & 4 &19/4& 3/4\\
  4430 & 23 & 5 & 24/5& -1/5\\
  4433 & 14& 2 & 9/2& 5/2\\
  4300 & 25 & 7 & 37/7& -12/7\\
  4301 & 24 & 6 & 5& -1\\
  4311 & 23 & 5 & 24/5& -1/5\\
  4100 & 17& 5 & 21/5& -4/5\\
  4103 & 8 & 2 & 3& 1\\
  4133 & -1 & -1 & 0& 1\\
  4011 & 14& 2 & 9/2& 5/2\\
  4031 & 6 & 0 &$\times $& $\times $\\
  4033& -2 & -2 & -1/2& 3/2\\
  3300 & 18 & 6 & 9/2& -3/2\\
  3301 & 17 & 5 & 21/5& -4/5\\
  3311 & 16& 4 & 4& 0\\
  3100 & 10 & 4 & 13/4& -3/4\\
  3011 & 7& 1 & 3& 4\\
  1100 & 2 & 2 & 3/2& -1/2\\
    \hline
\end{tabular}

\begin{exem}\label{exemple} Assume that  $\mathcal{A}=\{-2,1,2,-1\}$. Then $n=-2$ and $k=1$, so that $p_{1}=-2+q_{1}$ and $p_{2}=-2+q_{2}$. We only look at the integer values in the two last columns, and we find (up to permutation)
\begin{equation*}
4400:\;(2,2,-4,2,2,4),\quad 4301:\; (2,1,-3,-2,-1,3).
\end{equation*}
 Assume that  $\mathcal{A}=\{-1,5,7,1\}$. Then $n=-1$ and $k=2$, so that $p_{1}=-1+2q_{1}$ and $p_{2}=-1+2q_{2}$. In this case, we look at the half-integer values in the two last columns, and we find (up to permutation)
\begin{eqnarray*}
4400:\;(7,7,-5,-1,-1,11),& 4411:\;(7,7,-2,1,1,10), \\
4433:\; (7,7,4,5,5,8),&4301:\;(7,5,-3,-1,1,9),\\
4011:\; (7,-1,4,1,1,8),&4033:\;(7,-1,2,5,5,-2),\\
3300:\; (5,5,-4,-1,-1,8),&1100:\;(1,1,-2,-1,-1,2).
\end{eqnarray*}
\end{exem}




\end{document}